\documentclass[11pt]{article}
\usepackage[utf8]{inputenc}
\usepackage[a4paper, margin=1in]{geometry}
\usepackage{amsmath, amssymb, amsthm, mathtools}
\usepackage{graphicx, caption, subcaption}
\usepackage{enumitem}
\usepackage{hyperref}
\usepackage{color}
\usepackage{float}
\usepackage{xcolor}
\usepackage{times}
\usepackage[numbers,sort&compress]{natbib}

\usepackage{setspace}
\usepackage{changepage}

\renewenvironment{abstract}{%
	\begin{center}%
		\normalfont\fontsize{10pt}{12pt}\selectfont
		\begin{adjustwidth}{1.5pc}{1.5pc} 
		}{%
		\end{adjustwidth}%
	\end{center}%
}

\usepackage{comment}
\title{\textbf{Analysis of Chaos and Bifurcation in Nonlinear two-delay differential equation}}
\author{  Pragati Dutta*, Sachin Bhalekar\\
	\textit{School of Mathematics and Statistics, University of Hyderabad,India}\\
	*corresponding author email id: pragati.dutta2617@gmail.com
	\\sachinbhalekar@uohyd.ac.in}
\date{}

\newtheorem{theorem}{Theorem}[section]

\theoremstyle{remark}
\newtheorem{remark}[theorem]{Remark}

\begin{document}
	\maketitle

	\begin{abstract}
		This paper studies how complicated and irregular behavior, known as chaos, can arise in a simple mathematical model that includes time delays. The model is a delay differential equation in which the present rate of change depends not only on the current state but also on past states at two different delay times. The system is described by
		\begin{equation}
			\dot{x}(t)
			= -\gamma x(t)
			+ g\big(x(t - \tau_1)\big)
			- e^{-\gamma \tau_2}\, g\big(x(t - \tau_1 - \tau_2)\big),
		\end{equation}
		where  $g(x)=k \sin{x}, \; k\in\mathbf{R}$.
		Here, the delays $\tau_1$ and $\tau_2$ represent memory effects in the system, while the sine terms introduce strong nonlinearity. Numerical simulations are used to study the system behavior for different parameter values. Chaotic motion is identified using Lyapunov exponents, Poincar\'e map, power sectrum analysis and phase portraits, which show irregular and unpredictable dynamics. For certain parameter ranges, the system exhibits multi-scroll chaotic attractors, in which the motion alternates among several complex patterns. The switching analysis and the parameter sensitivity analysis are provided to justify the rich dynamics. Finally, chaos is controlled by adding a simple linear feedback term, which suppresses irregular oscillations and stabilizes the system. In addition, synchronization between master and slave systems is investigated using linear state feedback control. The a delay-independent as well as a delay-dependent conditions for synchronization are derived and verified numerically. The results show that even complex delayed systems can be effectively controlled and synchronized using simple feedback techniques.
	\end{abstract}

\section{Introduction}
\label{Intro}
Delay differential equations (DDEs) arise naturally in the mathematical modeling of dynamical systems in which the evolution of the system depends not only on its current state but also on past states \cite{hale1993introduction,diekmann2012delay,lakshmanan2011dynamics}. Such time delays are inherent in many real-world processes due to finite signal transmission times, feedback mechanisms, and processing lags. Compared to ordinary differential equations (ODEs), DDEs provide a more realistic description of many systems and have been successfully applied in engineering, population dynamics, biological systems, neural networks, laser physics, economics and control theory \cite{mackey1977oscillation,erneux2009applied,gopalsamy2013stability,lin2025hopf,krawiec2026lambert}.

Chaos, characterized by aperiodic oscillations and sensitive dependence on initial conditions, has been extensively studied in continuous-time ODEs and discrete dynamical systems \cite{strogatz2001nonlinear}. Classical results show that continuous-time ODEs require a minimum dimension of three to exhibit chaotic behavior. In contrast, DDEs possess an infinite-dimensional phase space due to the presence of time delays, which allows chaotic dynamics to arise even in scalar delay systems \cite{mackey1977oscillation,farmer1982chaotic,bhalekar2012dynamical,berezowski2026influence}. This fundamental difference has motivated substantial research on delay-induced chaos, including bifurcation analysis, multistability, and the emergence of scroll and multi-scroll attractors in systems with single or multiple delays \cite{mackey1977oscillation,farmer1982chaotic,lu2006generating,lakshmanan2011dynamics,bhalekar2025analysis}.

Considerable attention has also been devoted to the problems of chaos control and synchronization in delay differential equations. Various feedback-based control and synchronization strategies, such as delayed feedback control, active control, and linear state feedback, have been proposed to suppress chaotic oscillations and achieve convergence between coupled systems \cite{pyragas1992continuous,babloyantz1995control,mensour1998chaos,niculescu2002delay,michiels2007stability,bhalekar2011new,bhalekar2010synchronization}. Stability and synchronization issues in time-delay and chaotic systems have been investigated using both analytical and numerical approaches\cite{mensour1998synchronization,pyragas1998synchronization,michiels2009synchronization}.

{
Published two-delay nonlinear systems have largely been studied from the viewpoint of local stability and bifurcation. For example, Bélair and Campbell \cite{belair1994stability} considered a single-action mechanism with multiple negative delayed-feedback loops and used characteristic-root and centre-manifold analyses to determine Hopf and higher-codimension bifurcations. Li, Ruan, and Wei \cite{li1999stability} treated a general differential--difference equation in which two delays occur as independent arguments, deriving delay-dependent stability criteria and using normal-form theory to determine the direction and stability of Hopf-bifurcating periodic solutions. These works provide a detailed local description near an equilibrium, but their principal objective is the onset and local character of oscillations rather than the subsequent large-amplitude chaotic regimes, their feedback suppression, or master--slave synchronization.

The closest antecedent is the work of Boullu, Pujo-Menjouet, and Bélair \cite{boullu2020stability}, because it studies the same structural equation as~\eqref{eq:main} in a platelet-production model. In contrast to two independently weighted feedback channels, this equation contains linked delayed arguments, $\tau_1$ and $\tau_1+\tau_2$, while the coefficient of the latter is itself delay dependent through $e^{-\gamma\tau_2}$. Boullu et al. used D-decomposition and centre-manifold reductions to obtain stability regions and to analyze steady-state, Hopf, and double bifurcations, including the occurrence of a torus. Their analysis therefore gives a rigorous account of local bifurcation mechanisms for a general feedback function $g$, but it does not characterize the sinusoidally generated single- and double-scroll chaotic regimes, determine feedback gains for chaos suppression, or formulate a synchronization condition for coupled copies of the equation.

This identifies the research gap addressed here. The present study does not propose a new general theory of two-delay Hopf bifurcation; rather, it complements the above local theories by examining the post-bifurcation dynamics and control properties of the structurally coupled-delay model when $g(x)=k\sin x$. The bounded, nonmonotone sinusoidal feedback permits repeated transitions between periodic and chaotic motion and supports scroll-type attractors, while its global Lipschitz bound can be exploited analytically in the synchronization problem. Thus, the mathematical question is how the specific coupling of the two delays and the coefficient $e^{-\gamma\tau_2}$ affects nonlinear dynamics beyond the first loss of equilibrium stability, and whether this structure admits explicit control and synchronization criteria.

More precisely, the contributions of this work are as follows. First, for the sinusoidal feedback $g(x)=k\sin x$, we systematically investigate the dynamics generated by variation of the time delay $\tau_2$, including repeated periodic--chaotic transitions, embedded periodic windows, and transitions from single-scroll to double-scroll attractors. The observed dynamical regimes are characterized using bifurcation diagrams, time series, phase portraits, largest Lyapunov exponents, Poincar\'e return maps, and power spectral densities, and the robustness of these phenomena is further examined for alternative parameter sets. Second, we investigate the sensitivity of the dynamics to the model parameters $k$, $\gamma$, and $\tau_1$, thereby identifying their effects on the occurrence and extent of periodic and chaotic regimes. Third, for the controlled system, a linear state-feedback controller is introduced and the imaginary-root conditions of the controlled characteristic equation are used to determine the critical feedback gain $\mu_*$ associated with the change of stability of the trivial equilibrium. The effectiveness of the controller in suppressing chaotic oscillations is then verified numerically for representative parameter sets. Fourth, the master--slave synchronization problem is formulated through the corresponding error system. Using the global bound $|\sin u-\sin v|\leq |u-v|$, the error dynamics are reduced to a scalar comparison delay differential equation, from which an explicit delay-independent sufficient condition
\[
\gamma+\delta>2|k|
\]
is obtained, guaranteeing asymptotic synchronization for all $\tau_1,\tau_2\geq0$. In addition, a less conservative synchronization criterion is established for fixed $\tau_2$, providing a sharper sufficient condition when the time delay is prescribed. Finally, the model formulation is discussed from the perspective of its potential applications and scope, highlighting how the delay-dependent feedback architecture can serve as a framework for studying stability, bifurcation, chaos, control, and synchronization in nonlinear systems with multiple delays. Thus, the paper combines analytical stability and synchronization results with numerical characterization, parameter sensitivity, and control of complex dynamics in a two-delay feedback system.

\textit{The remainder of this paper is organized as follows. Section~\ref{sec:main} introduces the two-delay model, explains the roles of the delay-dependent feedback terms, motivates the choice $g(x)=k\sin x$, and discusses the scope and possible applications of the resulting system. Section~\ref{sec:stability} derives the linearized equation and characteristic equation at the trivial equilibrium and establishes delay-independent stability and instability conditions. Section~\ref{sec:switching} treats $\tau_2$ as the switching parameter, derives the critical-delay and root-crossing conditions, and illustrates a stable--unstable--stable transition. Section~\ref{sec:chaos} examines the post-bifurcation dynamics through time series, phase portraits, bifurcation diagrams, largest Lyapunov exponents, Poincar\'e return maps, and power spectral densities, including tests with alternative parameter sets. Section~\ref{sec:sensitivity} studies how variations in $k$, $\gamma$, and $\tau_1$ alter periodic windows and chaotic regimes. Section~\ref{sec:control} develops a linear state-feedback controller, computes critical gains from the controlled characteristic equation, and verifies chaos suppression for two parameter sets. Section~\ref{sec:sync} formulates the master--slave error system, proves both a delay-independent synchronization condition and a less conservative fixed-$\tau_2$ criterion, and validates the results numerically. Section~\ref{sec:conc} summarizes the principal findings and their limitations.}

}

\section{Mathematical Model}
\label{sec:main}
We consider the nonlinear delay differential equation
\begin{equation}
	\dot{x}(t)
	= -\gamma x(t)
	+ g\big(x(t - \tau_1)\big)
	- e^{-\gamma \tau_2}\, g\big(x(t - \tau_1 - \tau_2)\big),
	\label{eq:main}
\end{equation}
This model was originally introduced to describe regulatory feedback mechanisms in biological systems, such as platelet production.
The first delay $\tau_1$ represents the time required for production or maturation to occur.
The second delay $\tau_2$ represents the duration for which individuals remain active in the system before being removed. 
The exponential term $e^{-\gamma\tau_2}$ accounts for natural decay or loss that occurs during this time interval.

Let $x_*$ denote an equilibrium point of system~\eqref{eq:main}. Then $x_*$ satisfies
\[
-\gamma x_* + g(x_*) - e^{-\gamma \tau_2} g(x_*) = 0.
\]
Throughout this work, we assume that the nonlinear function $g$ satisfies $g(0)=0$, which guarantees that $x_*=0$ is always an equilibrium point of system~\eqref{eq:main}. Furthermore, assuming $g'(0)=k$, we linearize the system about the trivial equilibrium $x_*=0$.

{
\subsection*{Motivation and Importance of Sinusoidal feedback}
In this work, we choose
\[
g(x)=k\sin x,
\]
where $k$ is a real feedback gain. This choice is motivated primarily by its mathematical structure and by its established role as a canonical delayed-feedback nonlinearity. Sinusoidal feedback occurs naturally when the state variable represents a phase or a phase-dependent signal, as in optical ring cavities and related optoelectronic feedback systems; scalar DDEs with sinusoidal feedback are known to generate period-doubling, chaotic motion, and high-dimensional dynamics even with a single delay \cite{sprott2007simple}. In the present two-delay equation, the sine law therefore provides a minimal smooth mechanism with which to isolate the effect of the coupled delays without introducing additional shape parameters into $g$.

The function also has several analytical advantages. It is infinitely differentiable, odd, and bounded, with
\[
|g(x)|\leq |k|, \qquad |g'(x)|=|k\cos x|\leq |k|.
\]
Consequently, it supplies the prescribed local gain $g'(0)=k$ while remaining globally Lipschitz. The local gain permits a direct characteristic-equation analysis at the trivial equilibrium, whereas the global estimate
\[
|g(u)-g(v)|\leq |k|\,|u-v|
\]
is the key inequality used later to derive a delay-independent synchronization condition. Moreover, the repeated extrema and sign changes of the sine function create multiple nonlinear feedback branches. This feature is particularly suitable for studying recurrent periodic windows and transitions between single-scroll and double-scroll attractors, while boundedness prevents the delayed forcing itself from growing without limit.

Compared with commonly used alternatives, the choice involves a deliberate trade-off. Hill- and Mackey--Glass-type rational functions have a clearer interpretation for positive concentrations or cell populations and are therefore preferable for quantitative biological calibration, but their monotone or single-humped profiles do not provide the repeated feedback branches of a periodic nonlinearity. Sigmoidal functions such as $\tanh x$ are also smooth, bounded, and globally Lipschitz, but they are monotone and possess only two saturation plateaus. Polynomial feedback is smooth and convenient for local normal-form calculations, but it is generally unbounded and does not provide a global Lipschitz constant. Piecewise-linear feedback can generate multi-scroll dynamics efficiently, but its corners complicate derivative-based bifurcation analysis. The sine function combines smoothness, boundedness, an explicit global derivative bound, and repeated nonlinear branches in a single-parameter law.

We emphasize that $k\sin x$ is not asserted to be a quantitatively calibrated platelet-production function. If system~\eqref{eq:main} is interpreted biologically, $x$ should be regarded as a normalized deviation from a homeostatic level and the sine law as a phenomenological feedback approximation over the relevant operating range; a nonnegative Hill-type response would be more appropriate when the state denotes an absolute cell count over a wide range. The present choice is thus intended to investigate the mathematical capacity of the two-delay architecture to produce and regulate complex dynamics.
}

Substituting this choice into~\eqref{eq:main}, we obtain
\begin{equation}
	\dot{x}(t)
	= -\gamma x(t)
	+ k\sin\big(x(t - \tau_1)\big)
	- k e^{-\gamma \tau_2}\sin\big(x(t - \tau_1 - \tau_2)\big).
	\label{eq:sinx}
\end{equation}

{
\subsection*{Potential Applications and Scope of the Model}

The usefulness of~\eqref{eq:sinx} is not restricted to one physical setting. Its basic structure combines instantaneous dissipation, two linked memory times, attenuation over the second delay interval, and bounded nonlinear feedback. These ingredients occur in several classes of systems, although the interpretation of $x$, the delays, and the feedback function must be adapted to the application.

\textbf{Biological production and maturation systems.}
The most direct interpretation is the platelet-production model introduced in~\cite{boullu2020stability}, from which~\eqref{eq:main} originates. In that setting, $x(t)$ represents a mature population, $\tau_1$ is the production or maturation time, and $\tau_2$ is the time during which mature individuals remain in the circulating population. The factor $e^{-\gamma\tau_2}$ is the survival fraction over this interval. The stability and bifurcation results can therefore help identify combinations of maturation time, survival time, and loss rate associated with steady regulation or sustained oscillations. Similar linked-delay mechanisms arise in other age-structured cell-production models, including erythropoiesis \cite{belair1995age}. Nevertheless, the sinusoidal feedback used here is a qualitative prototype rather than a calibrated hematological response. Application to clinical platelet data would require a nonnegative feedback law, parameter estimation, positivity analysis, and validation against patient measurements. The feedback controller should likewise be interpreted as a mathematical regulation mechanism, not as a treatment prescription.

\textbf{Engineering and optoelectronic feedback systems.}
In an engineering interpretation, $x(t)$ may denote a normalized voltage, phase, optical intensity deviation, or process variable; $\gamma$ represents dissipation; and the delays represent sensing, transmission, processing, transport, or recirculation times. Sinusoidal nonlinearities arise in phase-sensitive components and interferometric feedback loops. Experiments on optoelectronic delayed-feedback devices have demonstrated periodic and high-dimensional chaotic oscillations produced by a nonlinear element and a feedback loop \cite{blakely2004highspeed}. A two-delay extension can represent two propagation paths or the entry and exit of a signal from a recirculating stage. In this context, the bifurcation diagrams indicate operating regions in which a change in delay can switch the device between steady, periodic, and chaotic outputs. The linear controller $-\mu x(t)$ provides a low-complexity damping mechanism when suppression of irregular oscillations is desired, while the chaotic regimes may instead be useful for broadband signal generation, testing, or ranging.

\textbf{Neural and networked systems.}
At a reduced-order level, $x(t)$ can represent a population-averaged neural activity or the phase of an oscillatory neural unit. The two delays may describe a local synaptic-processing time and a longer axonal or recurrent-loop delay. The bounded sine term is especially natural in phase-reduced oscillator descriptions, where the sign of the coupling changes with phase difference. Delayed neural networks are known to exhibit chaotic activity and synchronization, and feedback-control criteria have been developed for recurrent neural networks with time-varying delays \cite{cui2009synchronization}. The present scalar analysis may therefore serve as a building block for studying how two memory scales affect rhythmic activity, desynchronization, or entrainment. However, direct neuroscientific application would require a multidimensional network model incorporating connectivity, heterogeneous delays, noise, and physiologically meaningful activation functions.

\textbf{Chaos-based communication.}
The master--slave construction in Section~\ref{sec:sync} has a natural communication interpretation. The master system can act as a transmitter that generates a broadband chaotic carrier, while the controlled slave acts as a receiver. A low-amplitude message may be embedded by masking or parameter modulation; after synchronization, the receiver can reconstruct the chaotic carrier and separate the message. Synchronization of delay-differential systems has previously been proposed for private communication \cite{mensour1998synchronization}, and chaos-synchronized optical transmitters and receivers have been demonstrated over commercial fibre links \cite{argyris2005chaos}. The presence of two delays can increase waveform complexity and enlarge the parameter space available for system design. 

These examples show how the same mathematical results can have different roles: bifurcation analysis identifies parameter-sensitive transitions, chaos control determines how irregular dynamics may be suppressed, and synchronization provides conditions under which separated units can follow a common trajectory. Experimental or application-specific validation remains an important direction for future work.
}

\section{Stability Analysis}
\label{sec:stability}
The local stability of system~\eqref{eq:sinx} at $x_*=0$ is obtained by retaining the terms that are linear in a small perturbation. Write $x(t)=\varepsilon y(t)$, where $0<\varepsilon\ll1$. Taylor expansion of the delayed nonlinearities gives
\begin{align*}
	\sin\bigl(\varepsilon y(t-\tau_1)\bigr)
	&=\varepsilon y(t-\tau_1)+O(\varepsilon^3),\\
	\sin\bigl(\varepsilon y(t-\tau_1-\tau_2)\bigr)
	&=\varepsilon y(t-\tau_1-\tau_2)+O(\varepsilon^3).
\end{align*}
Substitution into~\eqref{eq:sinx}, division by $\varepsilon$, and passage to the limit $\varepsilon\to0$ yield the variational equation
\begin{equation}
	\dot y(t)
	=-\gamma y(t)
	+k y(t-\tau_1)
	-k e^{-\gamma\tau_2}y(t-\tau_1-\tau_2).
	\label{eq:linearized_origin}
\end{equation}
To derive the characteristic equation, take an exponential mode $y(t)=e^{\lambda t}$. Then
\[
	y(t-\tau_1)=e^{\lambda t}e^{-\lambda\tau_1},
	\qquad
	y(t-\tau_1-\tau_2)=e^{\lambda t}e^{-\lambda(\tau_1+\tau_2)}.
\]
Substituting these expressions into~\eqref{eq:linearized_origin} gives
\[
	\lambda e^{\lambda t}
	=-\gamma e^{\lambda t}
	+k e^{\lambda t}e^{-\lambda\tau_1}
	-k e^{-\gamma\tau_2}e^{\lambda t}e^{-\lambda(\tau_1+\tau_2)}.
\]
Since $e^{\lambda t}\neq0$, division by this common factor produces the characteristic equation 
\begin{equation}
	P(\lambda)
	=\lambda+\gamma
	-k e^{-\lambda\tau_1}
	+k e^{-\gamma\tau_2}e^{-\lambda(\tau_1+\tau_2)}
	=0.
	\label{eq:characteristic_origin}
\end{equation}
\textbf{Definition (Stability) \cite{lakshmanan2011dynamics}:} The equilibrium $x_* = 0$ is asymptotically stable if all roots $\lambda$ of the characteristic equation satisfy $\text{Re}(\lambda) < 0$.\\
In \cite{dutta2025some}, we proposed a stability result for the fractional order version of equation (\ref{eq:main}). It can be restated for the integer order case as below:
\begin{theorem}
	\label{prel}
	The equilibrium point $x_* = 0$ of system~\eqref{eq:sinx} exhibits the following stability properties (see Figure ~\ref{ind}):
	\begin{enumerate}[label=(\alph*)]
		\item \textbf{Delay-independent stability:}
		The equilibrium is asymptotically stable for all delay values
		$\tau_1 \ge 0$ and $\tau_2 \ge 0$ if either
		\begin{enumerate}[label=(\roman*)]
			\item $\gamma > 2k > 0$, or
			\item $\gamma > -2k > 0$.
		\end{enumerate}
		
		\item \textbf{Delay-independent instability:}
		The equilibrium is unstable for all delay values
		$\tau_1 \ge 0$ and $\tau_2 \ge 0$ whenever $\gamma < 0$, that is, when
		$(k,\gamma)$ lies in the third or fourth quadrant of the $k\gamma$-plane.
	\end{enumerate}
\end{theorem}
{
	\begin{proof}
The local stability of the equilibrium is completely determined by the roots of the characteristic equation \eqref{eq:characteristic_origin}.

\textbf{(a) Delay-independent stability.}

First consider $\tau_2=0,\;\tau_1\ge0$. Then the characteristic equation \eqref{eq:characteristic_origin} reduces to
\[
\lambda=-\gamma.
\]
Hence the equilibrium is asymptotically stable whenever $\gamma>0$.

Now let $\tau_2>0$. Suppose a change in stability occurs. Then there exists a purely imaginary characteristic root $\lambda=iv$, where $v>0$. Separating the real and imaginary parts of \eqref{eq:characteristic_origin} gives
\begin{align}
0&=-\gamma+k\cos(v\tau_1)-k e^{-\gamma\tau_2}\cos\bigl(v(\tau_1+\tau_2)\bigr), \label{eq:real_int}\\
0&=-v-k\sin(v\tau_1)+k e^{-\gamma\tau_2}\sin\bigl(v(\tau_1+\tau_2)\bigr). \label{eq:imag_int}
\end{align}

If $\gamma>2k>0$, then
\[
-\gamma+k\cos(v\tau_1)-k e^{-\gamma\tau_2}\cos(v(\tau_1+\tau_2))
\le -\gamma+k+k e^{-\gamma\tau_2}
<-\gamma+2k<0,
\]
which contradicts \eqref{eq:real_int}.

Similarly, if $\gamma>-2k>0$ ($k<0$),
\[
-\gamma+k\cos(v\tau_1)-k e^{-\gamma\tau_2}\cos(v(\tau_1+\tau_2))
\le -\gamma-k-k e^{-\gamma\tau_2}
<-\gamma-2k<0,
\]
again contradicting \eqref{eq:real_int}.

Therefore, no purely imaginary characteristic root exists for any $\tau_2>0$. Since the equilibrium is stable at $\tau_2=0$, it remains asymptotically stable for all $\tau_1,\tau_2\ge0$.

\medskip
\textbf{(b) Delay-independent instability.}

\textbf{Case (i):} Suppose $\gamma<0$ and $k<0$. At $\tau_2=0$, the characteristic equation becomes
\[
\lambda=-\gamma>0,
\]
so the equilibrium is unstable.

Now let $\tau_2>0$ and define
\[
P(\lambda)=\lambda+\gamma-k e^{-\lambda\tau_1}
+k e^{-\gamma\tau_2}e^{-\lambda(\tau_1+\tau_2)}.
\]
Since
\[
P(0)=\gamma-k+k e^{-\gamma\tau_2}
=\gamma+k\bigl(e^{-\gamma\tau_2}-1\bigr)<0,
\]
and
\[
\lim_{\lambda\to\infty}P(\lambda)=\infty,
\]
the Intermediate Value Theorem guarantees the existence of a positive real root of the characteristic equation. Hence the equilibrium is unstable for all $\tau_1,\tau_2\ge0$.

\textbf{Case (ii):} Suppose $\gamma<0$ and $k>0$. Then
\[
P(-\gamma)=0.
\]
Since $-\gamma>0$, the characteristic equation possesses a positive real root. Therefore, the equilibrium is unstable for all $\tau_1,\tau_2\ge0$ in the fourth quadrant as well.
\end{proof}
}
\begin{figure}[H]
	\centering
	\includegraphics[scale=0.58]{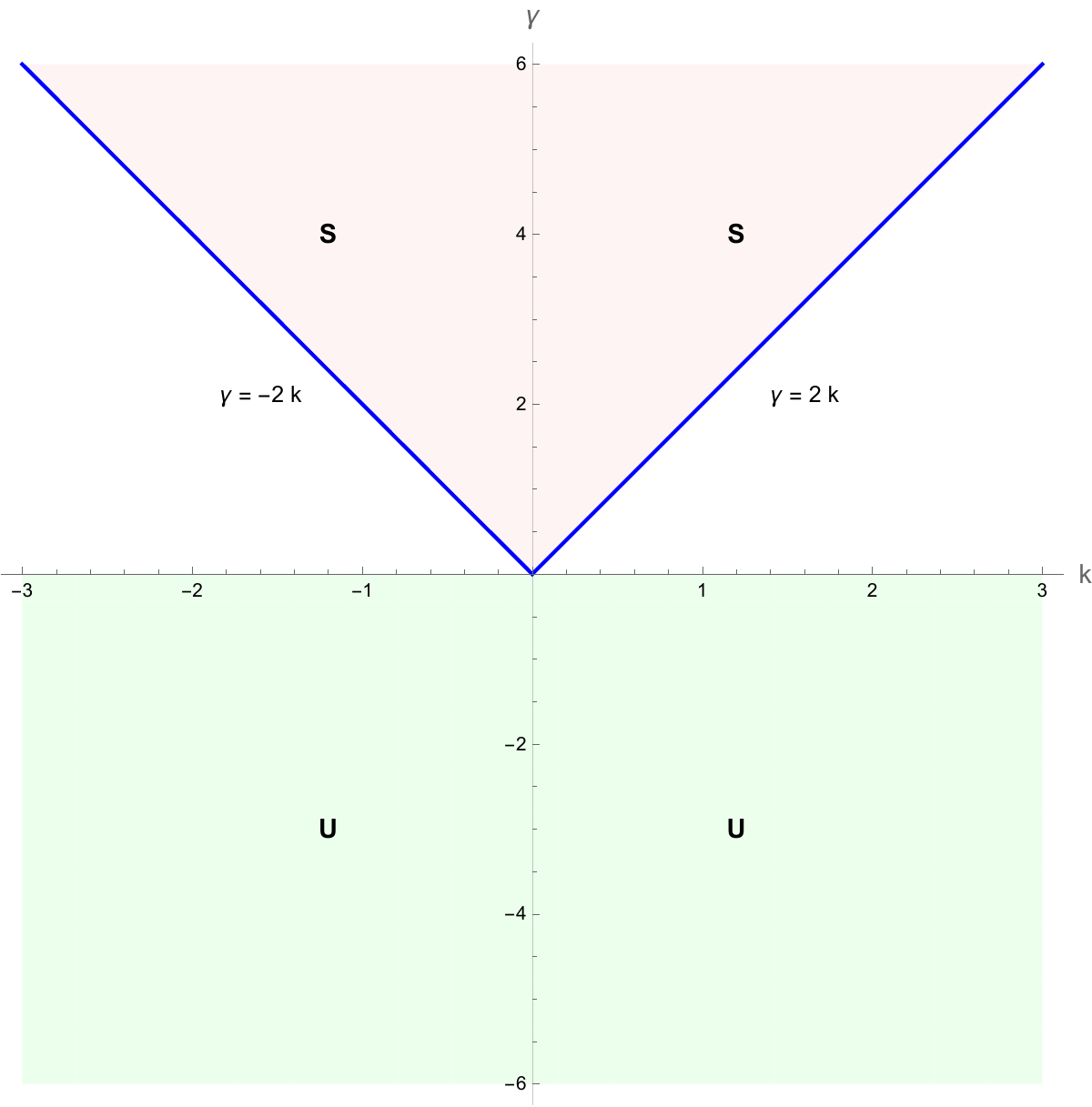}
	\caption{Delay Independent stability regions for equation (\ref{eq:main})}
	\label{ind}
\end{figure}
Since the instability is the necessary condition for chaos, we will set the parameter values outside the stable regions.

{

\section{Stability Switching Analysis}
\label{sec:switching}

In this section, we investigate the existence of stability switching of the trivial equilibrium with respect to the delay parameter $\tau_2$. The analysis is based on the characteristic root approach, which has been extensively used to study stability switching in delay differential equations \cite{grossman1982discrete,beretta2002geometric,bhalekar2016stability,bhalekar2024stability}.

Consider the linearized system~\eqref{eq:linearized_origin}
with $\gamma\neq0$, $\tau_1>0$ and $\tau_2>0$.

The corresponding characteristic equation is given by equation (\ref{eq:characteristic_origin}). To determine the critical values of $\tau_2$, we assume that (\ref{eq:characteristic_origin}) possesses a pair of purely imaginary roots
$
\lambda=iv,\;v>0.
$

Substituting $\lambda=iv$ into~\eqref{eq:characteristic_origin} and separating the real and imaginary parts, we obtain Equations ~\eqref{eq:real_int} and~\eqref{eq:imag_int}.

Squaring and adding ~\eqref{eq:real_int} and~\eqref{eq:imag_int}, we get
\begin{equation}
k^2e^{-2\gamma\tau_2}
=
k^2+v^2+\gamma^2
-2k\gamma\cos(\tau_1v)
+2kv\sin(\tau_1v)
\end{equation}

which gives
\begin{equation}
\tau_2(v)
=
\frac{1}{2\gamma}
\ln
\left(
\frac{k^2}
{k^2+v^2+\gamma^2
-2k\gamma\cos(\tau_1v)
+2kv\sin(\tau_1v)}
\right).
\label{tau2_formula}
\end{equation}

Substituting \eqref{tau2_formula} into the real part (\ref{eq:real_int}), we obtain an equation 

\begin{equation}
F(v)=0,
\label{Fv}
\end{equation}
where
\begin{align}
F(v)
=&\
\gamma-k\cos(\tau_1v)
\nonumber\\
&
+\frac{k}{|k|}
\sqrt{
k^2+v^2+\gamma^2
-2k\gamma\cos(\tau_1v)
+2kv\sin(\tau_1v)}
\nonumber\\
&\times
\cos\left[
\left(
-\tau_1
-
\frac{1}{2\gamma}
\ln
\left(
\frac{k^2}
{k^2+v^2+\gamma^2
-2k\gamma\cos(\tau_1v)
+2kv\sin(\tau_1v)}
\right)
\right)v
\right].
\end{align}

Therefore, for fixed values of $k$, $\gamma$ and $\tau_1$, the determination of all possible critical delays reduces to finding the positive roots of the nonlinear equation
\[
F(v)=0.
\]

Suppose
$
v_1<v_2<\cdots<v_m
$
are the positive solutions of \eqref{Fv}. For each $v_i$, the corresponding critical delay is obtained from \eqref{tau2_formula} as
\[
\tau_{2,i}=\tau_2(v_i).
\]

Only those values satisfying
$
\tau_{2,i}>0
$
are admissible.

To determine the direction in which the characteristic roots cross the
imaginary axis, we examine the sign of
$
\operatorname{Re}\!\left(\frac{d\lambda}{d\tau_2}\right)\big|_{u=0},
$
where $\lambda=u+iv$.

Assuming that $\lambda$ is a differentiable function of $\tau_2$ and
differentiating the characteristic equation i.e (\ref{eq:characteristic_origin}) equation with respect to $\tau_2$, we obtain
\begin{equation}
\frac{d\lambda}{d\tau_2}
+k\tau_1e^{-\tau_1\lambda(\tau_2)}
\frac{d\lambda}{d\tau_2}
+k e^{-\gamma\tau_2-(\tau_1+\tau_2)\lambda(\tau_2)}
\left[
-\gamma
-\lambda(\tau_2)
-(\tau_1+\tau_2)\frac{d\lambda}{d\tau_2}
\right]
=0.
\end{equation}

Hence,
\begin{equation}
\frac{d\lambda}{d\tau_2}
=
\frac{
k e^{\tau_1\lambda(\tau_2)}
\left(\gamma+\lambda(\tau_2)\right)
}{
e^{\tau_2\gamma+(2\tau_1+\tau_2)\lambda(\tau_2)}
-
k\tau_1 e^{\tau_1\lambda(\tau_2)}
+
k\tau_1 e^{\tau_2\gamma+(\tau_1+\tau_2)\lambda(\tau_2)}
-
k\tau_2 e^{\tau_1\lambda(\tau_2)}
}.
\label{der}
\end{equation}

Now, substituting $\lambda=u+iv$ into \eqref{der}, taking the real part,
and then setting $u=0$, we obtain

\begin{align}
Re\left.\left(\frac{d\lambda}{d\tau_2}\right)\right|_{u=0}
&=
\frac{k\,N(\gamma,k,v,\tau_1,\tau_2,v)}
     {D(\gamma,k,v,\tau_1,\tau_2,v)},
\end{align}
where
\begin{align}
N &=
-k\tau_1\gamma
-k\tau_2\gamma
+e^{\tau_2\gamma}k\tau_1\gamma\cos(\tau_2v)
+e^{\tau_2\gamma}\gamma\cos\bigl((\tau_1+\tau_2)v\bigr)
\nonumber\\
&\qquad
+e^{\tau_2\gamma}k\tau_1v\sin(\tau_2v)
+e^{\tau_2\gamma}v\sin\bigl((\tau_1+\tau_2)v\bigr),
\end{align}
and
\begin{align}
D &=
e^{2\tau_2\gamma}
+k^2\tau_1^2
+e^{2\tau_2\gamma}k^2\tau_1^2
+2k^2\tau_1\tau_2
+k^2\tau_2^2
\nonumber\\
&\qquad
+2e^{2\tau_2\gamma}k\tau_1\cos(\tau_1v)
-2e^{\tau_2\gamma}k^2\tau_1(\tau_1+\tau_2)\cos(\tau_2v)
\nonumber\\
&\qquad
-2e^{\tau_2\gamma}k\tau_1
\cos\bigl((\tau_1+\tau_2)v\bigr)
-2e^{\tau_2\gamma}k\tau_2
\cos\bigl((\tau_1+\tau_2)v\bigr).
\end{align}

Since stability can change only when a pair of characteristic roots crosses
the imaginary axis, we evaluate
$
Re\left.\left(\frac{d\lambda}{d\tau_2}\right)\right|_{u=0}
$
at the critical values
$
(v_i,\tau_{2,i})
$
for fixed values of $k$, $\gamma$, and $\tau_1$. The sign of
$
Re\left.\left(\frac{d\lambda}{d\tau_2}\right)\right|_{u=0}
$
determines the crossing direction of the characteristic roots.

If
$
Re\left.\left(\frac{d\lambda}{d\tau_2}\right)\right|_{u=0}>0,
$
then the characteristic roots cross the imaginary axis from the left half-plane
to the right half-plane as $\tau_2$ increases, and the trivial equilibrium
loses its stability.

If
$
Re\left.\left(\frac{d\lambda}{d\tau_2}\right)\right|_{u=0}<0,
$
then the characteristic roots cross from the right half-plane to the left
half-plane, and the trivial equilibrium regains its stability.

Therefore, stability switching occurs whenever the signs of
$
Re\left.\left(\frac{d\lambda}{d\tau_2}\right)\right|_{u=0}
$
at two successive critical delays are opposite, that is,
$
\left(Re\left.\left(\frac{d\lambda}{d\tau_2}\right)\right|_{u=0}\right)_{\tau_{2,i}}
\left(Re\left.\left(\frac{d\lambda}{d\tau_2}\right)\right|_{u=0}\right)_{\tau_{2,i+1}}
<0.
$

Hence, as $\tau_2$ increases, the characteristic roots cross the imaginary
axis alternately, leading to alternating intervals of stability and
instability of the trivial equilibrium. \\

Below, we present an example showing the stability switching of type S-U-S for the parameter values
$
k=-5,\;\gamma=4.3,\;\tau_1=1.
$
For $\tau_2=0$, the trivial equilibrium is stable (see Figure \ref{tau20}).
For the above parameter values, the variation of $\tau_2(v)$ given by
\eqref{tau2_formula} is shown in Figure~\ref{fig:tau2v}. It can be seen that
$\tau_2(v)$ is positive only over two small intervals of $v$. Therefore, only
the positive roots of $F(v)=0$ lying in these intervals give admissible
critical delays. Since no other positive roots satisfy $\tau_2(v)>0$, only
these values are considered in the stability switching analysis.

\begin{figure}[htbp]
\centering
\includegraphics[width=0.65\textwidth]{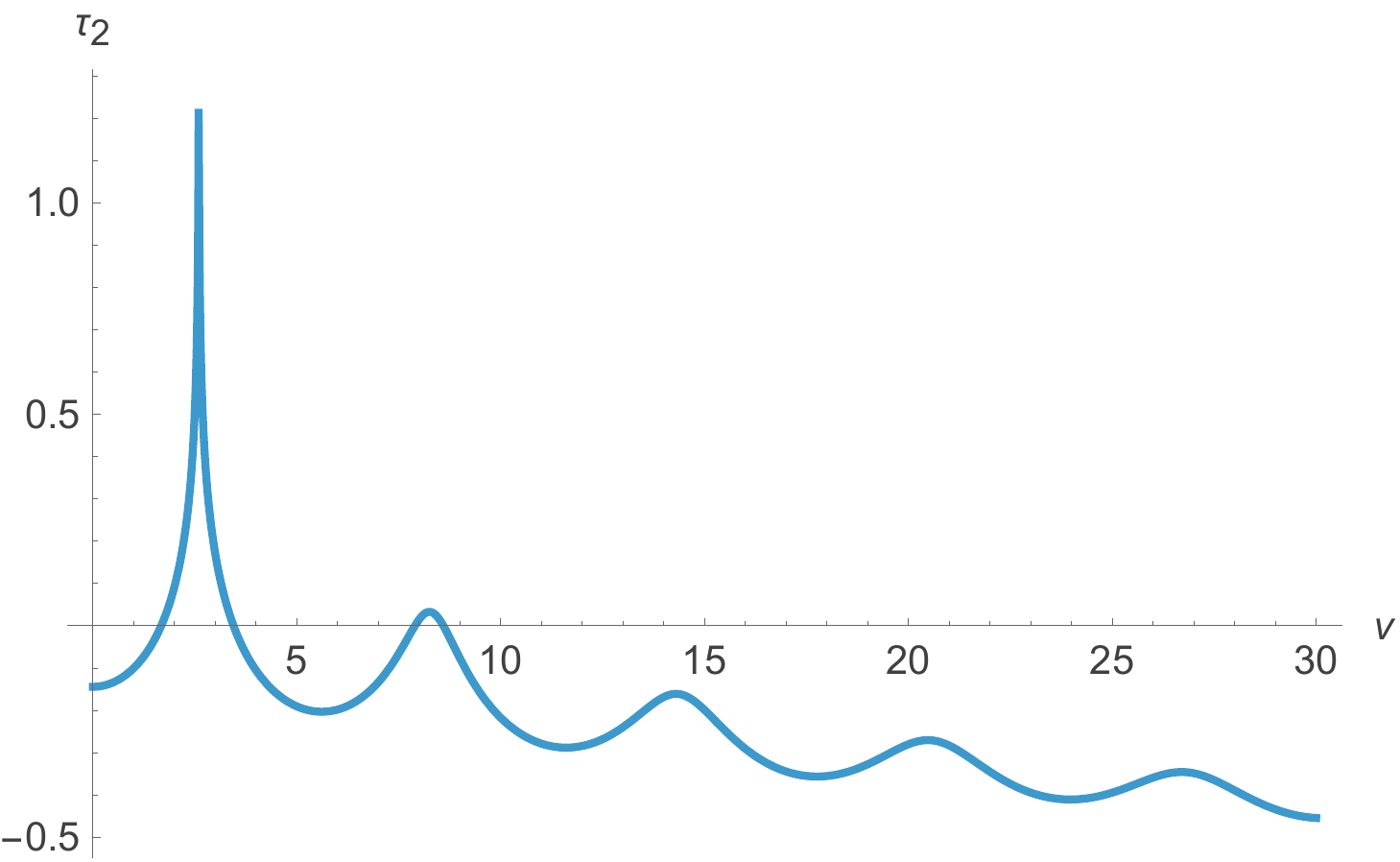}
\caption{Plot of $\tau_2(v)$}
\label{fig:tau2v}
\end{figure}
Substituting these parameter values into equation \eqref{Fv}, we obtain two positive solutions,
$
v_1=2.5977
$
and
$
v_2=2.6510
$, satisfying the criteria.
Using equation \eqref{tau2_formula}, the corresponding critical delay values are found to be
$
\tau_{2a}=0.64274
$
and
$
\tau_{2b}=1.2430.
$

To verify whether these critical delays change the stability of the equilibrium, we compute
$
Re\left.\left(\frac{d\lambda}{d\tau_2}\right)\right|_{u=0}
$
at the corresponding critical delays. We obtain
\[
\left(\operatorname{Re}\left.\left(\frac{d\lambda}{d\tau_2}\right)\right|_{u=0}\right)_{\tau_{2a}} = 0.11539
\quad\text{and}\quad
\left(\operatorname{Re}\left.\left(\frac{d\lambda}{d\tau_2}\right)\right|_{u=0}\right)_{\tau_{2b}} = -0.01728.
\]

Since
$
Re\left.\left(\frac{d\lambda}{d\tau_2}\right)\right|_{u=0}>0
$
at $\tau_{2a}$, the characteristic root crosses the imaginary axis from the left half-plane to the right half-plane, and the equilibrium loses its stability. Similarly,
$
Re\left.\left(\frac{d\lambda}{d\tau_2}\right)\right|_{u=0}<0
$
at $\tau_{2b}$, indicating that the root returns back to the left half-plane and the equilibrium regains its stability. Therefore, as $\tau_2$ increases, the system exhibits a stability switching of type S-U-S.
In Figure (\ref{fig:stability_switch}), we have given time series for different values of $\tau_2$ to support our claim. 
\begin{figure}[htbp]
\centering

\begin{subfigure}{0.32\textwidth}
    \centering
    \includegraphics[width=\linewidth]{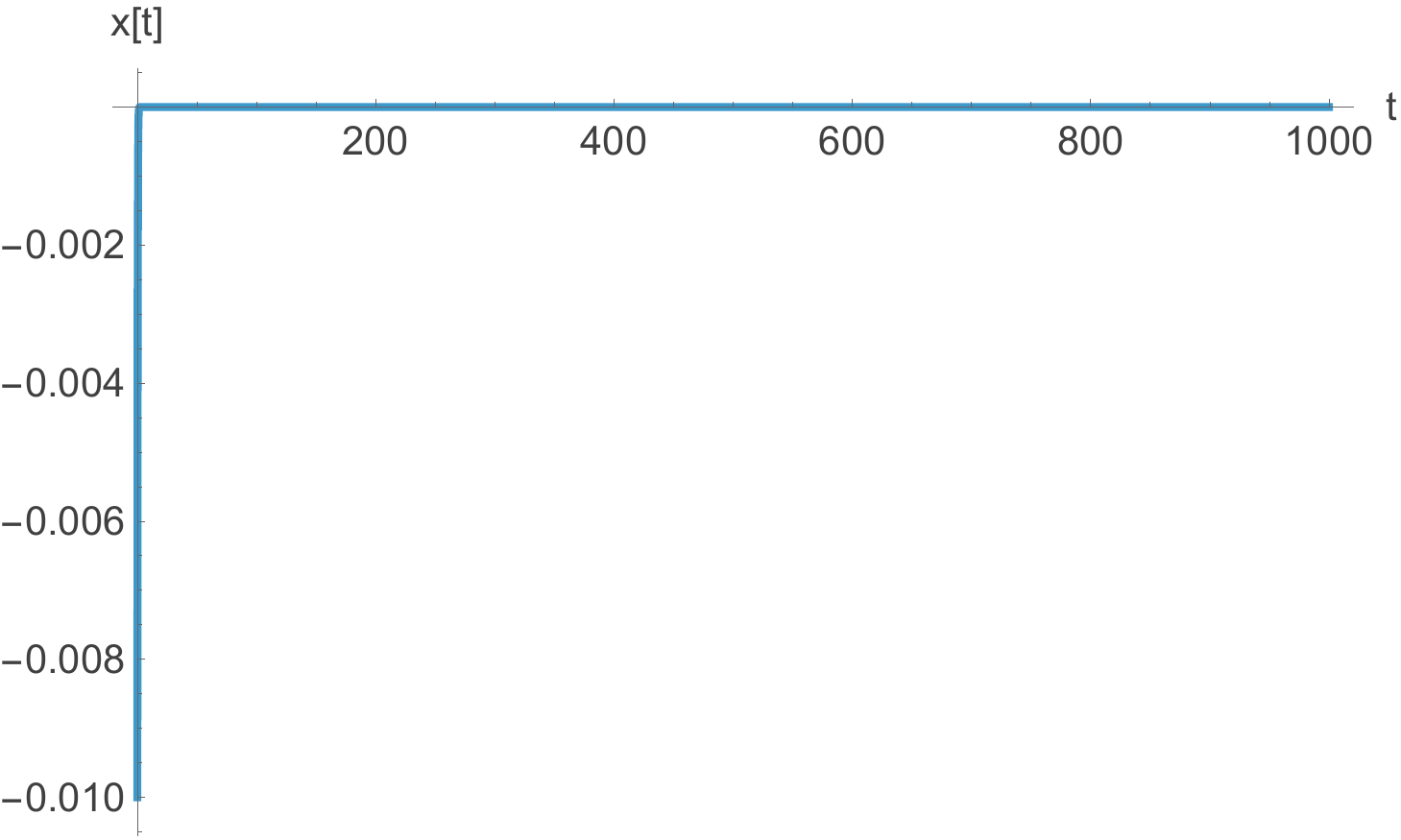}
    \caption{$\tau_2=0$}
	\label{tau20}
\end{subfigure}
\hfill
\begin{subfigure}{0.32\textwidth}
    \centering
    \includegraphics[width=\linewidth]{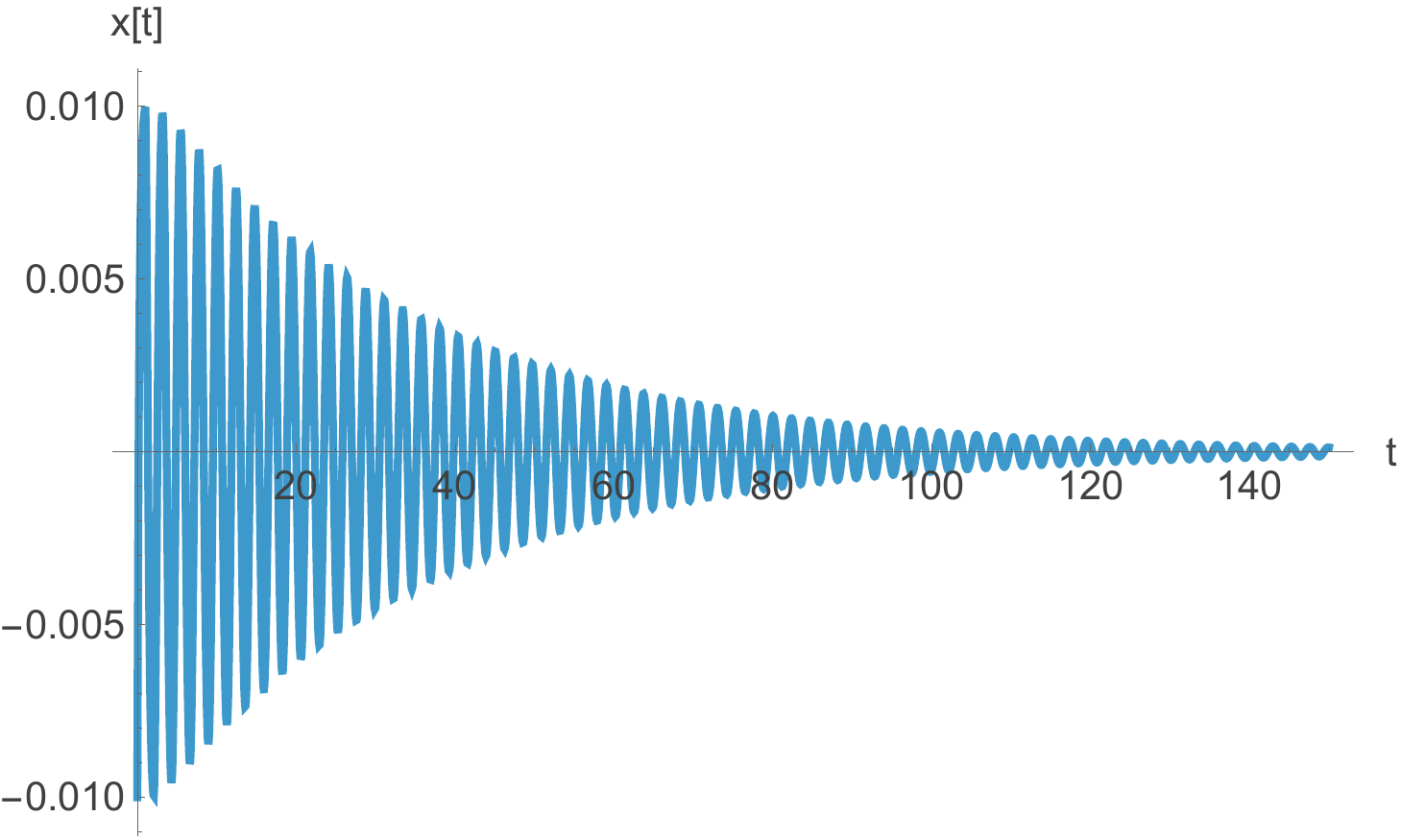}
    \caption{$\tau_2=0.5$}
\end{subfigure}
\hfill
\begin{subfigure}{0.32\textwidth}
    \centering
    \includegraphics[width=\linewidth]{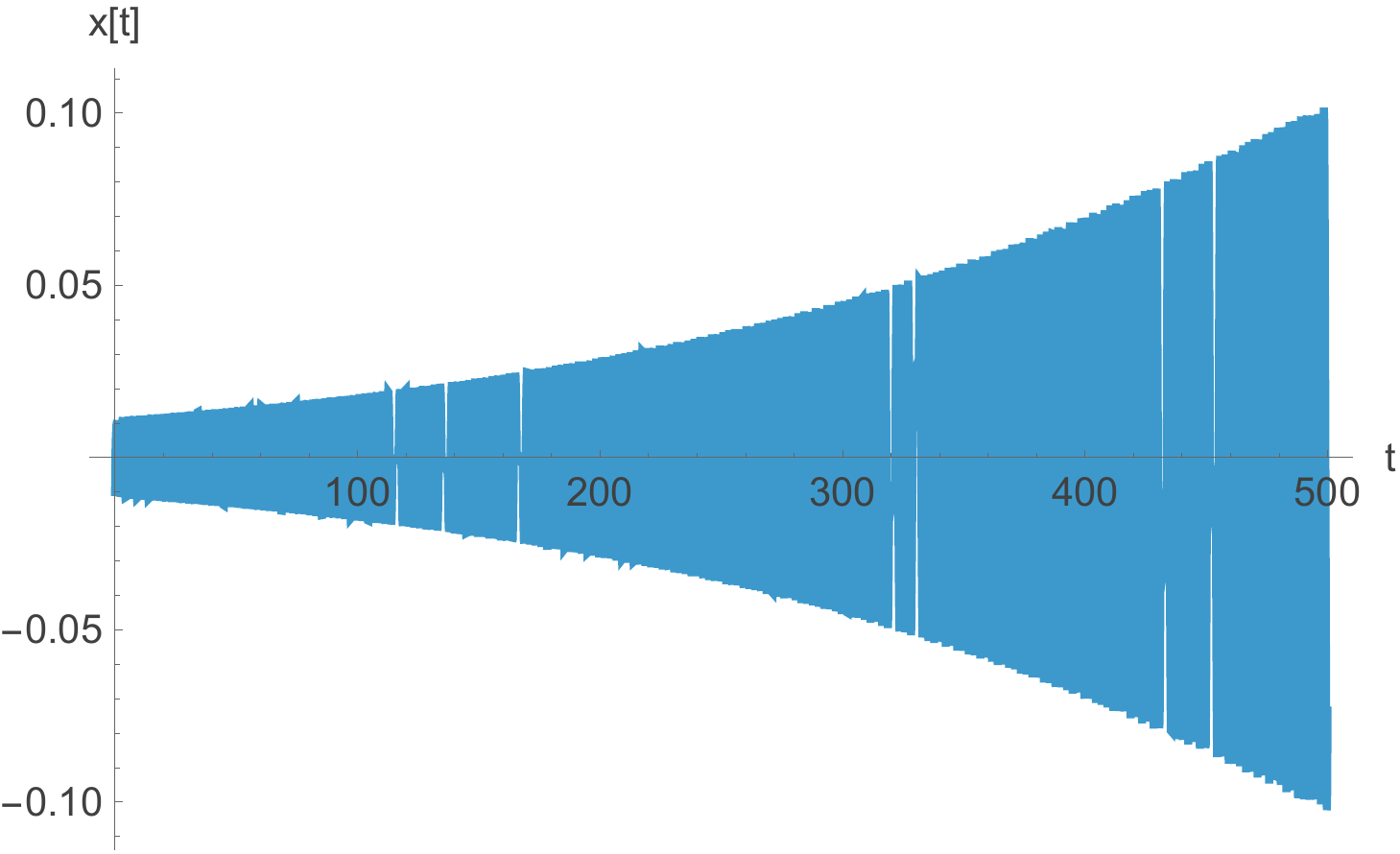}
    \caption{$\tau_2=0.7$}
\end{subfigure}

\vspace{0.4cm}

\begin{subfigure}{0.32\textwidth}
    \centering
    \includegraphics[width=\linewidth]{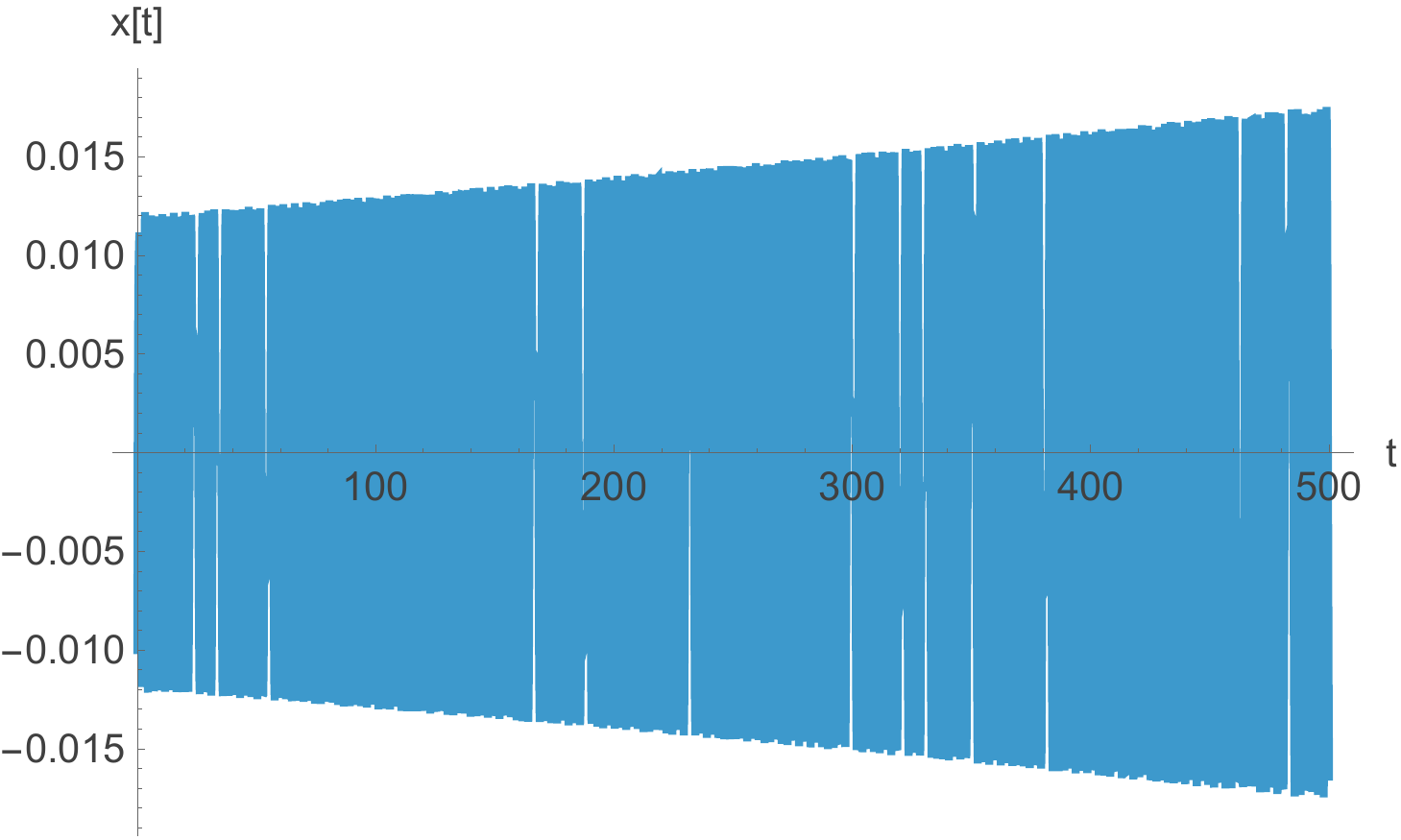}
    \caption{$\tau_2=1.2$}
\end{subfigure}
\hfill
\begin{subfigure}{0.32\textwidth}
    \centering
    \includegraphics[width=\linewidth]{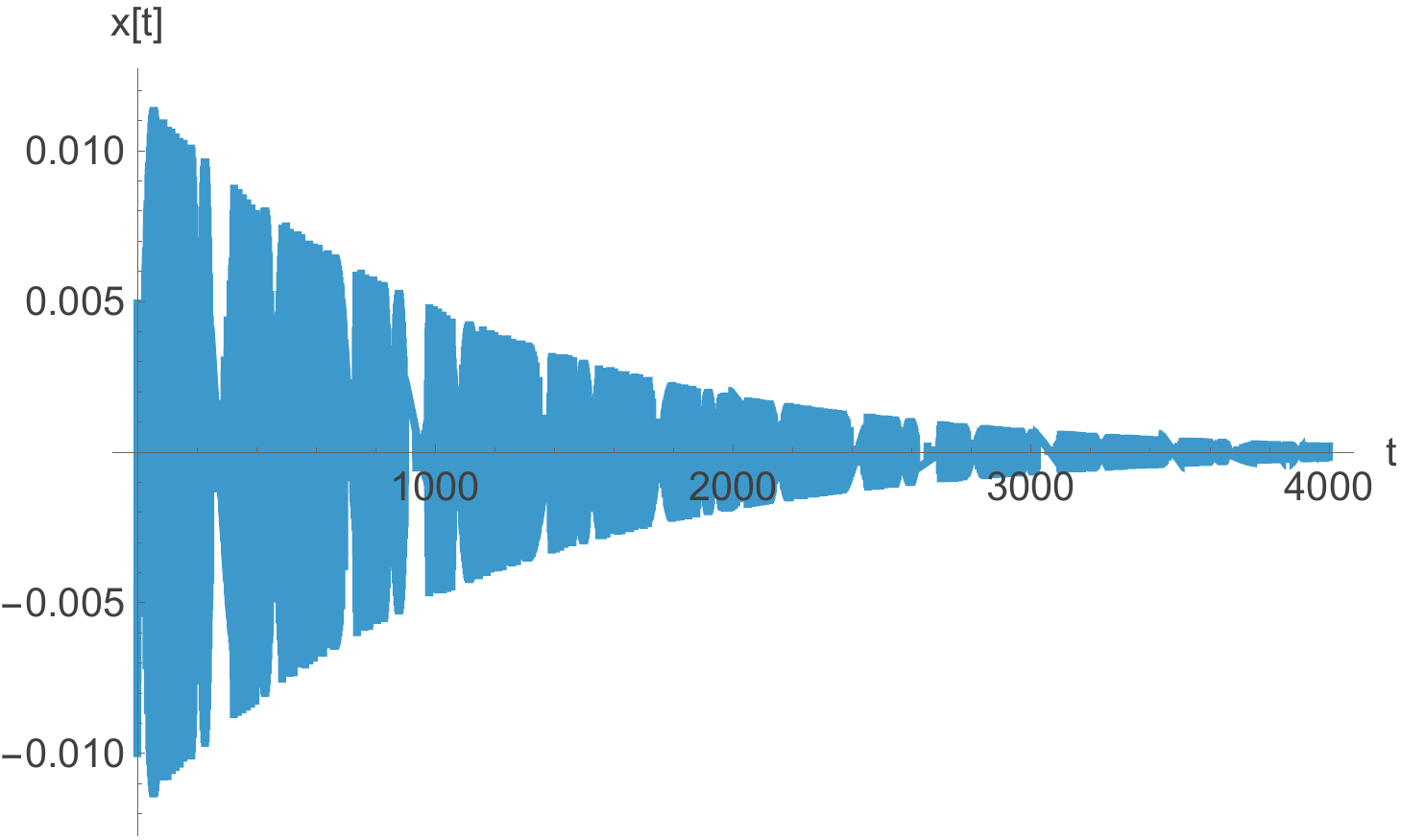}
    \caption{$\tau_2=1.3$}
\end{subfigure}
\hfill
\begin{subfigure}{0.32\textwidth}
    \centering
    \includegraphics[width=\linewidth]{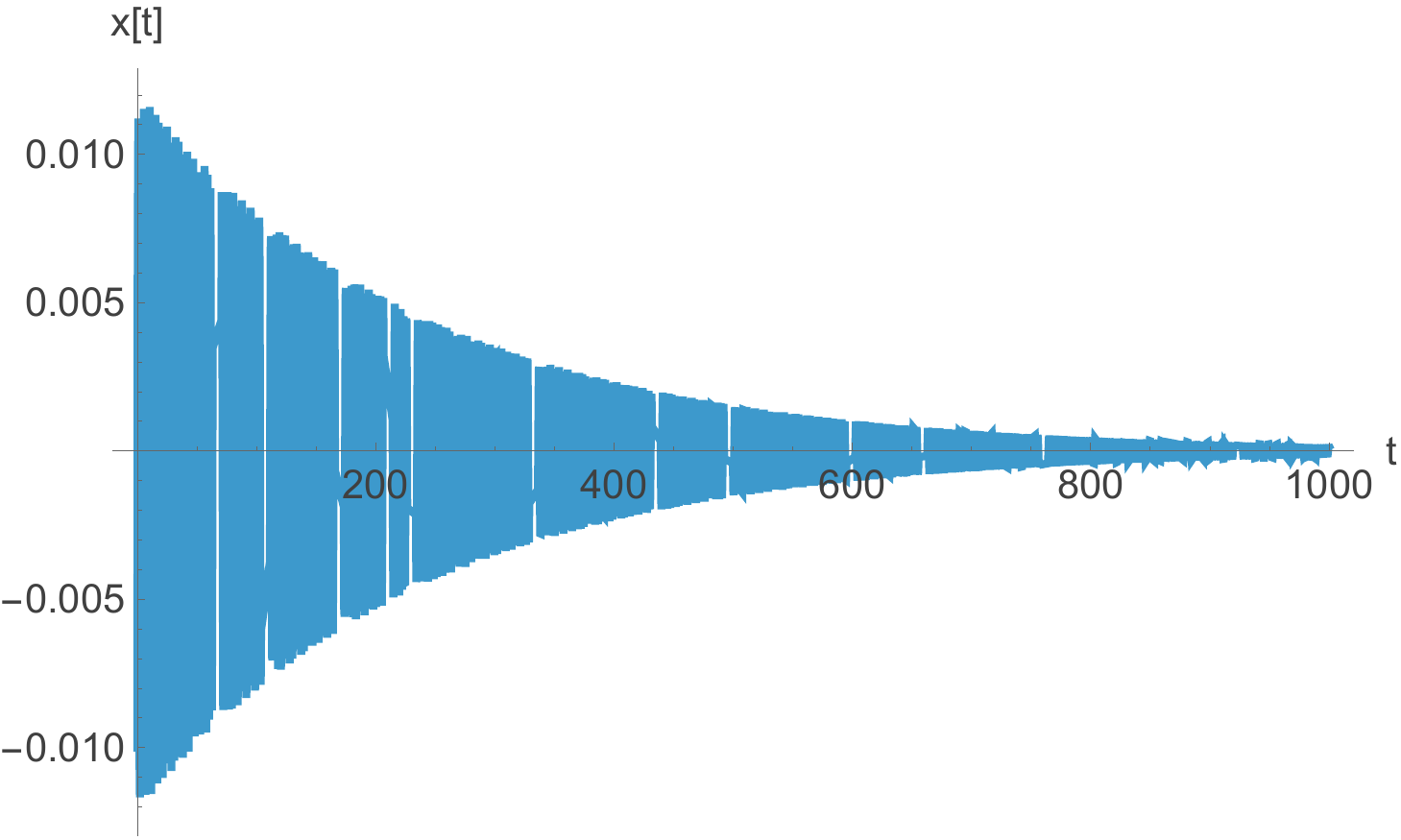}
    \caption{$\tau_2=2$}
\end{subfigure}

\caption{Time series for $k=-5,\gamma=4.3,\tau_1=1$}
\label{fig:stability_switch}
\end{figure}
}
\section{Chaotic Dynamics}
\label{sec:chaos}
In this section, we investigate the nonlinear dynamical behavior of system~\eqref{eq:sinx} induced by time-delay effects. Numerical simulations are carried out to examine the influence of the second delay parameter $\tau_2$, while the remaining parameters are kept fixed. All simulations, phase portraits, and bifurcation diagrams presented in this section are constructed with respect to the trivial equilibrium point $x_*=0$.
All numerical simulations presented in this section are performed by prescribing a constant initial history
\[
x(t)=0.01, \qquad t \in \big[-(\tau_1+\tau_2),\,0\big],
\]
which corresponds to a small perturbation in a neighborhood of the trivial equilibrium point $x=0$.

For $k = 1$, $\gamma = 0.1$, and $\tau_1 = 5.1$, the system exhibits a rich variety of dynamical behaviors as the delay parameter $\tau_2$ is varied. To illustrate these transitions, a bifurcation diagram with respect to $\tau_2 > 0$ is constructed and shown in Figure ~\ref{4_g}.

\begin{figure}[h!]
	\centering
	\includegraphics[scale=0.4]{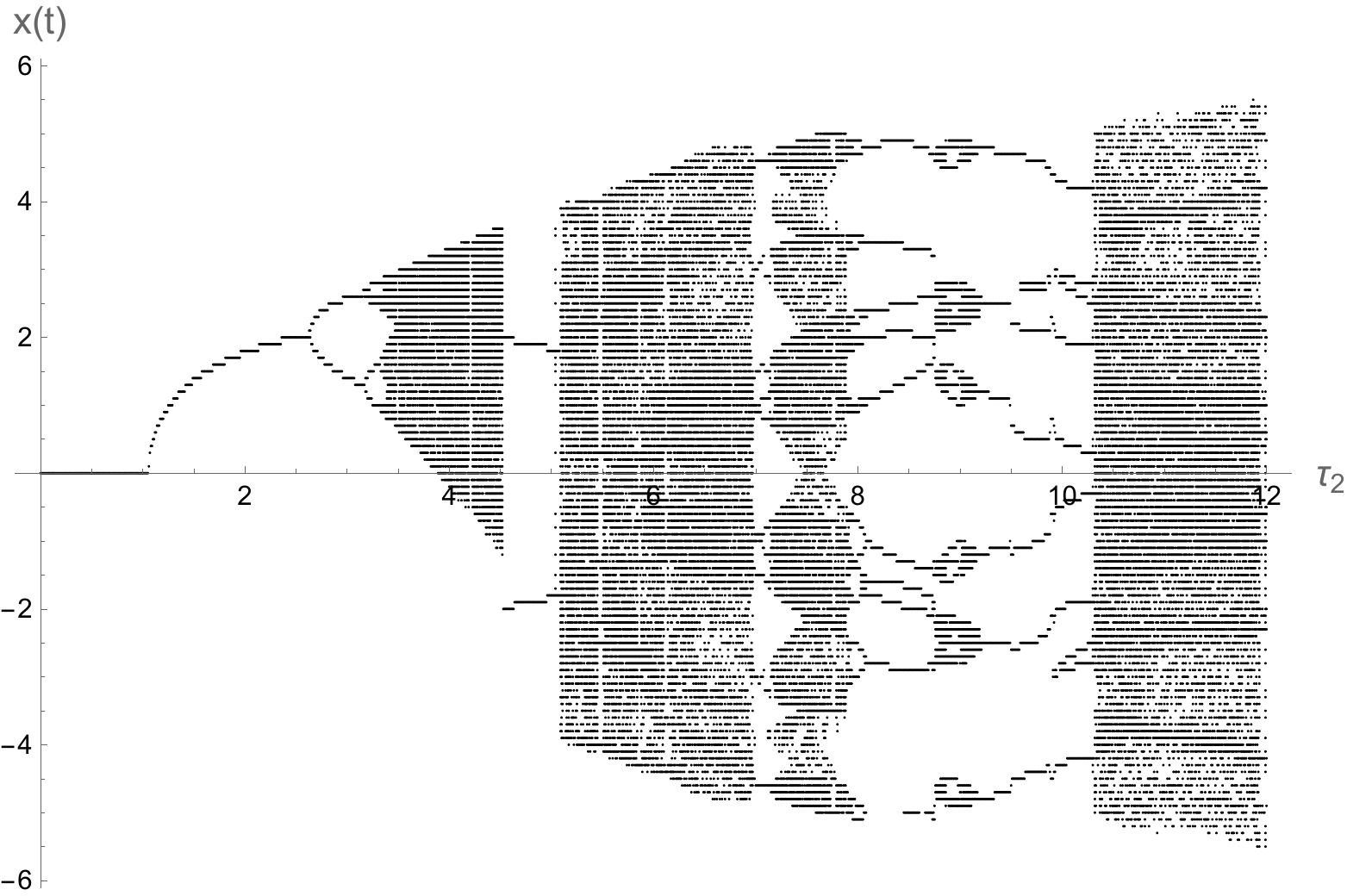}
	\caption{Bifurcation diagram of system~\eqref{eq:sinx} with respect to $\tau_2$ for $k=1$, $\gamma=0.1$, and $\tau_1=5.1$.}
	\label{4_g}
\end{figure}

Based on the bifurcation diagram and numerical simulations, the following observations are made:

\begin{itemize}
	\item \textbf{$0 \leq \tau_2 \leq 1.05$:}  
	The equilibrium point $x_* = 0$ is asymptotically stable. All trajectories converge to the trivial equilibrium, as illustrated in Figure ~\ref{fig:tau2_1} for $\tau_2 = 1$.
	
	\item \textbf{$1.05 < \tau_2 \leq 2.63$:}  
	The trivial equilibrium becomes unstable; however, trajectories are attracted to another equilibrium point that depends on $\tau_2$. For example, at $\tau_2 = 2$, the system converges to a nontrivial equilibrium $x_{2*} = 1.7750$, as shown in Figure ~\ref{fig:tau2_2}.
	
	\item \textbf{$2.63 < \tau_2 \leq 3.12$:}  
	The system exhibits stable period-2 oscillations. The equilibrium loses stability through a delay-induced bifurcation, giving rise to periodic motion. This behavior is confirmed for $\tau_2 = 2.8$ in Figure ~\ref{fig:tau2_28}.
	
	\item \textbf{$3.12 < \tau_2 \leq 3.5$:}  
	A period-doubling phenomenon is observed, leading to higher-period oscillations. This transition is evident in the phase portrait and time series for $\tau_2 = 3.3$, shown in Figures ~\ref{fig:tau2_33_phase} and~\ref{fig:tau2_33_ts}.
	
	\item \textbf{$3.5 < \tau_2 \leq 4.19$:}  
	The system exhibits chaotic oscillations for the first time in this interval. The periodic orbit loses stability and a single-scroll chaotic attractor emerges, as illustrated in Figure ~\ref{fig:tau2_38} for $\tau_2 = 3.8$.
	
	\item \textbf{$4.19 < \tau_2 \leq 4.22$:}  
	A narrow periodic window appears within the chaotic regime. Periodic oscillations and the corresponding time series are shown in Figures ~\ref{fig:tau2_42_phase} and ~\ref{fig:tau2_42_ts} for $\tau_2 = 4.2$.
	
	\item \textbf{$4.22 < \tau_2 \leq 4.51$:}  
	Chaotic behavior reappears after the periodic window. The chaotic attractor observed for $\tau_2 = 4.38$ is shown in Figure ~\ref{fig:tau2_438}.
	
	\item \textbf{$4.51 < \tau_2 < 5.11$:}  
	Another periodic window is observed. For $\tau_2 = 5$, the system exhibits stable periodic oscillations, as illustrated in the phase portrait and time series in Figures ~\ref{fig:tau2_5_phase} and~\ref{fig:tau2_5_ts}.
	
\item \textbf{$5.11 < \tau_2 < 6.98$:}  
This interval is predominantly chaotic, interrupted by narrow embedded periodic 
windows. For $\tau_2 = 5.4$, the system exhibits chaotic dynamics characterized by a 
double-scroll attractor (Figure ~\ref{fig:tau2_54}). As $\tau_2$ is increased slightly, 
the chaotic motion is suppressed and a periodic solution emerges at 
$\tau_2 = 5.47$, as illustrated in Figures ~\ref{fig:tau2_547_phase} 
and~\ref{fig:tau2_547_ts}. Upon further increase of $\tau_2$, the system transitions 
back to chaos, and for $\tau_2 = 6.2$ the chaotic attractor again displays a 
double-scroll structure (Figure ~\ref{fig:tau2_62}).

\end{itemize}

\begin{figure}[h!]
	\centering
	\begin{subfigure}[b]{0.31\textwidth}
		\includegraphics[height=3cm]{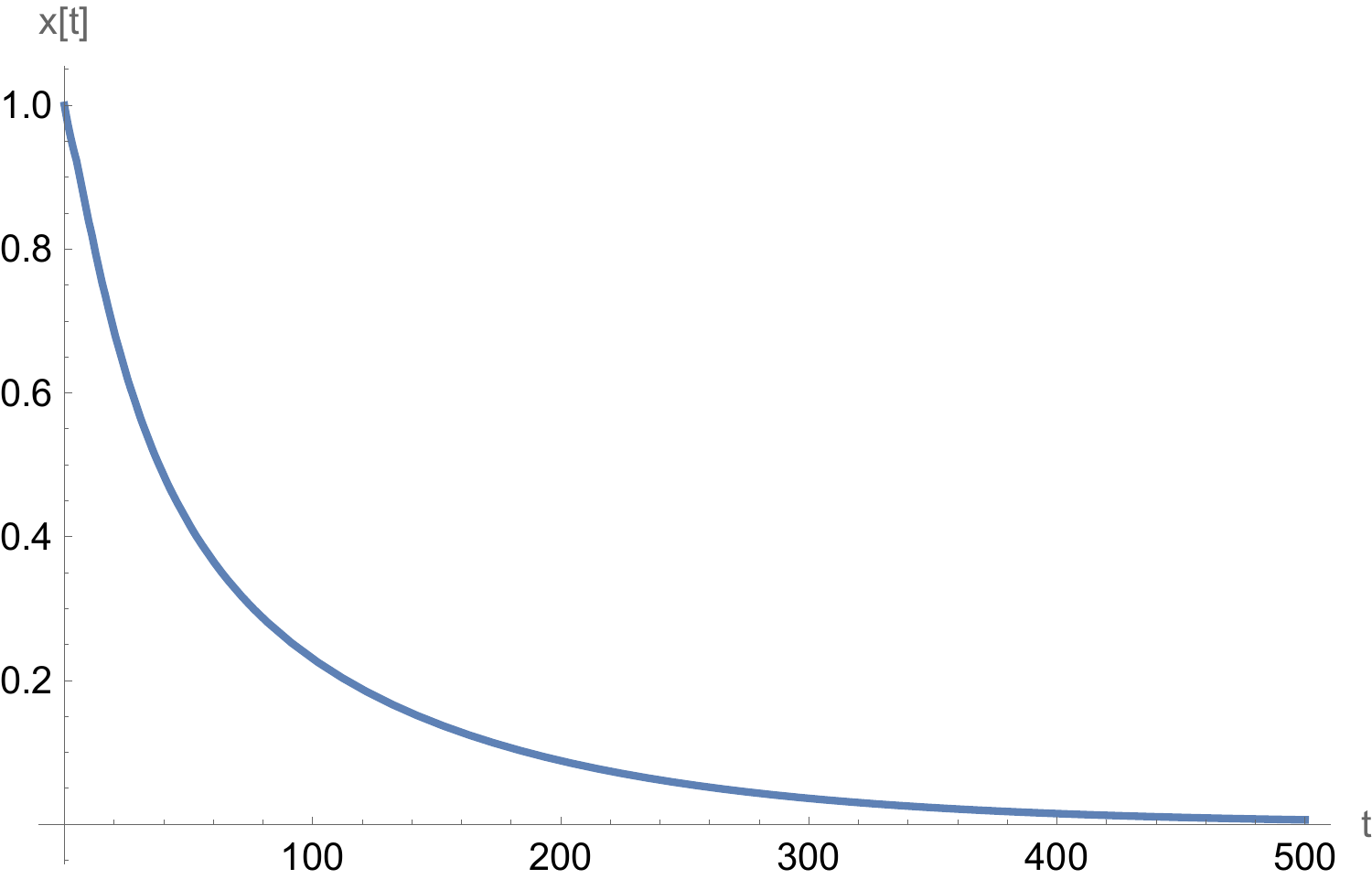}
		\caption{$\tau_2 = 1$}
		\label{fig:tau2_1}
	\end{subfigure}
	\hfill
	\begin{subfigure}[b]{0.31\textwidth}
		\includegraphics[height=3cm]{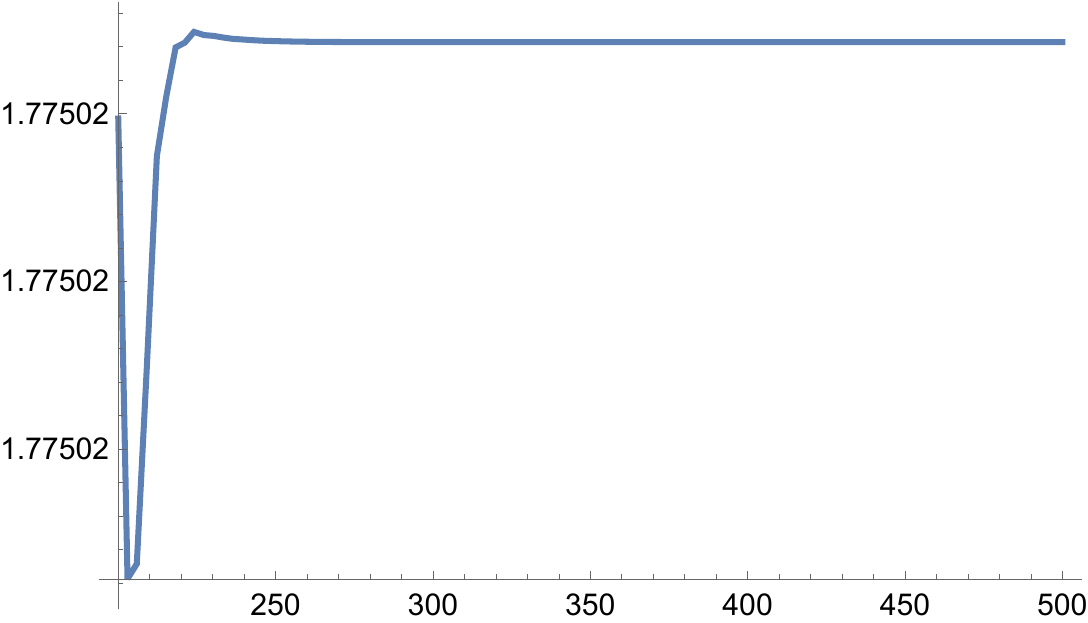}
		\caption{$\tau_2 = 2$}
		\label{fig:tau2_2}
	\end{subfigure}
	\hfill
	\begin{subfigure}[b]{0.31\textwidth}
		\includegraphics[height=3cm]{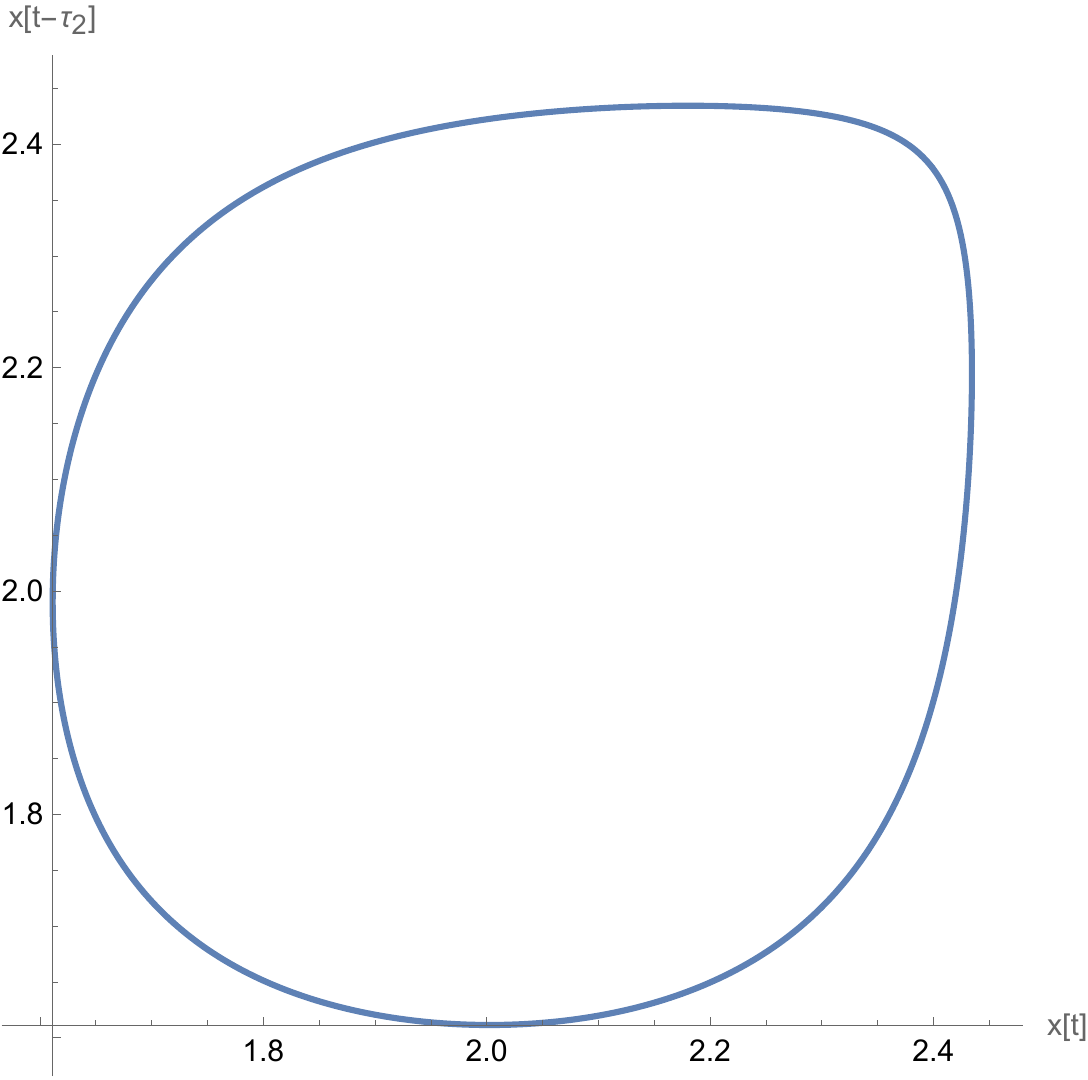}
		\caption{$\tau_2 = 2.8$}
		\label{fig:tau2_28}
	\end{subfigure}
	\caption{Phase portraits showing convergence to equilibrium and the emergence of periodic oscillations for increasing values of $\tau_2$.}
\end{figure}

\begin{figure}[h!]
	\centering
	\begin{subfigure}[b]{0.31\textwidth}
		\includegraphics[height=3cm]{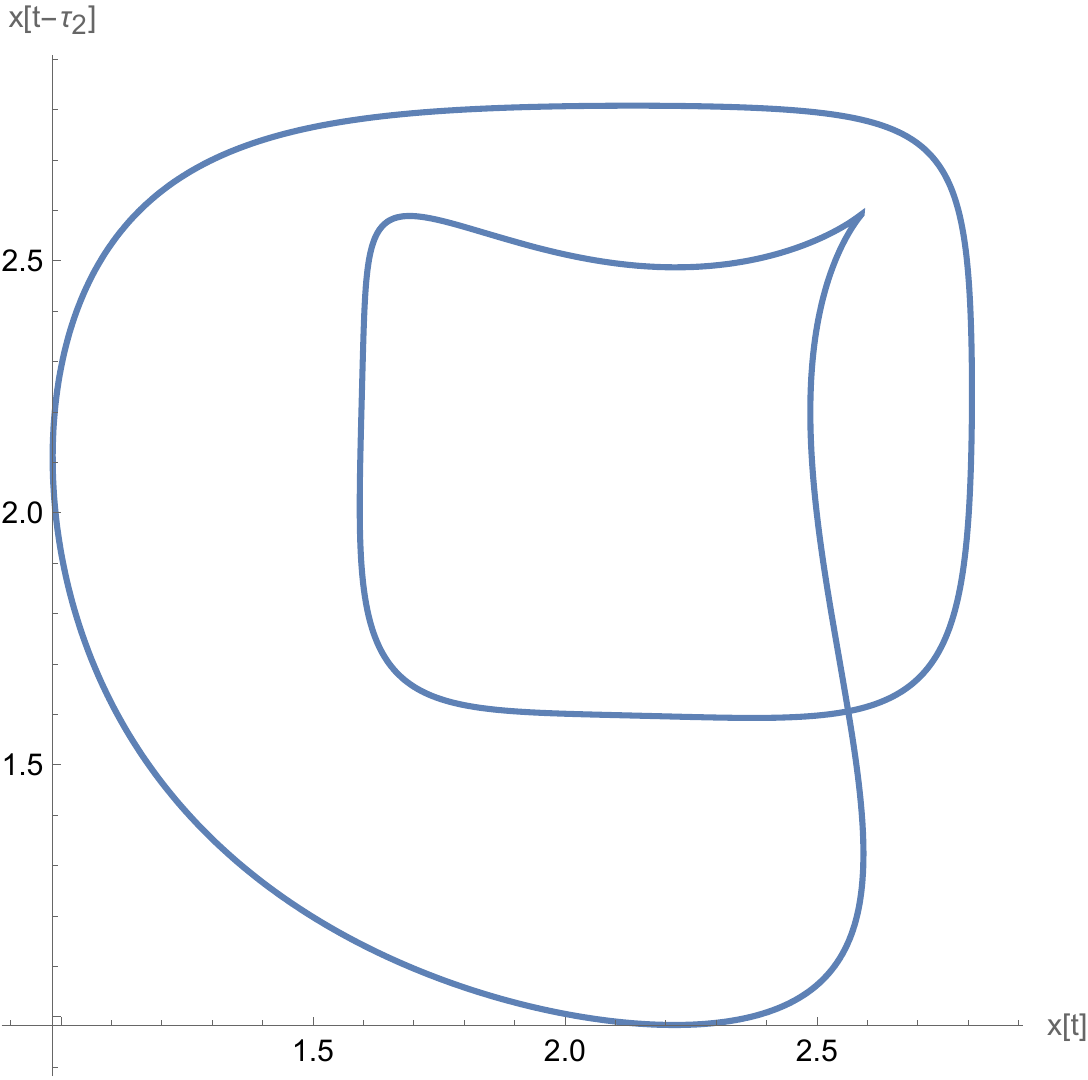}
		\caption{$\tau_2 = 3.3$}
		\label{fig:tau2_33_phase}
	\end{subfigure}
	\hfill
	\begin{subfigure}[b]{0.31\textwidth}
		\includegraphics[height=3cm]{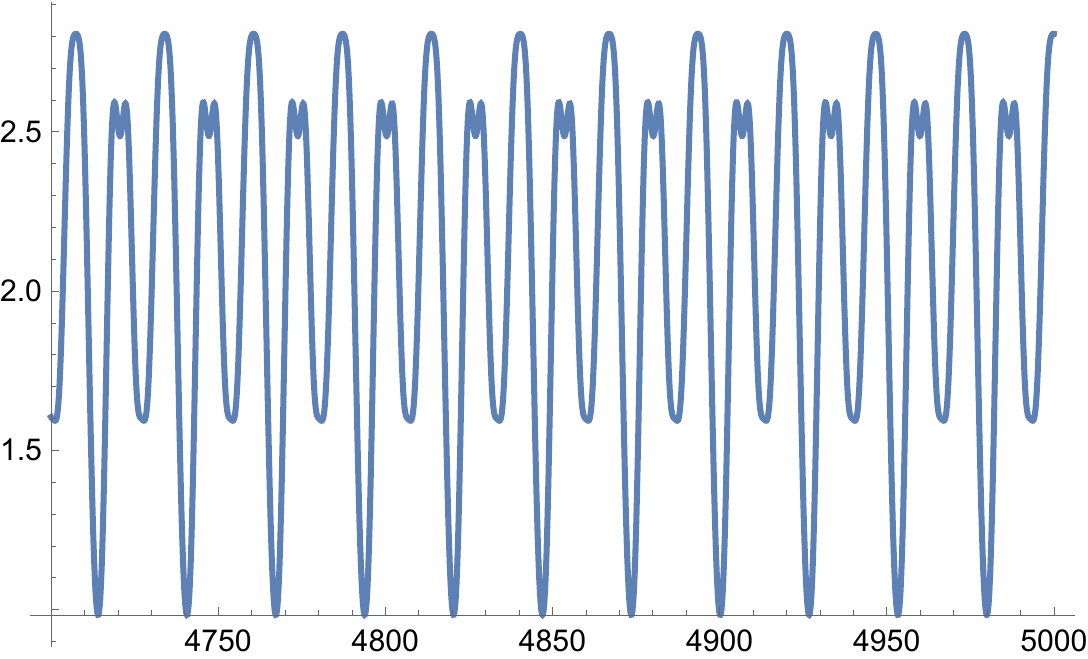}
		\caption{Time series for $\tau_2 = 3.3$}
		\label{fig:tau2_33_ts}
	\end{subfigure}
	\hfill
	\begin{subfigure}[b]{0.31\textwidth}
		\includegraphics[height=3cm]{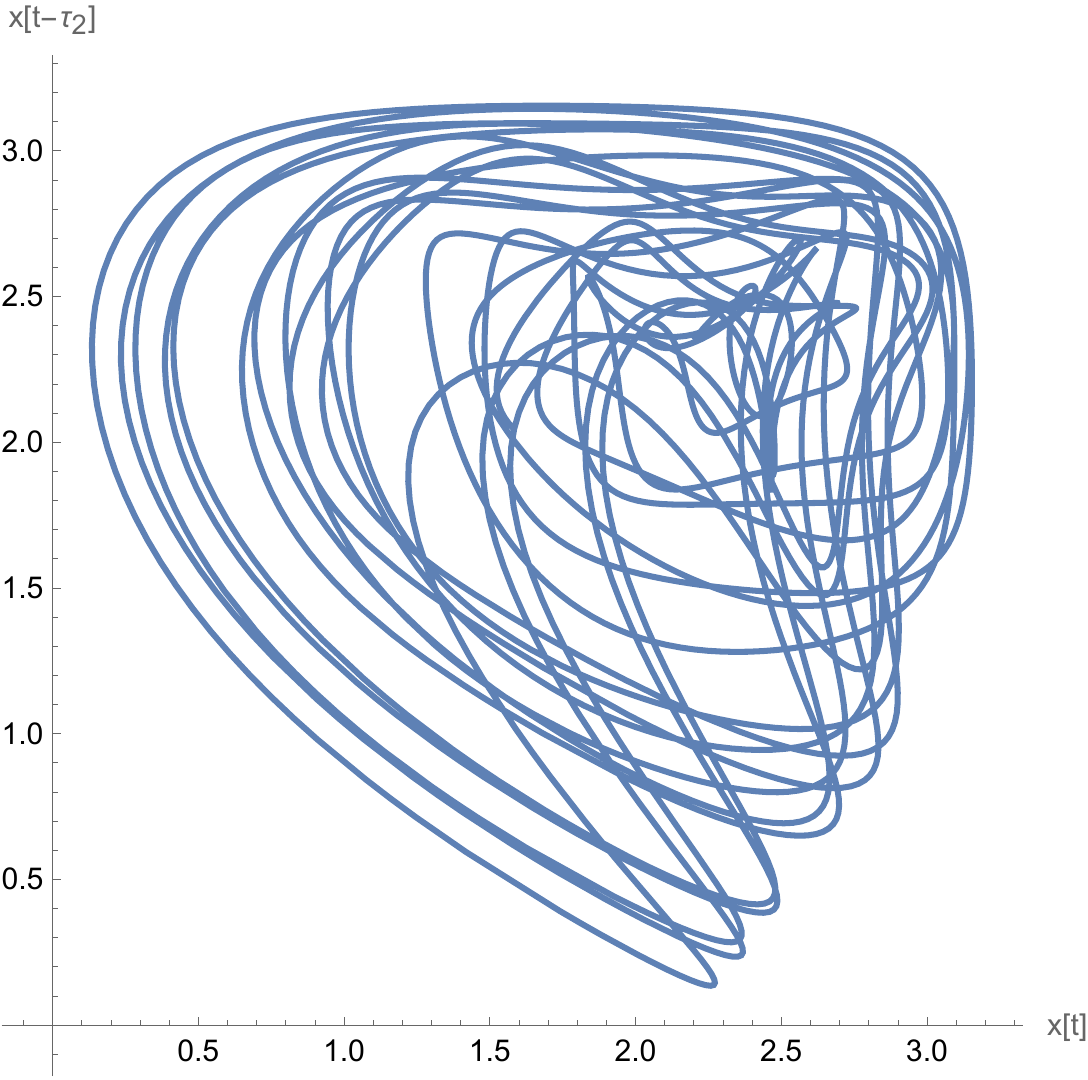}
		\caption{$\tau_2 = 3.8$}
		\label{fig:tau2_38}
	\end{subfigure}
	\caption{Transition from periodic motion to single-scroll chaotic oscillations.}
\end{figure}

\begin{figure}[h!]
	\centering
	\begin{subfigure}[b]{0.31\textwidth}
		\includegraphics[height=3cm]{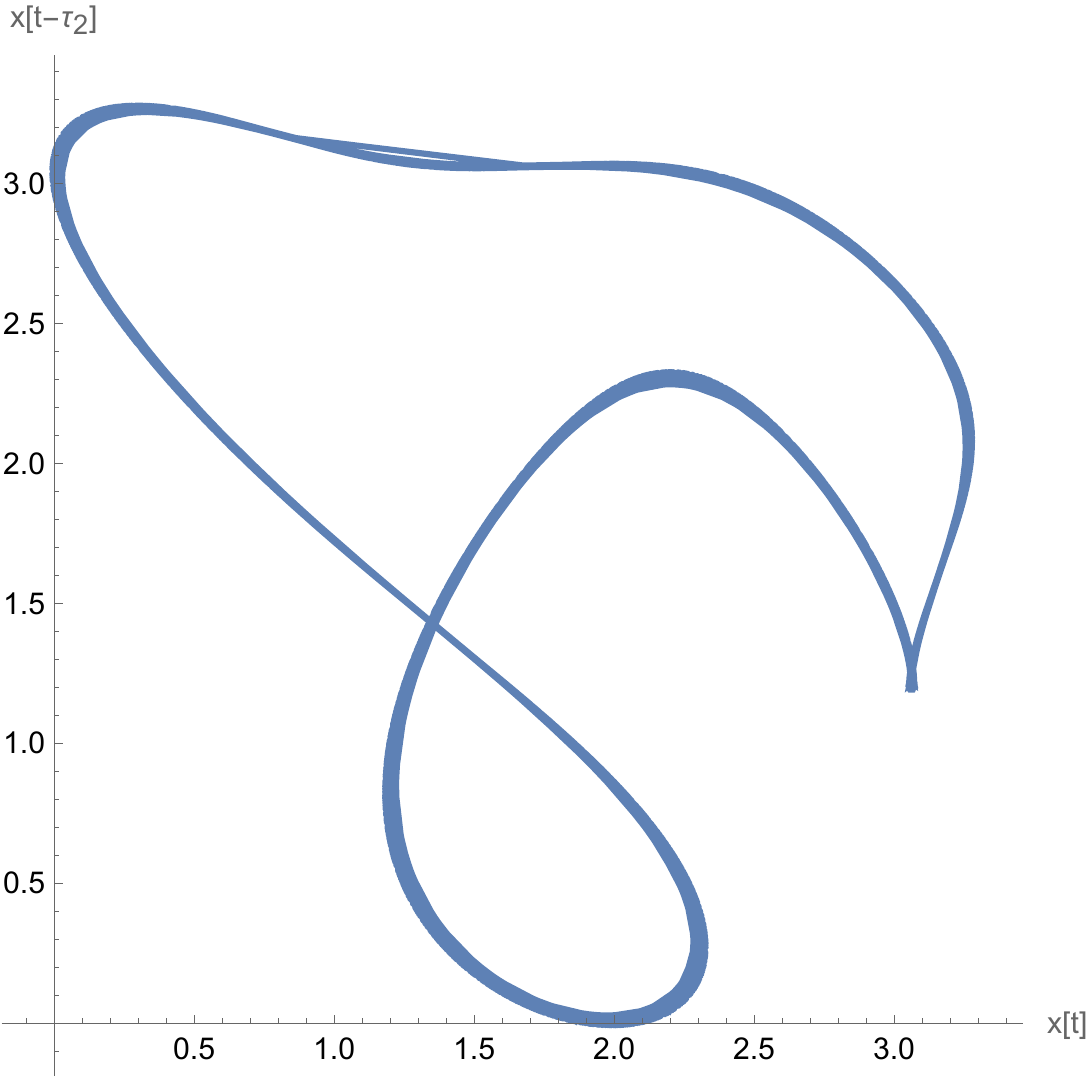}
		\caption{$\tau_2 = 4.2$}
		\label{fig:tau2_42_phase}
	\end{subfigure}
	\hfill
	\begin{subfigure}[b]{0.31\textwidth}
		\includegraphics[height=3cm]{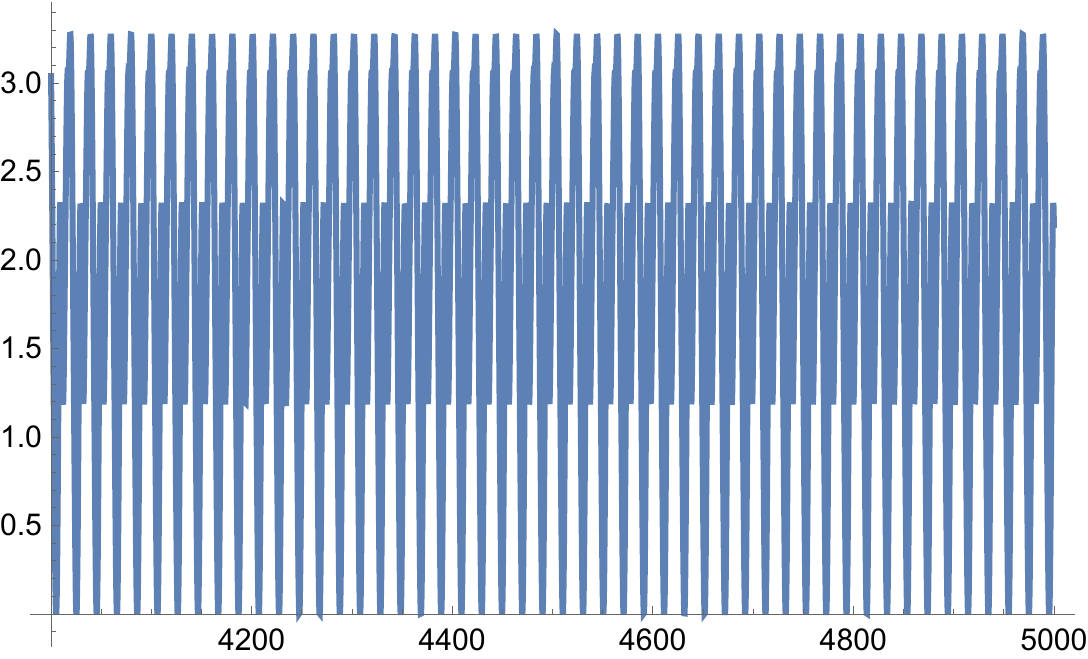}
		\caption{Time series for $\tau_2 = 4.2$}
		\label{fig:tau2_42_ts}
	\end{subfigure}
	\hfill
	\begin{subfigure}[b]{0.31\textwidth}
		\includegraphics[height=3cm]{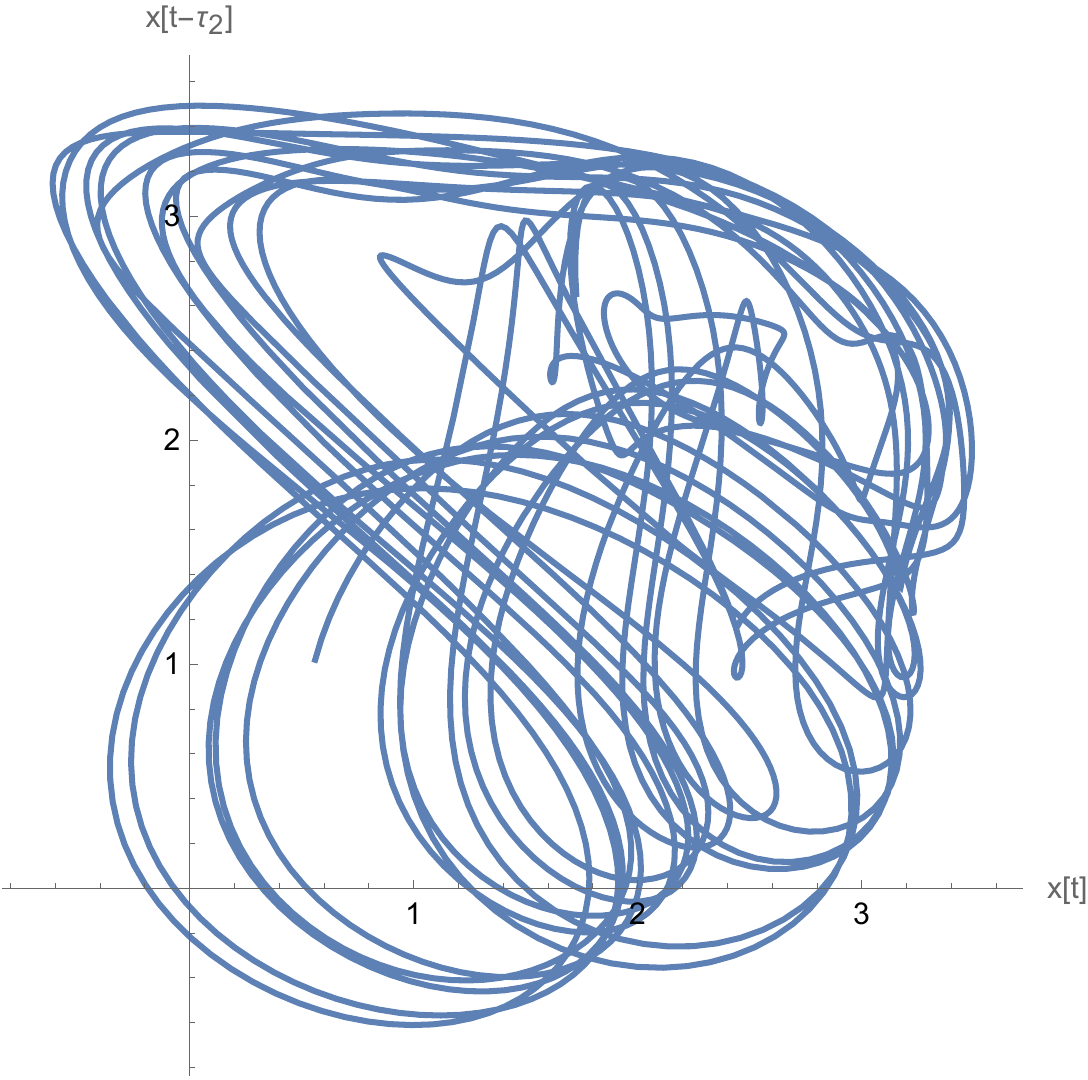}
		\caption{$\tau_2 = 4.38$}
		\label{fig:tau2_438}
	\end{subfigure}
	\caption{Periodic window followed by chaotic oscillations.}
\end{figure}

\begin{figure}[h!]
	\centering
	\begin{subfigure}[b]{0.31\textwidth}
		\includegraphics[height=3cm]{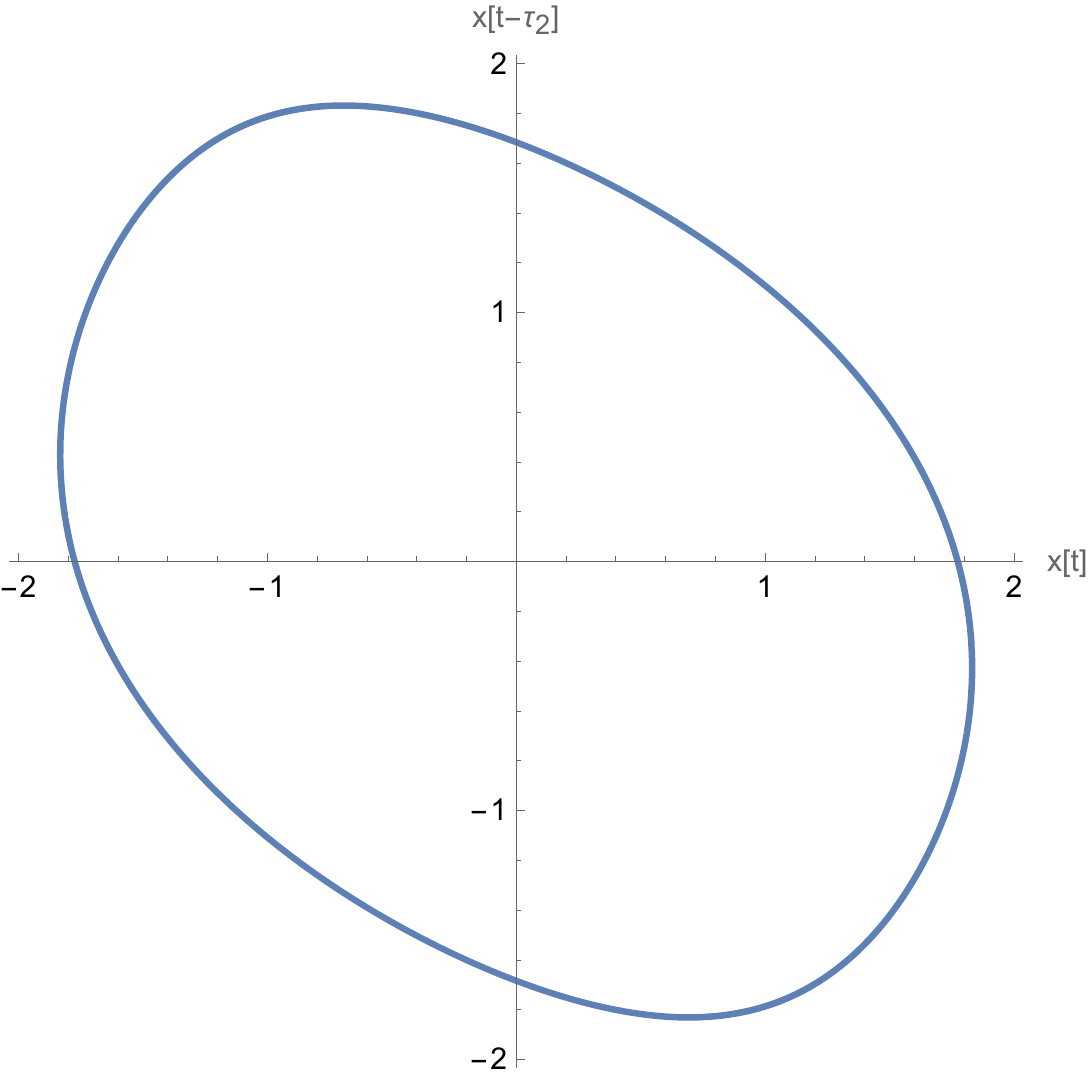}
		\caption{$\tau_2 = 5$}
		\label{fig:tau2_5_phase}
	\end{subfigure}
	\hfill
	\begin{subfigure}[b]{0.31\textwidth}
		\includegraphics[height=3cm]{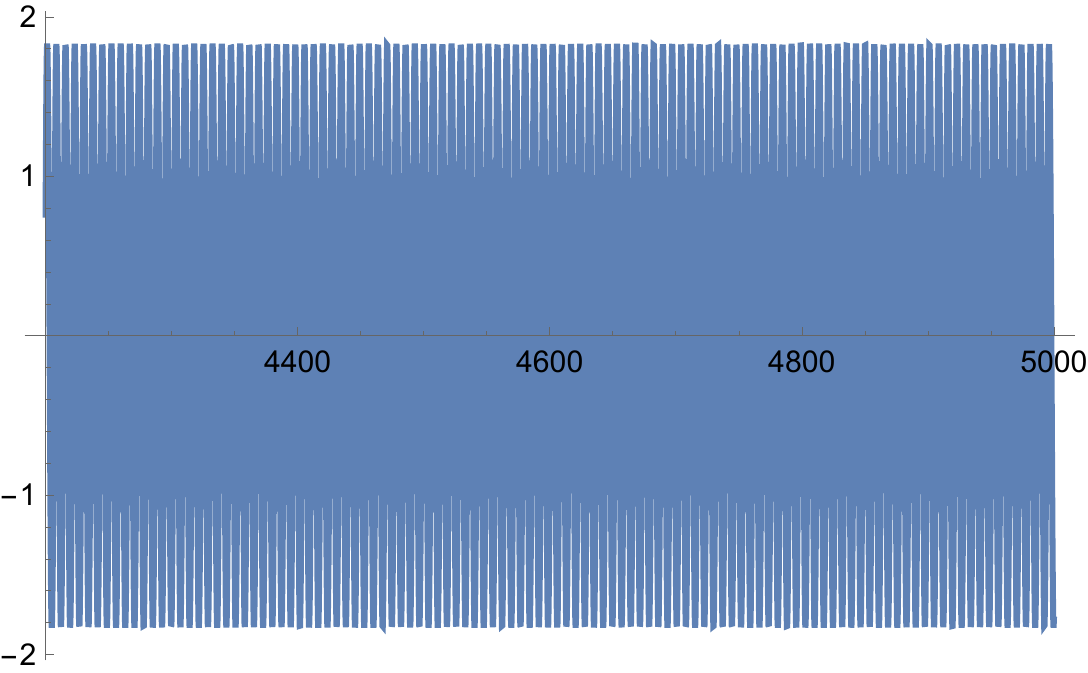}
		\caption{Time series for $\tau_2 = 5$}
		\label{fig:tau2_5_ts}
	\end{subfigure}
	\hfill
	\begin{subfigure}[b]{0.31\textwidth}
		\includegraphics[height=3cm]{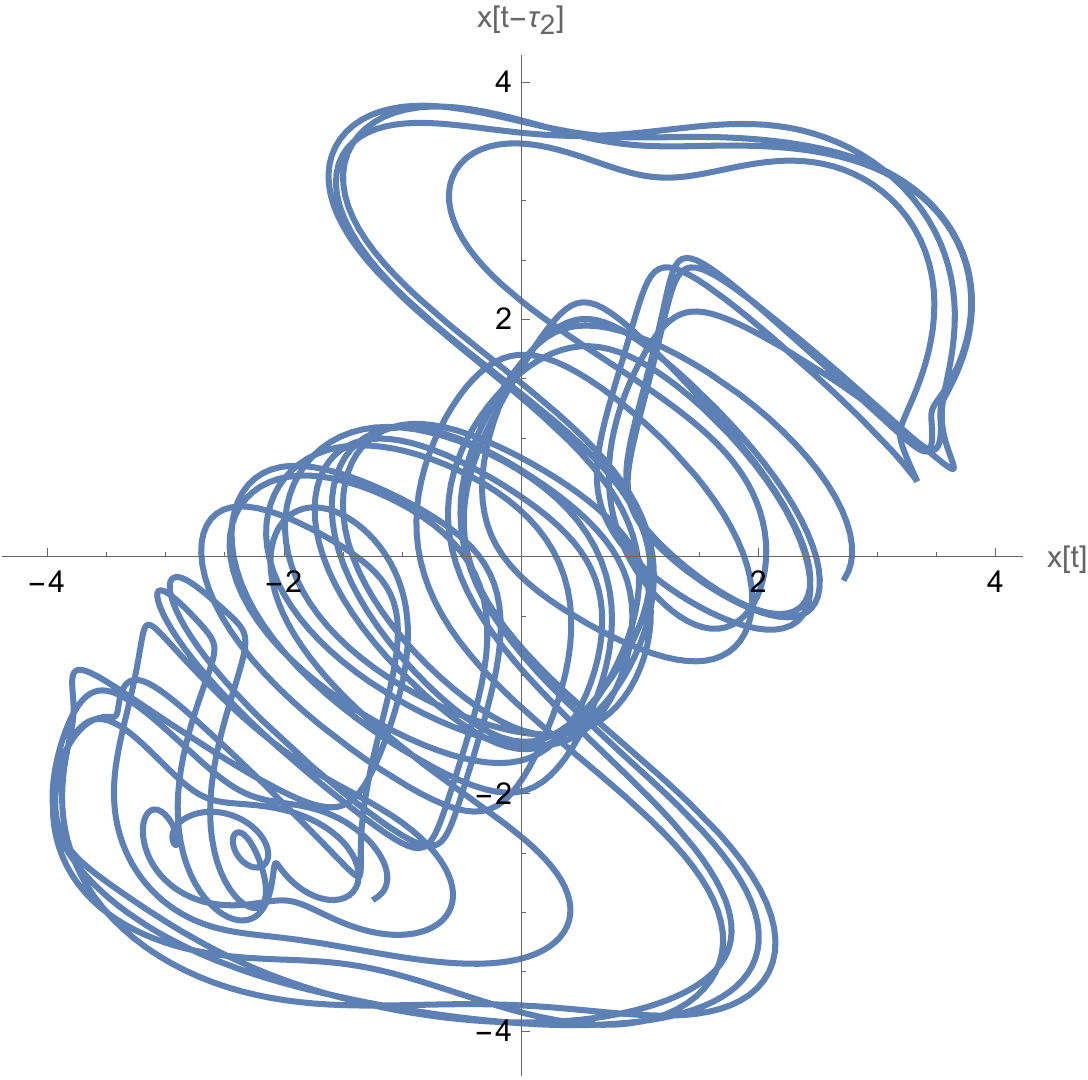}
		\caption{$\tau_2 = 5.4$}
		\label{fig:tau2_54}
	\end{subfigure}
	\caption{Transition from periodic dynamics to double-scroll chaotic behavior as $\tau_2$ is increased.}
\end{figure}

\begin{figure}[h!]
	\centering
	\begin{subfigure}[b]{0.31\textwidth}
		\includegraphics[height=3cm]{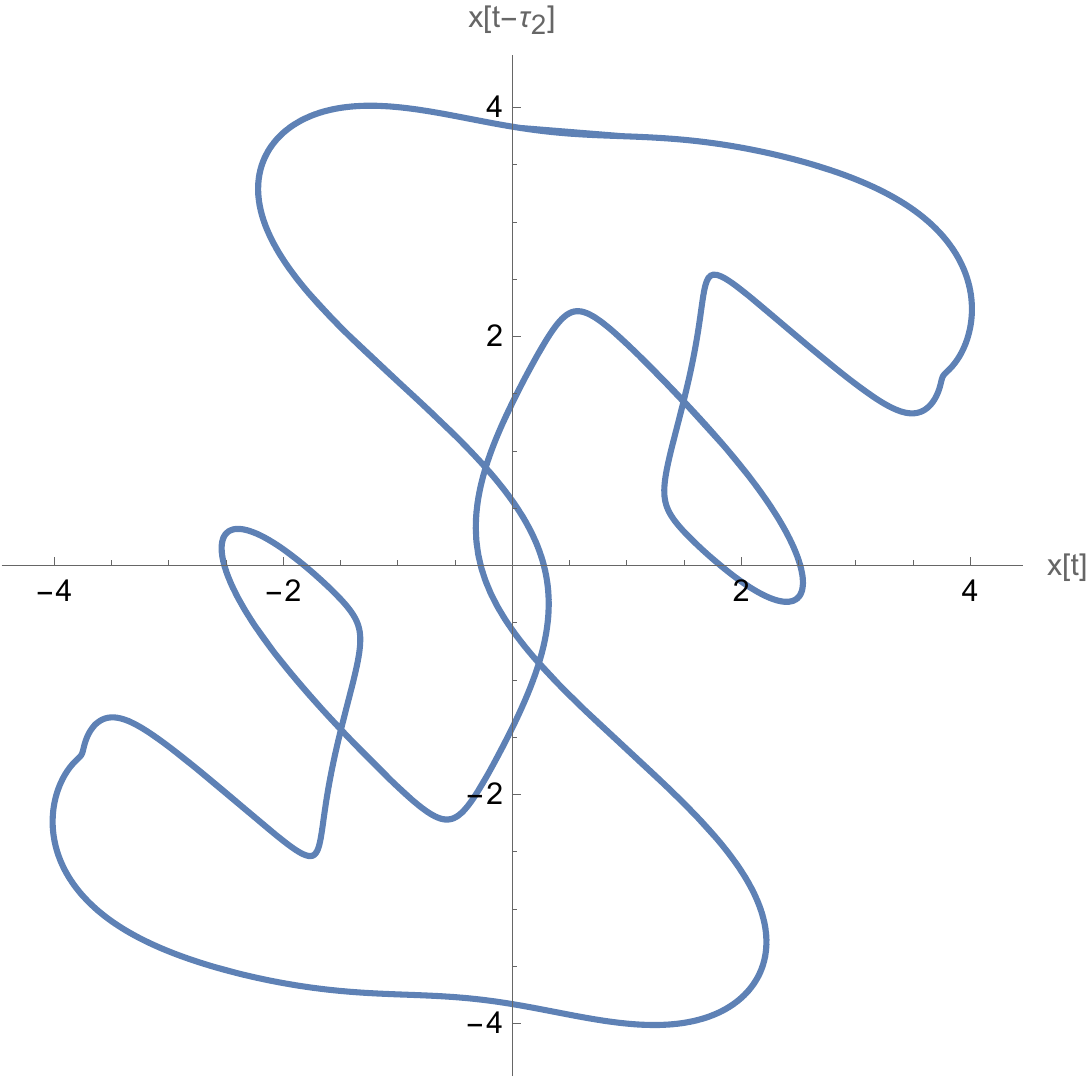}
		\caption{Phase portrait showing periodic oscillations at $\tau_2 = 5.47$}
		\label{fig:tau2_547_phase}
	\end{subfigure}
	\hfill
	\begin{subfigure}[b]{0.31\textwidth}
		\includegraphics[height=3cm]{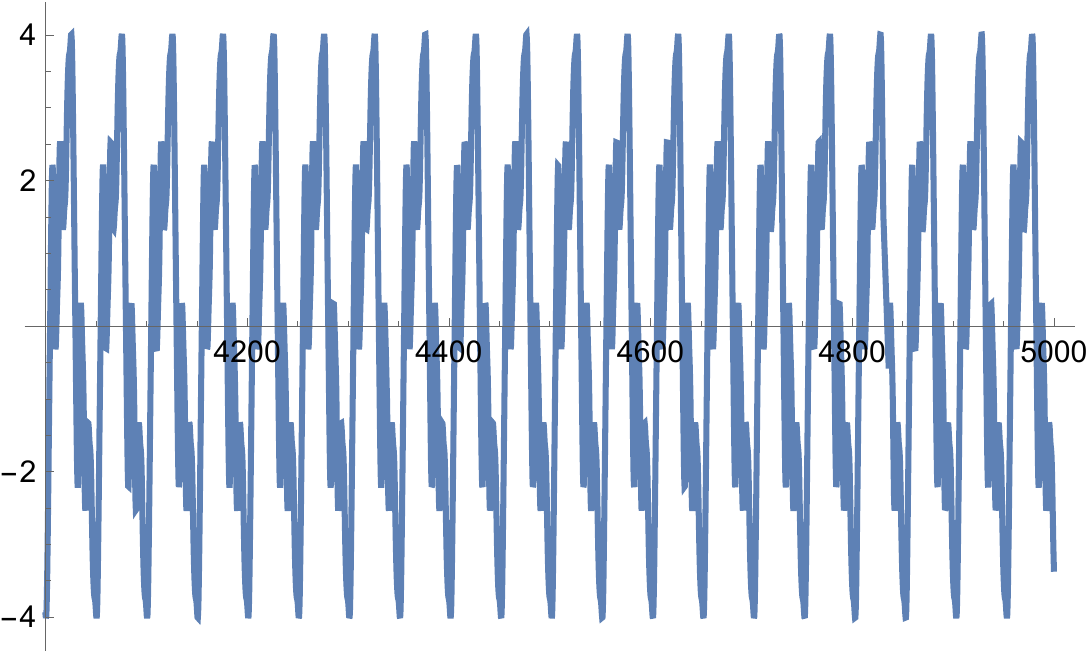}
		\caption{Time series confirming periodic dynamics for $\tau_2 = 5.47$}
		\label{fig:tau2_547_ts}
	\end{subfigure}
	\hfill
	\begin{subfigure}[b]{0.31\textwidth}
		\includegraphics[height=3cm]{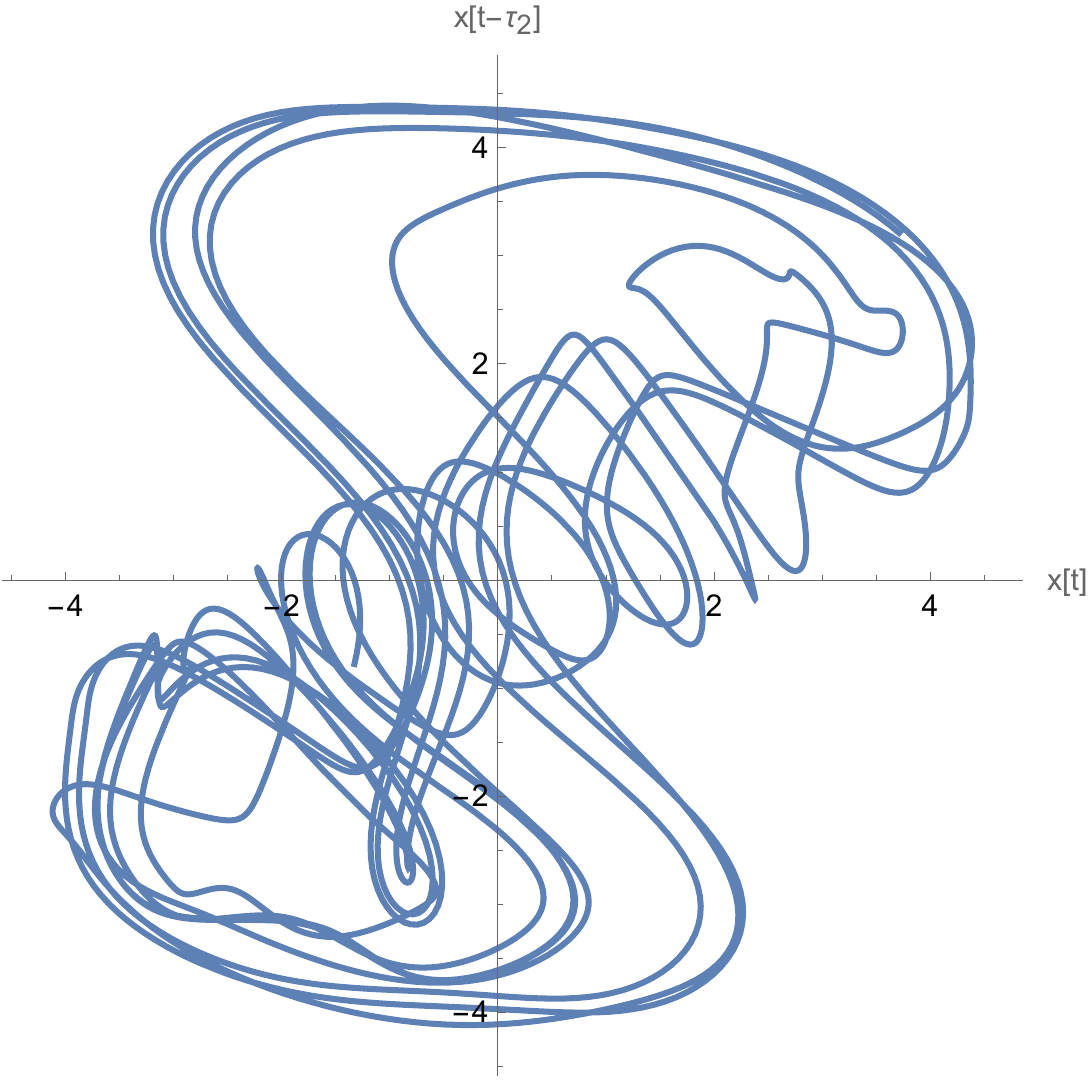}
		\caption{Double-scroll chaotic attractor at $\tau_2 = 6.2$}
		\label{fig:tau2_62}
	\end{subfigure}
	\caption{Reappearance of periodic oscillations within a chaotic regime followed by a return to double-scroll chaos.}
\end{figure}

Consequently, as $\tau_2$ increases, the system undergoes repeated dynamical transitions of the form
\[
\text{chaos}
\;\longrightarrow\;
\text{periodicity}
\;\longrightarrow\;
\text{chaos}
\;\longrightarrow\;
\text{periodicity}
\;\longrightarrow\;
\text{chaos},
\]
demonstrating multiple stability switchings .

	The chaotic nature of the observed attractors is quantitatively confirmed by computing the largest Lyapunov exponent using the algorithm described by Kodba et al. \cite{kodba2005detecting}, which is based on the time series analysis techniques and the work by Wolf et al. \cite{wolf1985determining}. The corresponding values of the maximal Lyapunov exponent for some of the representative values of $\tau_2$ are summarized in Table~\ref{tab:lyap}.

\begin{table}[h!]
	\centering
	\caption{Largest Lyapunov exponents for different values of $\tau_2$.}
	\label{tab:lyap}
	\begin{tabular}{c c c}
		\hline
		$\tau_2$ & $\lambda_{\max}$ & Dynamical behavior \\
		\hline
		1.0 & $-0.162552$ & Stable equilibrium \\
		2.8 & $-0.01$ & Periodic oscillations \\
		3.3 & $3.241006$  & Chaotic motion \\
       	4.2 & $0.07$ & Periodic oscillations\\
		5.4 & $6.261435$  & Strong chaos (double-scroll) \\
		\hline
	\end{tabular}
\end{table}
{
\subsection{Poincar\'e Return Maps and Power Spectral Analysis}
To gain further insight into the nature of the system dynamics, Poincar\'e return maps and Power Spectral Densities (PSDs) are examined for representative periodic and chaotic solutions, as these tools have been widely used to characterize and distinguish periodic and chaotic dynamics in nonlinear systems \cite{shahriari2023ordinal,yan2024construction,thamilmaran2024extreme}.

The Poincar\'e return map is constructed from the sequence of successive local extrema of the time series \cite{burykin2014dynamical}. If $\{x_n\}$ denotes the ordered maxima and minima of $x(t)$, the points $(x_n,x_{n+1})$ are plotted to reveal the underlying structure of the dynamics. This provides a simple geometric representation that distinguishes between regular and irregular motion.

The PSD is computed from the steady-state portion of the time series after removing the initial transient. A Hann window together with segment averaging is used to estimate the one-sided power spectral density, which reduces spectral leakage and provides a smoother spectral estimate \cite{mensour1998power}.

For the periodic case with $\tau_2=2.8$, the Poincar\'e return map consists of a finite set of isolated points, reflecting the repeating nature of the orbit (see Figure (\ref{PR28})) .The corresponding PSD contains sharp spectral lines located at the fundamental frequency and its harmonics, which is characteristic of periodic oscillations (see Figure (\ref{PSD28})) .

In contrast, for the chaotic case with $\tau_2=3.8$, the Poincar\'e return map no longer consists of isolated points but instead forms a scattered cloud distributed over a finite region of the plane (see Figure (\ref{PR38})). This behavior reflects the irregular evolution of the trajectory and its sensitive dependence on initial conditions. The corresponding PSD exhibits a broadband spectrum with several dominant peaks superimposed on a continuous background, indicating that the oscillations contain a wide range of frequencies (see Figure (\ref{PSD38})). These observations agree with the positive largest Lyapunov exponent reported in Table~\ref{tab:lyap}, further confirming the presence of chaotic dynamics.

\begin{figure}[H]
\centering

\begin{subfigure}[b]{0.32\textwidth}
\includegraphics[height=3.5cm]{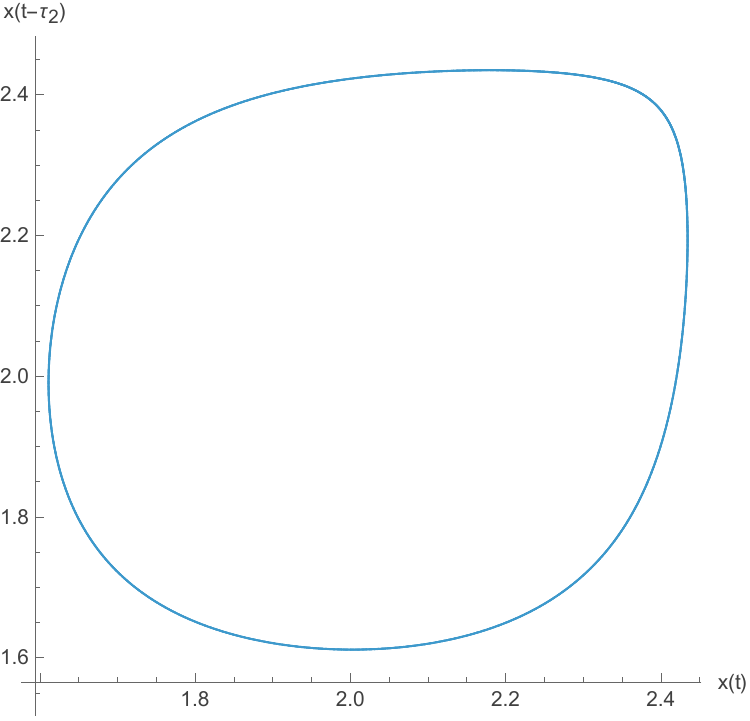}
\caption{Phase portrait ($\tau_2=2.8$)}
\end{subfigure}
\hfill
\begin{subfigure}[b]{0.32\textwidth}
\includegraphics[height=3.5cm]{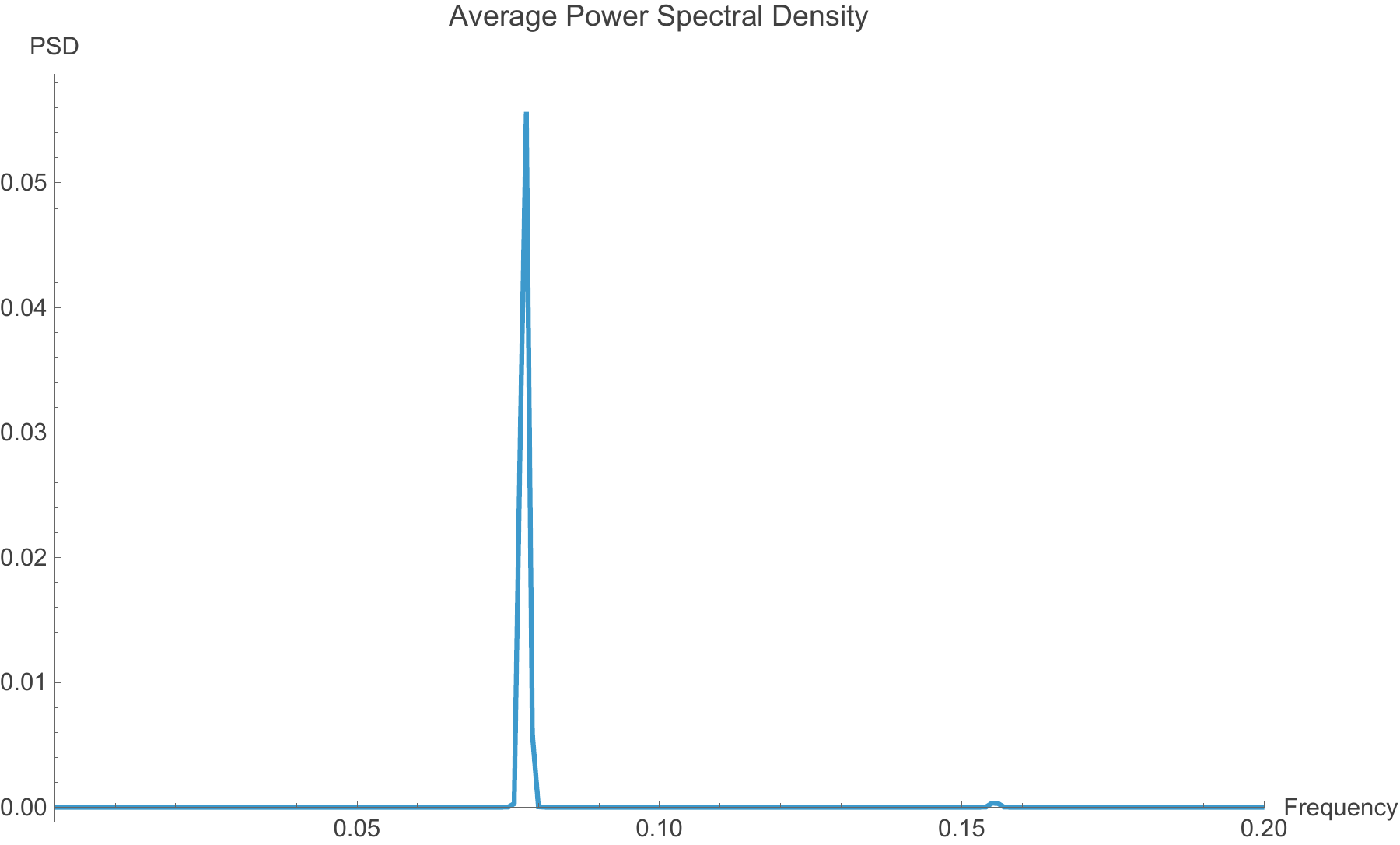}
\caption{PSD ($\tau_2=2.8$)}
\label{PSD28}
\end{subfigure}
\hfill
\begin{subfigure}[b]{0.32\textwidth}
\includegraphics[height=3.5cm]{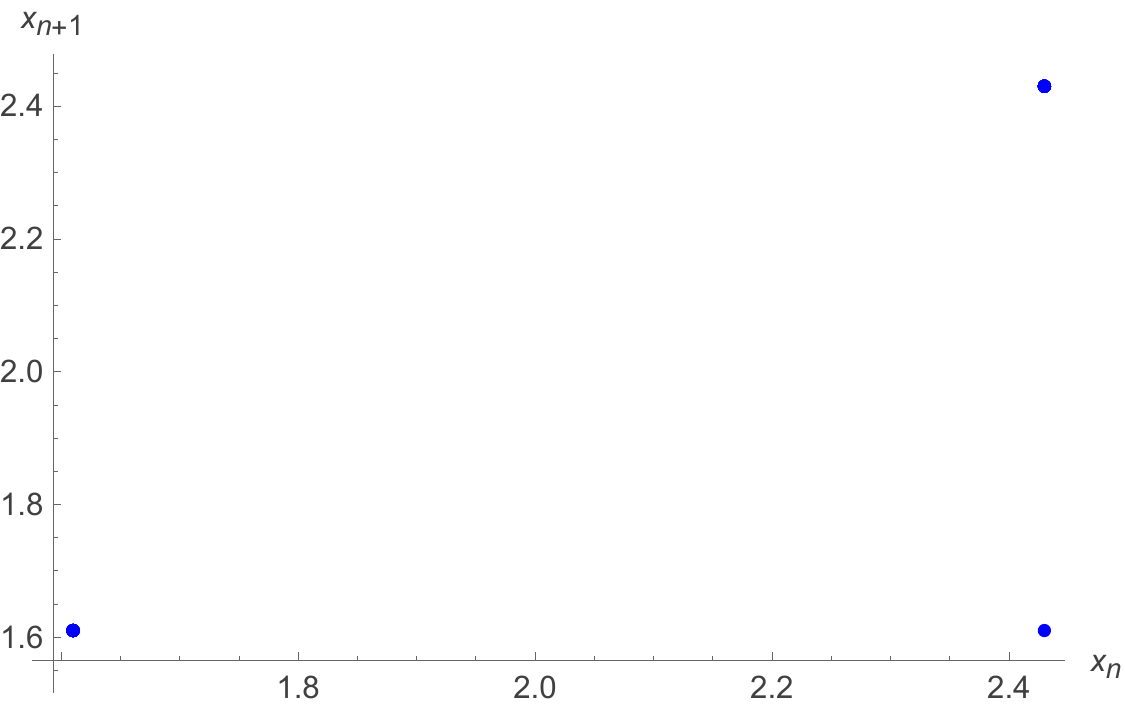}
\caption{Poincar\'e Return map ($\tau_2=2.8$)}
\label{PR28}
\end{subfigure}

\vspace{0.4cm}

\begin{subfigure}[b]{0.32\textwidth}
\includegraphics[height=3.5cm]{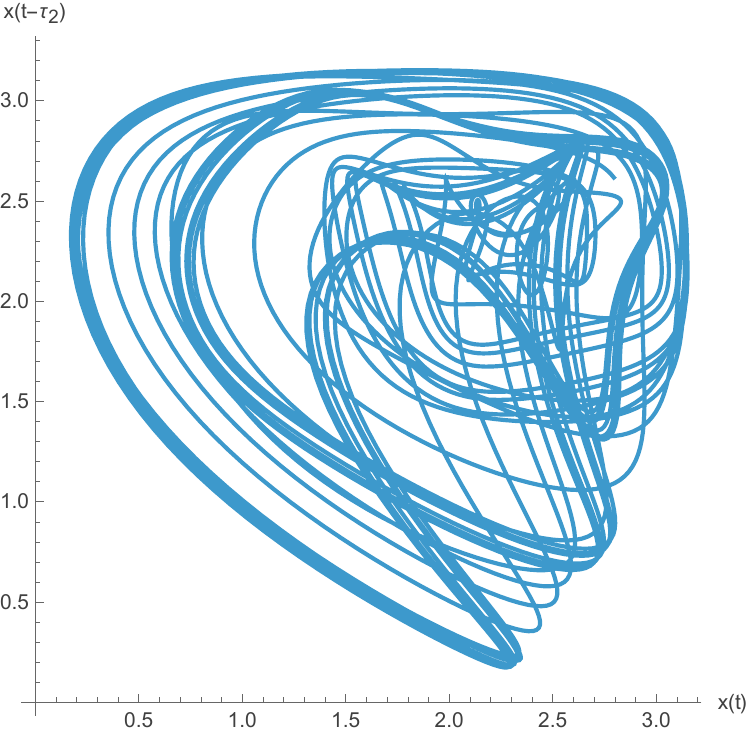}
\caption{Phase portrait ($\tau_2=3.8$)}
\end{subfigure}
\hfill
\begin{subfigure}[b]{0.32\textwidth}
\includegraphics[height=3.5cm]{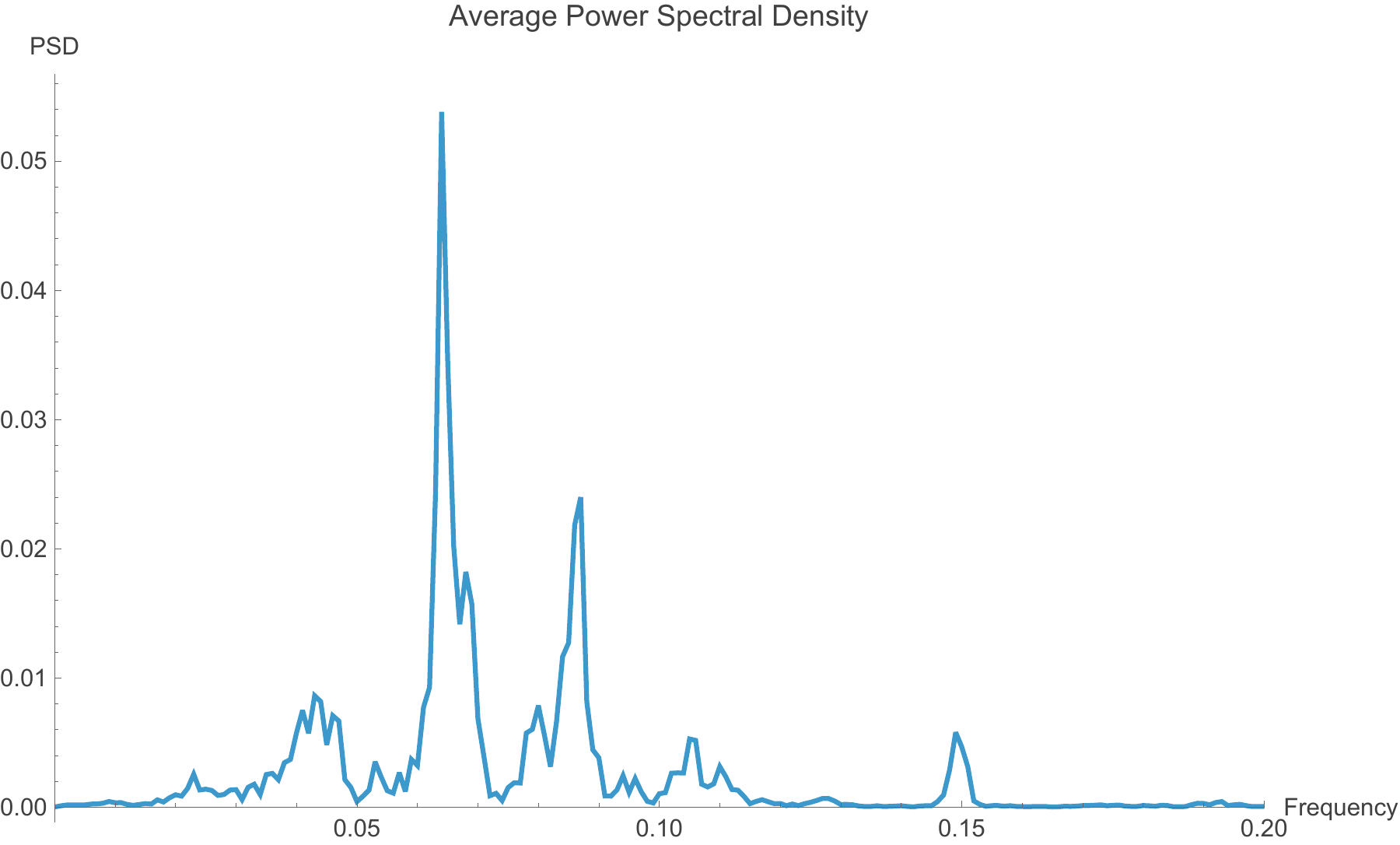}
\caption{PSD ($\tau_2=3.8$)}
\label{PSD38}
\end{subfigure}
\hfill
\begin{subfigure}[b]{0.32\textwidth}
\includegraphics[height=3.5cm]{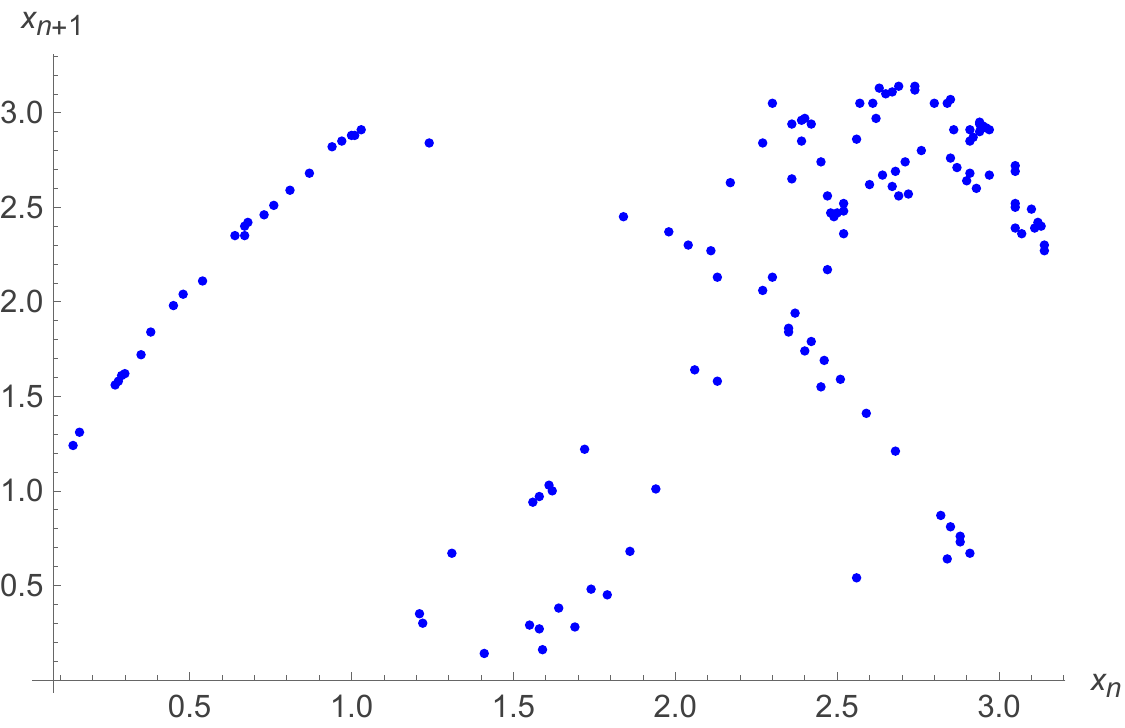}
\caption{Poincar\'e map ($\tau_2=3.8$)}
\label{PR38}
\end{subfigure}

\caption{Comparison of nonlinear diagnostics for periodic and chaotic regimes of system~\eqref{eq:sinx}. The periodic solution at $\tau_2=2.8$ exhibits a discrete line spectrum and isolated points in the Poincar\'e map, whereas the chaotic attractor at $\tau_2=3.8$ is characterized by a broadband power spectrum and a scattered Poincar\'e map.}
\label{fig:comparison}
\end{figure}
}

\subsection{Dynamics for alternative parameter sets}

To demonstrate that the observed delay-induced dynamical behavior is not restricted to a particular choice of parameters, we consider two additional parameter sets and investigate the system dynamics by varying the second delay parameter $\tau_2$.

First, the parameters are chosen as
\[
k=1.5,\qquad \gamma=0.3,\qquad \tau_1=4.5,
\]
while $\tau_2$ is treated as the bifurcation parameter. Similar to the baseline case, the system undergoes a sequence of dynamical transitions from a stable equilibrium to periodic oscillations and eventually to chaotic behavior as $\tau_2$ increases. The corresponding largest Lyapunov exponents, listed in Table~\ref{tab:lyap2}, confirm these transitions. Negative values of the largest Lyapunov exponent correspond to stable and periodic solutions, whereas a positive value indicates the onset of chaos. The associated bifurcation diagram is presented in Figure ~\ref{fig:bif2}, which clearly illustrates the successive transitions between periodic and chaotic dynamics.

\begin{table}[h!]
    \centering
    \caption{Largest Lyapunov exponents for $k=1.5$, $\gamma=0.3$, and $\tau_1=4.5$.}
    \label{tab:lyap2}
    \begin{tabular}{c c c}
        \hline
        $\tau_2$ & $\lambda_{\max}$ & Dynamical behavior \\
        \hline
        0.7 & $-0.224775$ & Stable equilibrium \\
        2.4 & $-0.023890$ & Periodic oscillations \\
        3.0 & $7.517371$ & Strong chaos \\
        \hline
    \end{tabular}
\end{table}
{

As a second example, we consider the parameter set
\[
k=2,\qquad \gamma=0.3,\qquad \tau_1=3.5,
\]
and again vary the delay parameter $\tau_2$. The corresponding bifurcation diagram is shown in Figure \ref{fig:bif3}. Similar to the previous parameter sets, the system exhibits alternating regions of periodic and chaotic behavior as $\tau_2$ is increased. The repeated appearance of chaotic bands interspersed with periodic windows demonstrates that the second delay continues to play a crucial role in shaping the long-term dynamics over a substantially different parameter regime.
\begin{figure}[h!]
    \centering
    \includegraphics[scale=0.4]{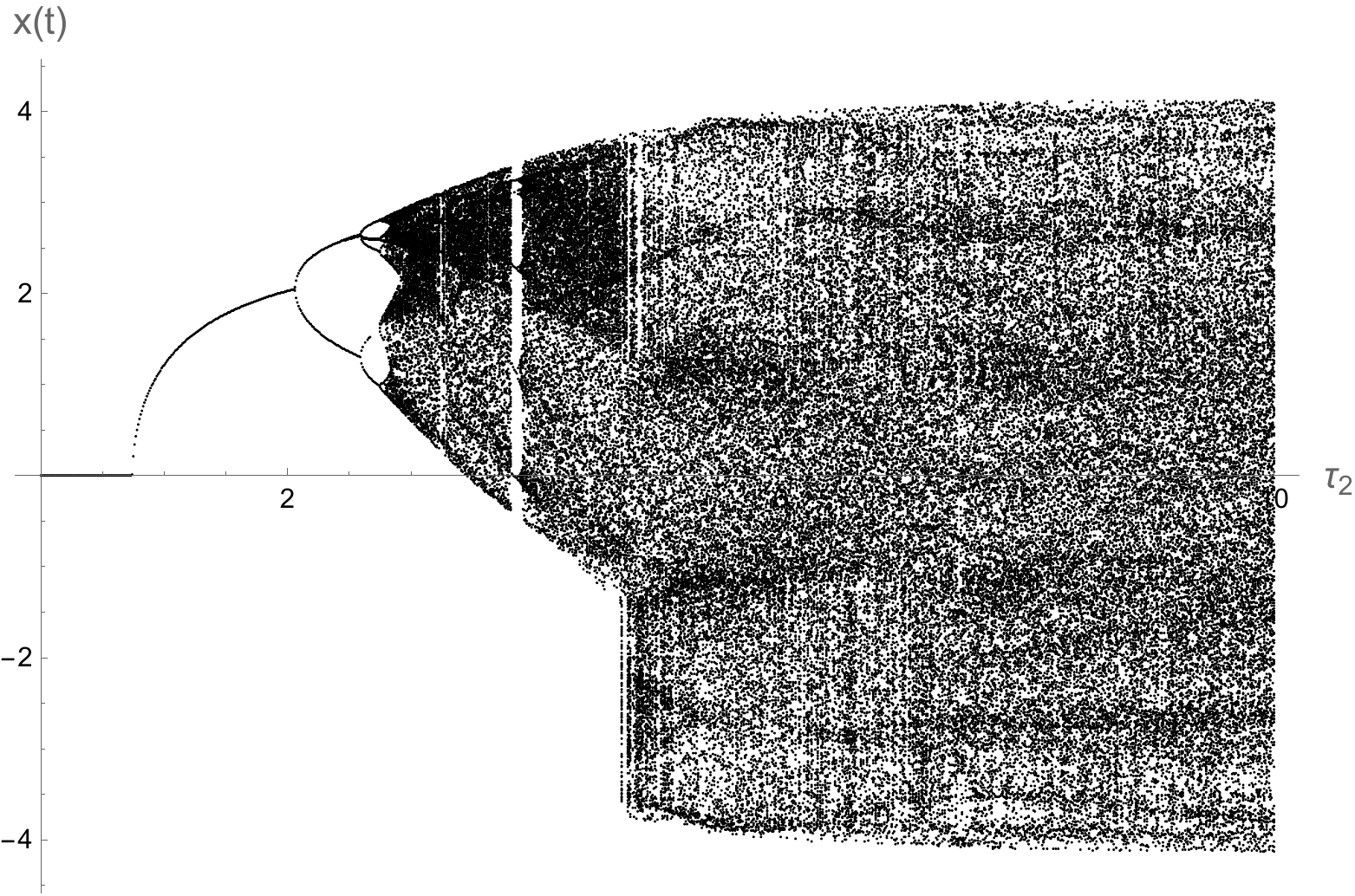}
    \caption{Bifurcation diagram with respect to $\tau_2$ for $k=1.5$, $\gamma=0.3$, and $\tau_1=4.5$.}
    \label{fig:bif2}
\end{figure}

The bifurcation diagrams in Figures ~\ref{fig:bif2} and~\ref{fig:bif3} demonstrate that the qualitative dynamical features observed for the baseline parameter set persist under significant variations of the system parameters. In both cases, the second delay $\tau_2$ acts as an effective bifurcation parameter, generating transitions between stable, periodic, and chaotic dynamics. These results indicate that the delay-induced chaotic behavior is robust and is an intrinsic characteristic of the proposed nonlinear two-delay system rather than an isolated phenomenon associated with a specific parameter choice.
\begin{figure}[h!]
    \centering
    \includegraphics[scale=0.4]{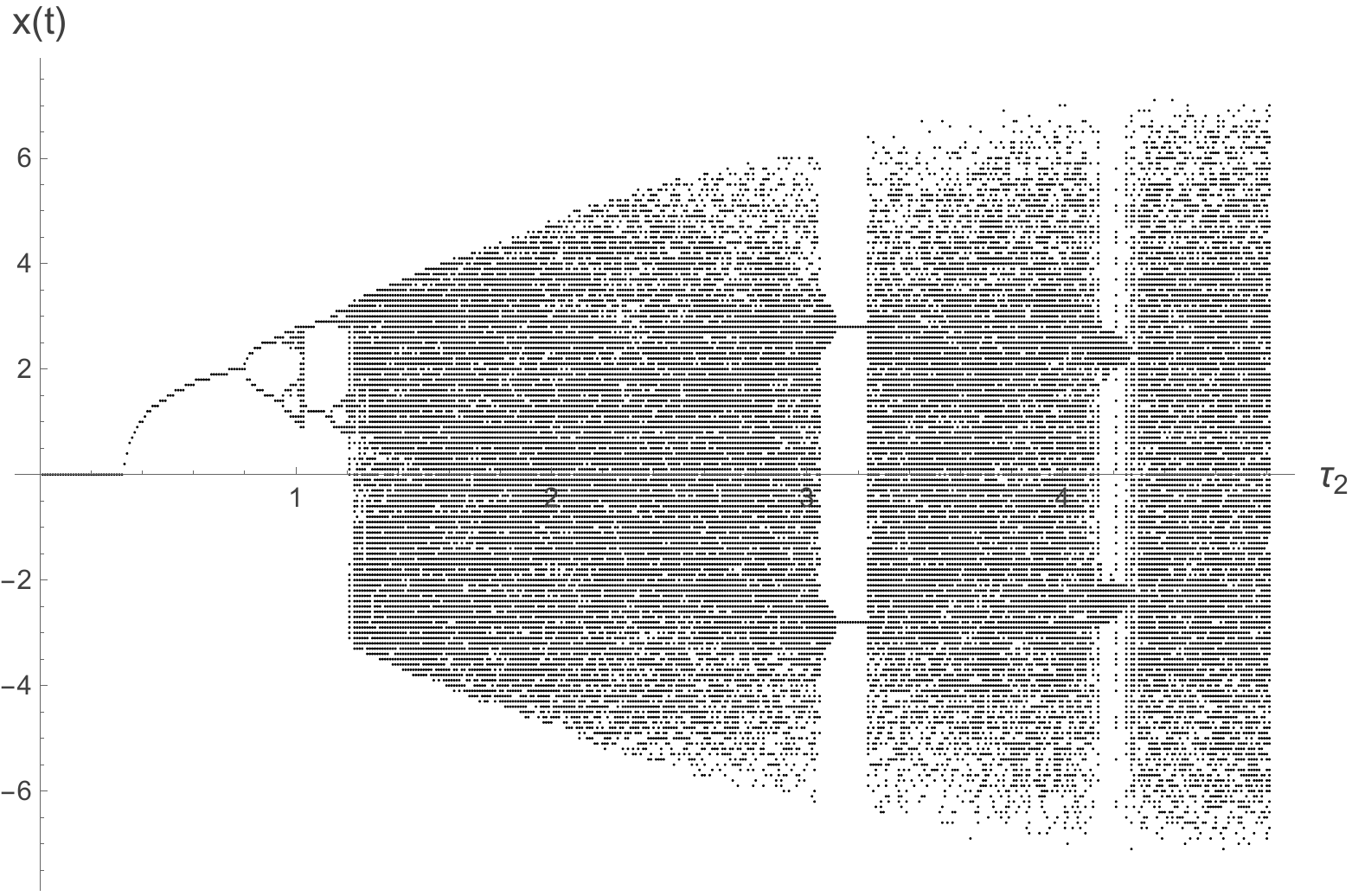}
    \caption{Bifurcation diagram with respect to $\tau_2$ for $k=2$, $\gamma=0.3$, and $\tau_1=3.5$.}
    \label{fig:bif3}
\end{figure}

}

{
\section{Parameter Sensitivity Analysis}
\label{sec:sensitivity}

To study the influence of the system parameters on the dynamics, a one-at-a-time sensitivity analysis is performed. In each case, one parameter is varied while the remaining parameters are fixed at their baseline values,
\[
(k,\gamma,\tau_1,\tau_2)=(1,\,0.1,\,5.1,\,3.8),
\]
for which the system exhibits chaotic behavior.

We first study the effect of the feedback strength $k$. Throughout this study,
$\gamma=0.1$, $\tau_1=5.1$, and $\tau_2=3.8$ are fixed, while $k$ is varied over
$0\le k\le2$.
Figure~\ref{fig:bif_k} shows the bifurcation diagram with respect to the feedback strength $k$. For small values of $k$, the system exhibits stable periodic oscillations. As $k$ increases, the periodic solution loses stability and chaotic dynamics appear near $k\approx0.95$. Around $k\approx1.1$, the system returns to a periodic state, forming a periodic window inside the chaotic region. This periodic behavior persists up to about $k\approx1.6$, after which another transition to chaos occurs. For larger values of $k$, the chaotic attractor expands, indicating increasingly complex dynamics.\\
\begin{figure}[H]
\centering
\includegraphics[width=0.75\textwidth]{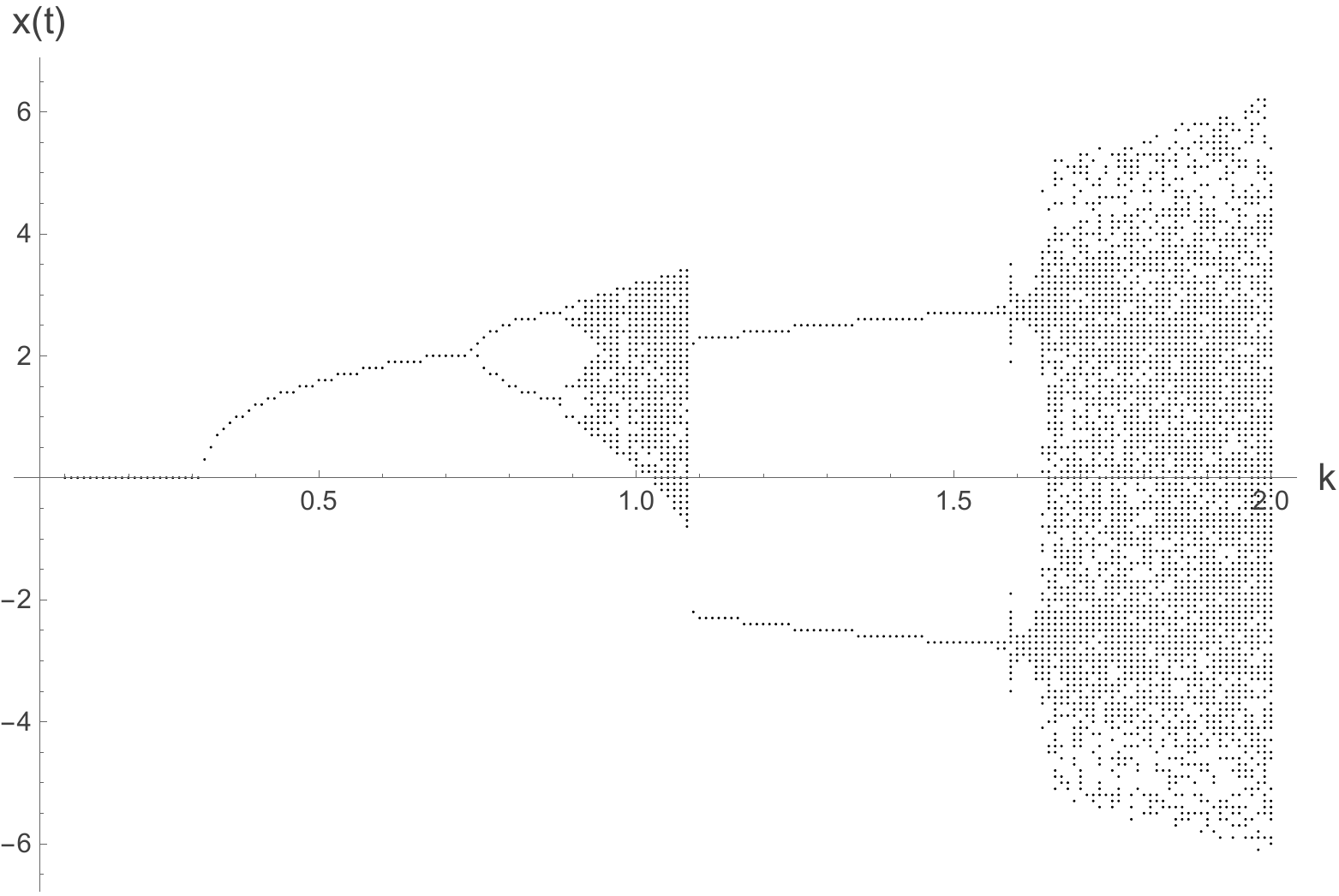}
\caption{Bifurcation diagram obtained by varying  $k$ while keeping $\gamma=0.1$, $\tau_1=5.1$, and $\tau_2=3.8$ fixed.}
\label{fig:bif_k}
\end{figure}

\begin{figure}[H]
\centering

\begin{subfigure}[b]{0.19\textwidth}
\centering
\includegraphics[width=\linewidth]{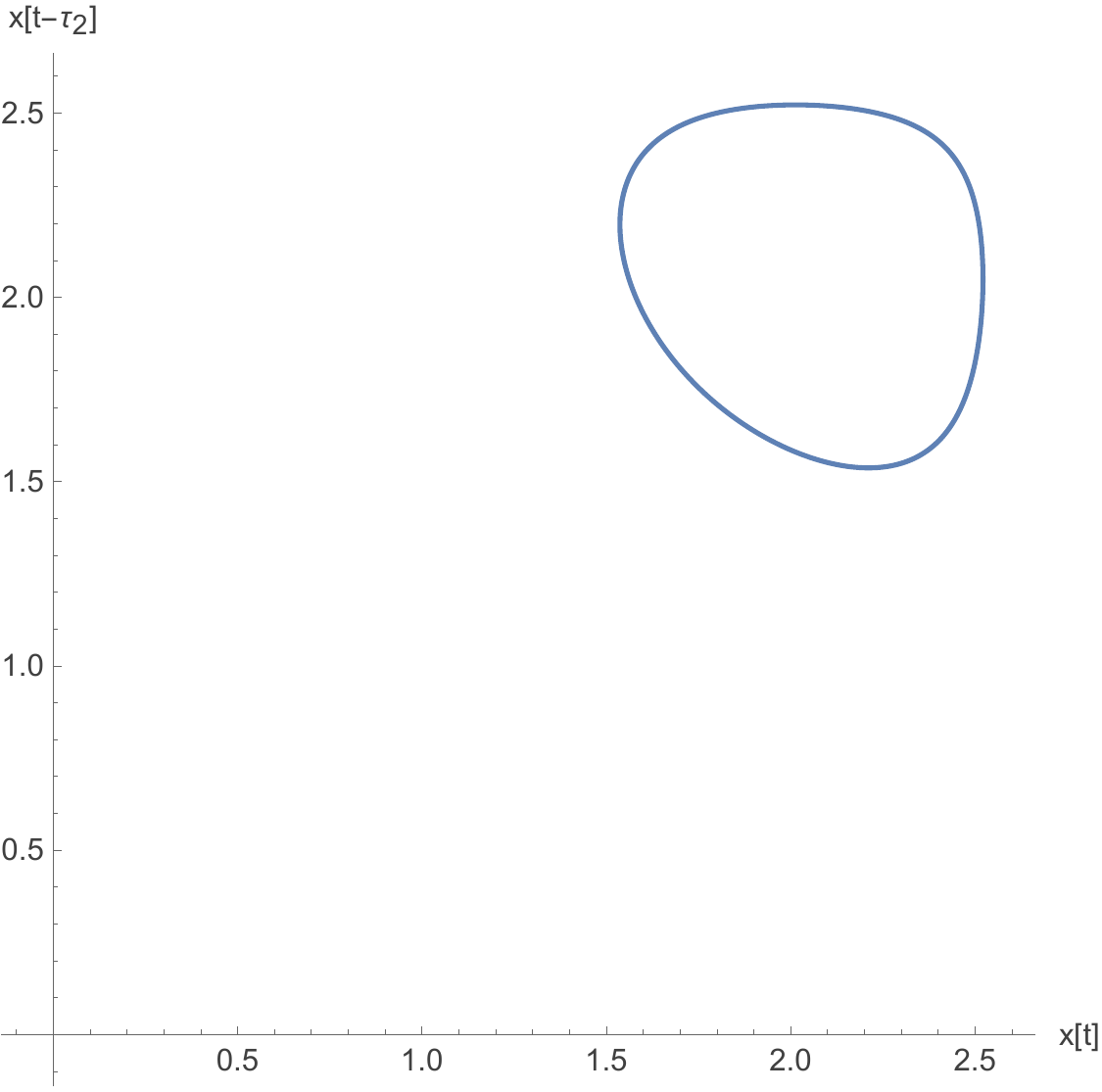}
\caption{$k=0.8$}
\end{subfigure}
\hfill
\begin{subfigure}[b]{0.19\textwidth}
\centering
\includegraphics[width=\linewidth]{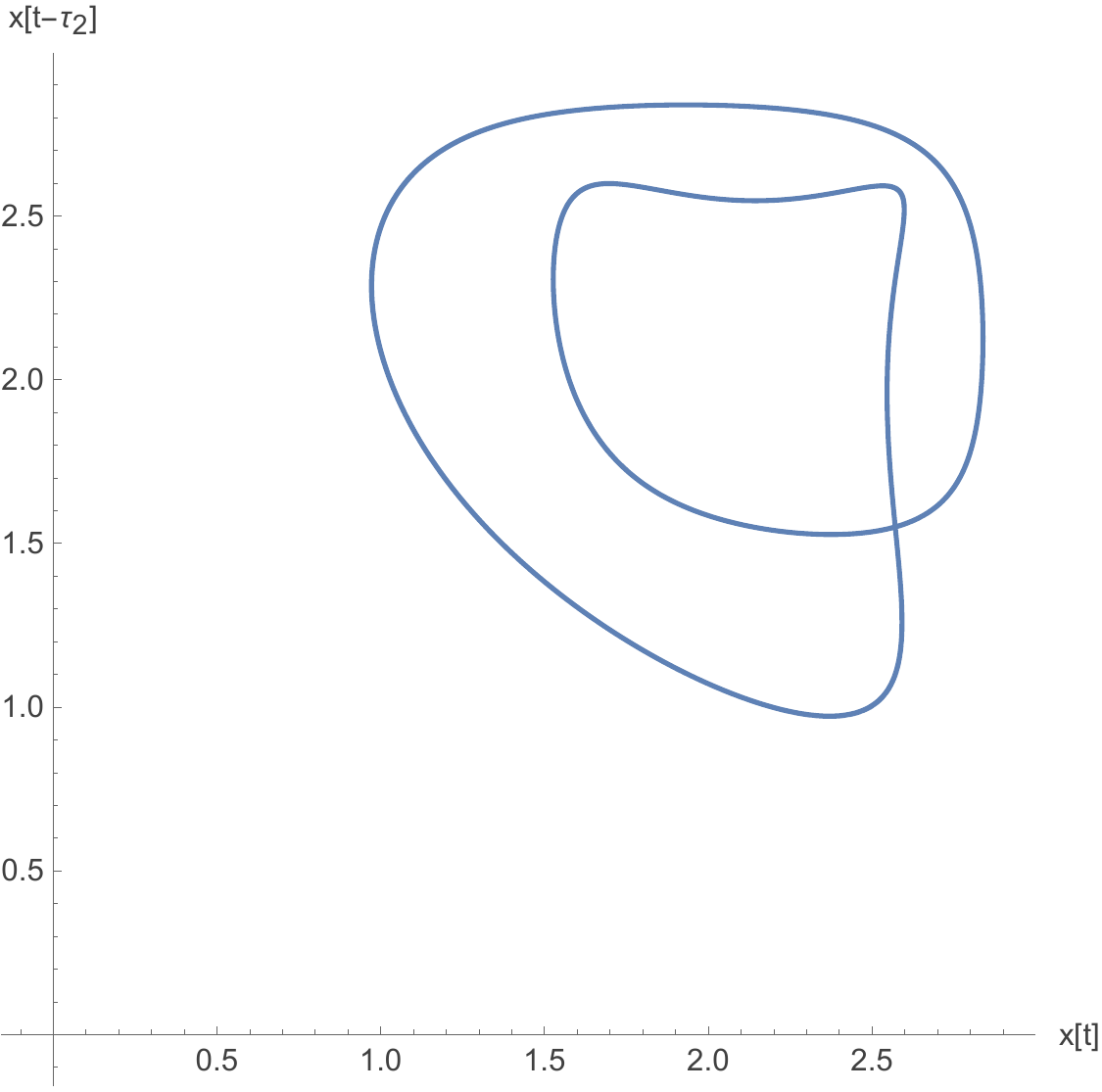}
\caption{$k=0.9$}
\end{subfigure}
\hfill
\begin{subfigure}[b]{0.19\textwidth}
\centering
\includegraphics[width=\linewidth]{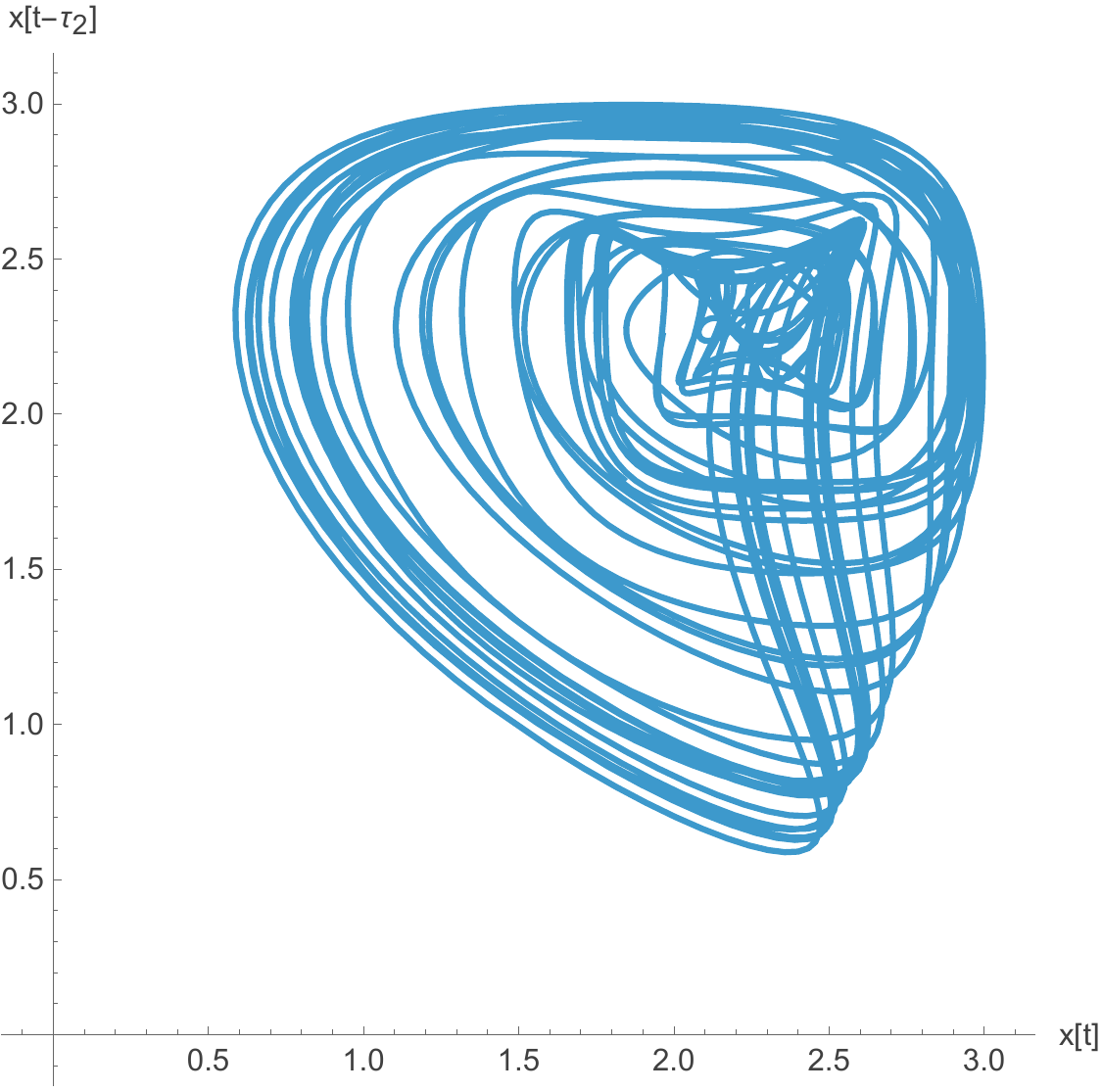}
\caption{$k=0.95$}
\end{subfigure}
\hfill
\begin{subfigure}[b]{0.19\textwidth}
\centering
\includegraphics[width=\linewidth]{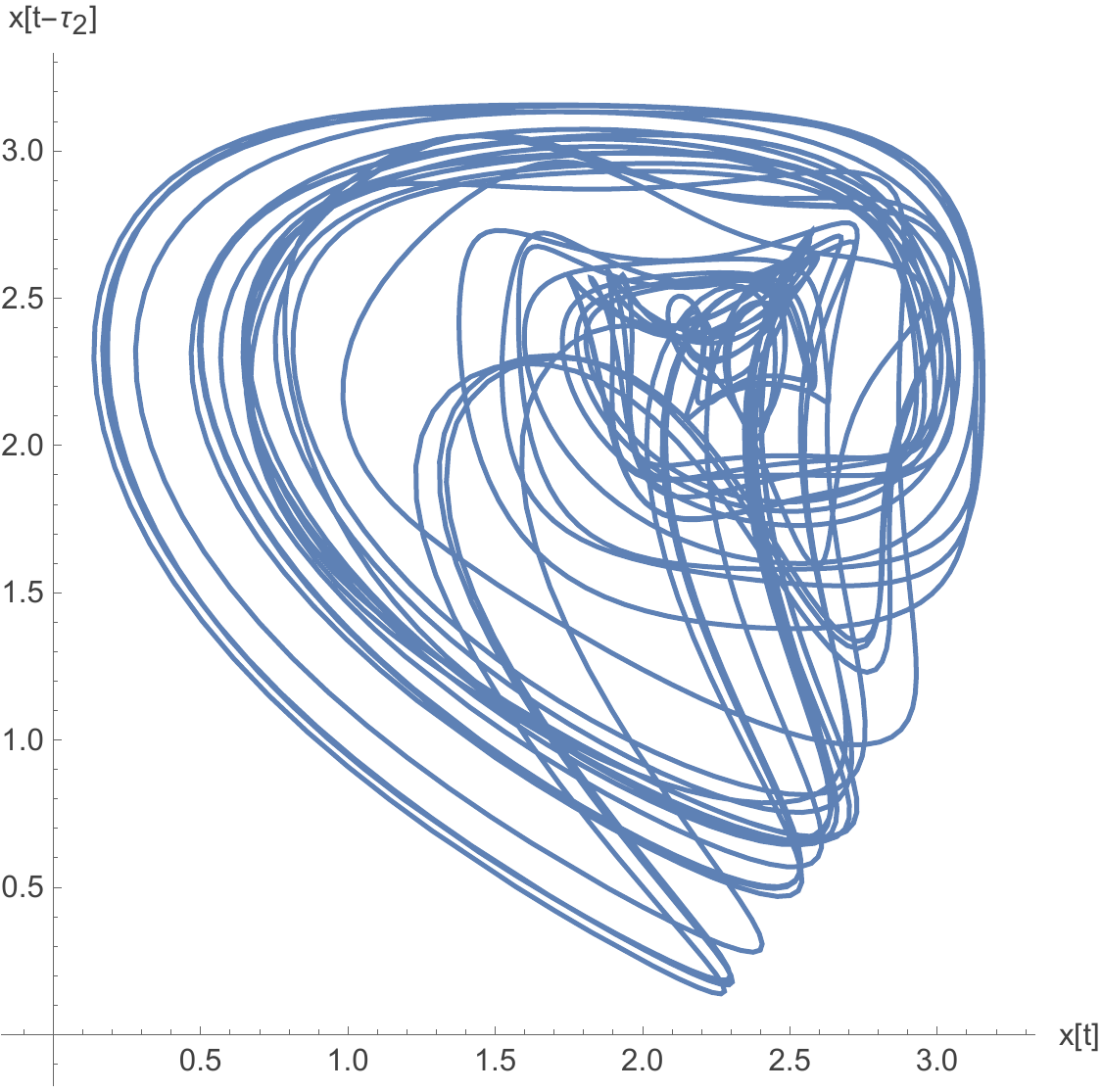}
\caption{$k=1.0$}
\end{subfigure}
\hfill
\begin{subfigure}[b]{0.19\textwidth}
\centering
\includegraphics[width=\linewidth]{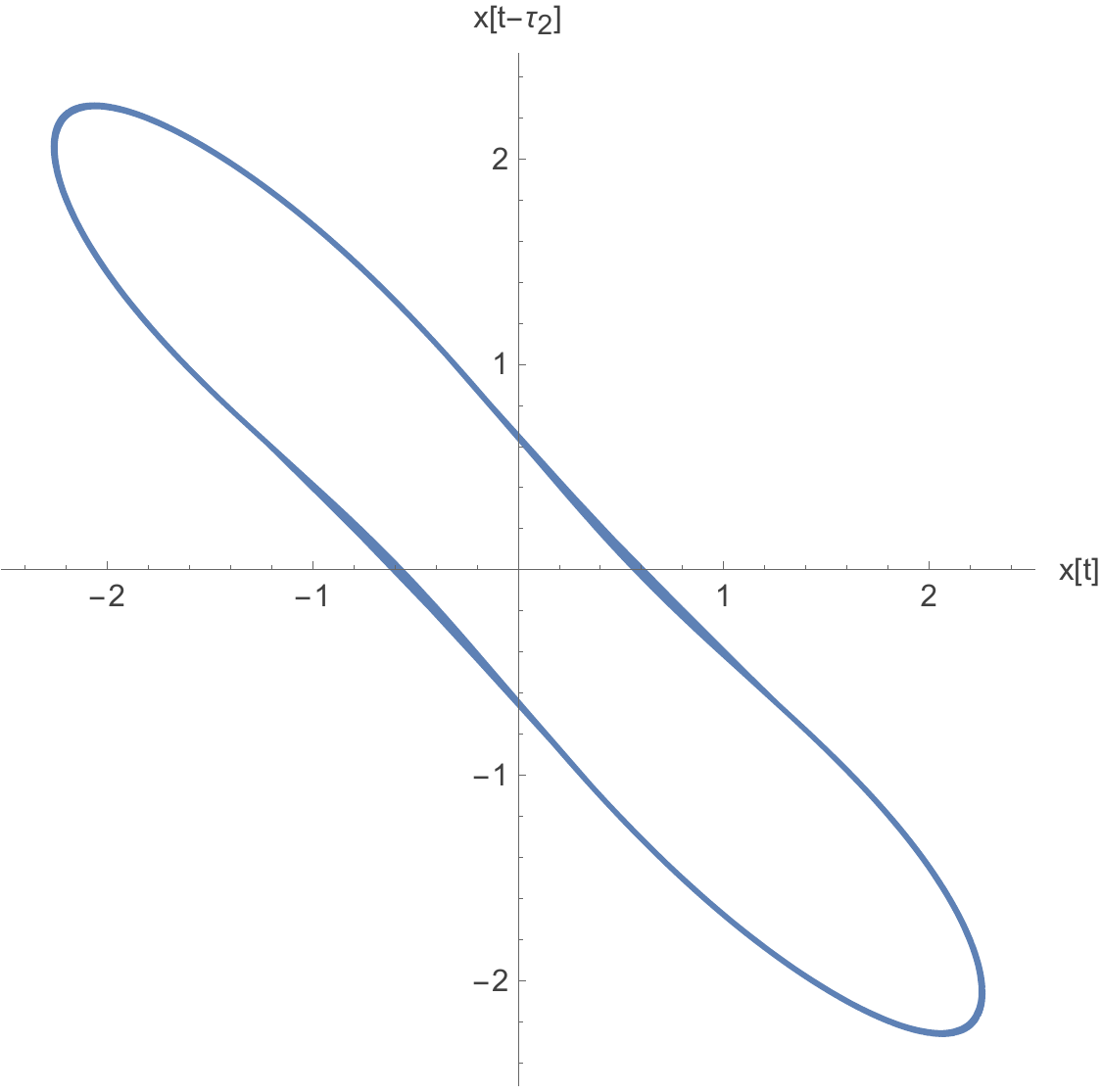}
\caption{$k=1.1$}
\end{subfigure}

\caption{Representative phase portraits illustrating the transition from periodic motion to chaos and the subsequent reappearance of a periodic attractor as  $k$ increases.}
\label{fig:k_sensitivity1}
\end{figure}
\begin{figure}[H]
\centering

\begin{subfigure}[b]{0.19\textwidth}
\centering
\includegraphics[width=\linewidth]{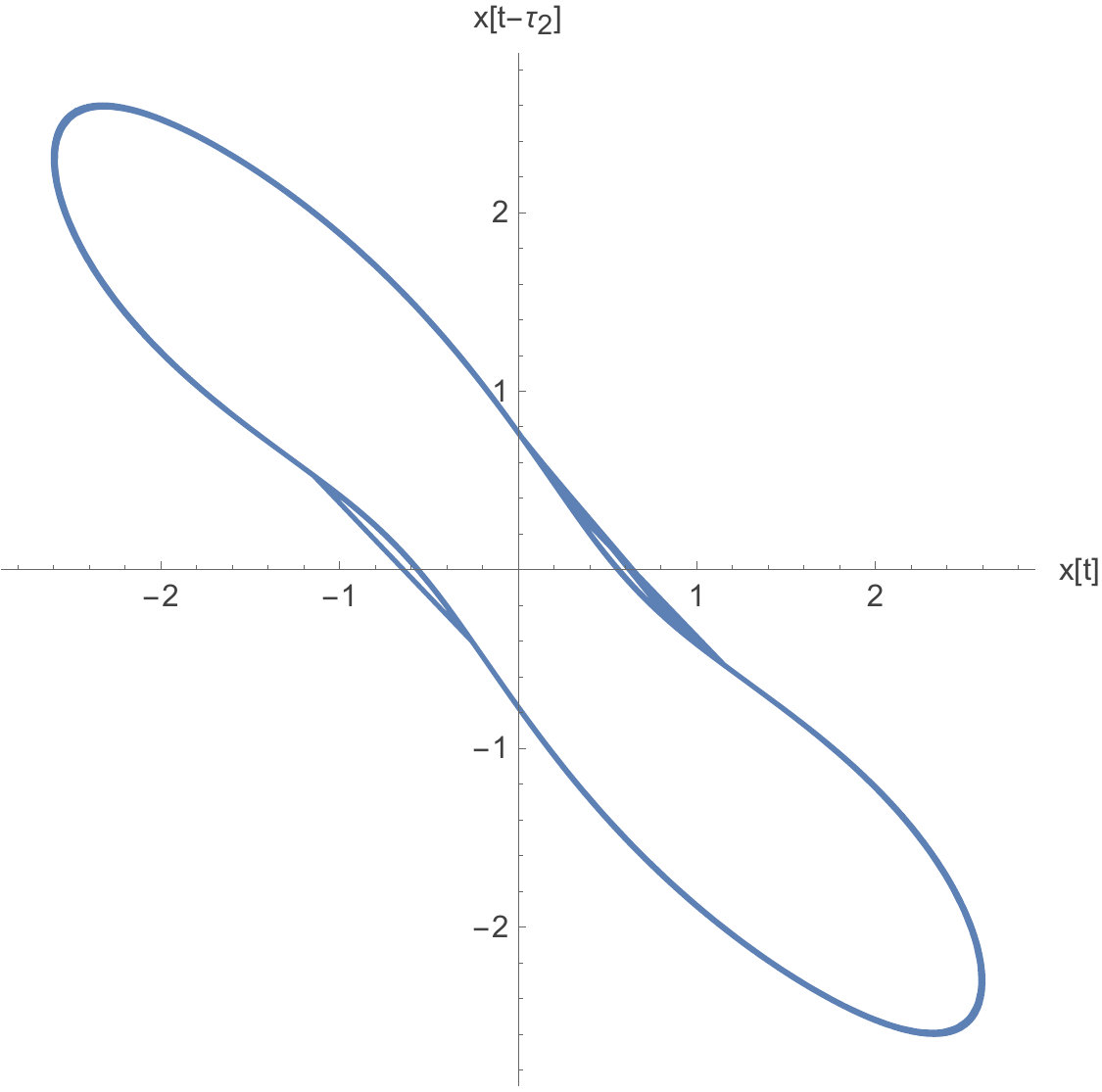}
\caption{$k=1.4$}
\end{subfigure}
\hfill
\begin{subfigure}[b]{0.19\textwidth}
\centering
\includegraphics[width=\linewidth]{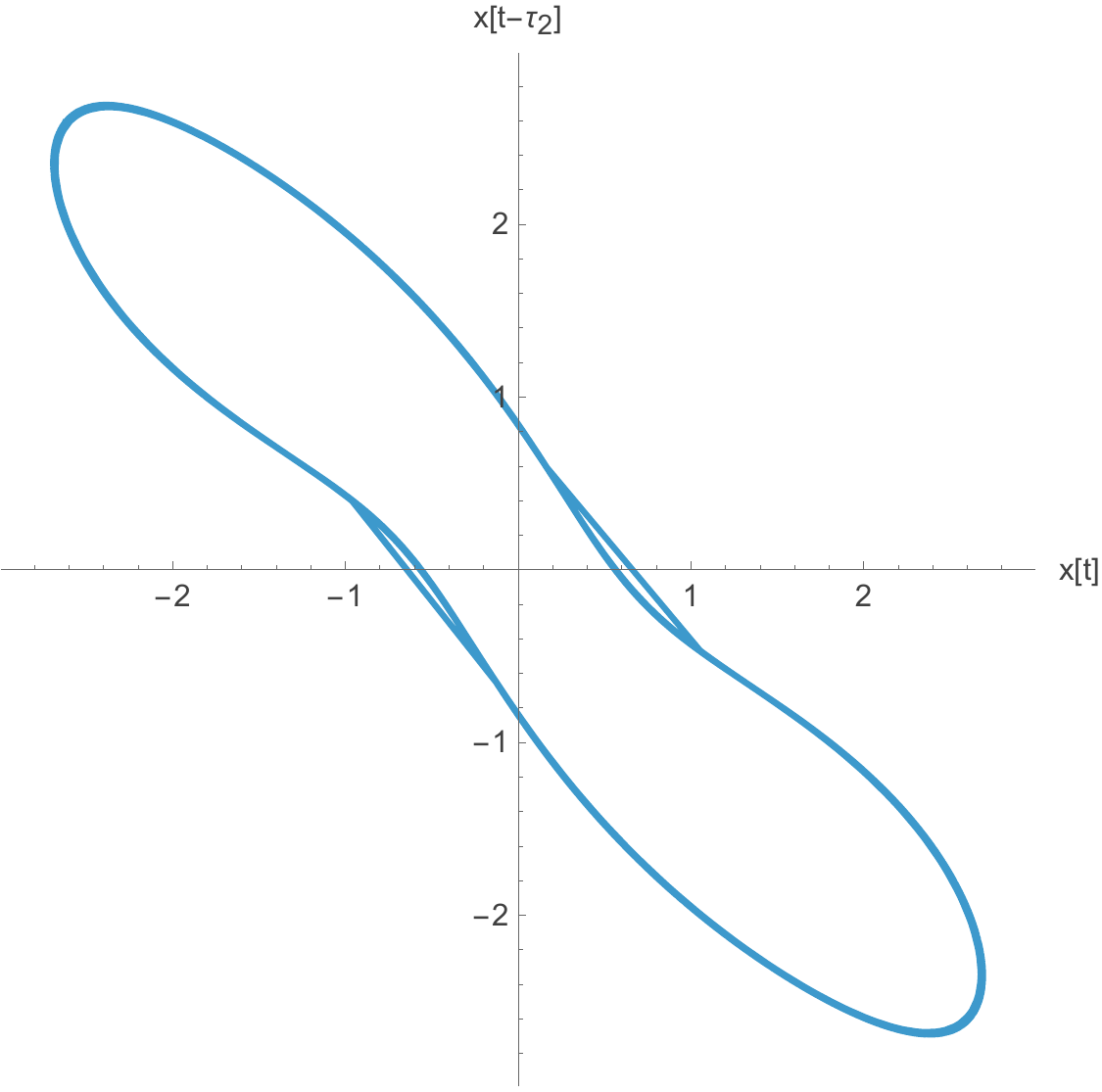}
\caption{$k=1.5$}
\end{subfigure}
\hfill
\begin{subfigure}[b]{0.19\textwidth}
\centering
\includegraphics[width=\linewidth]{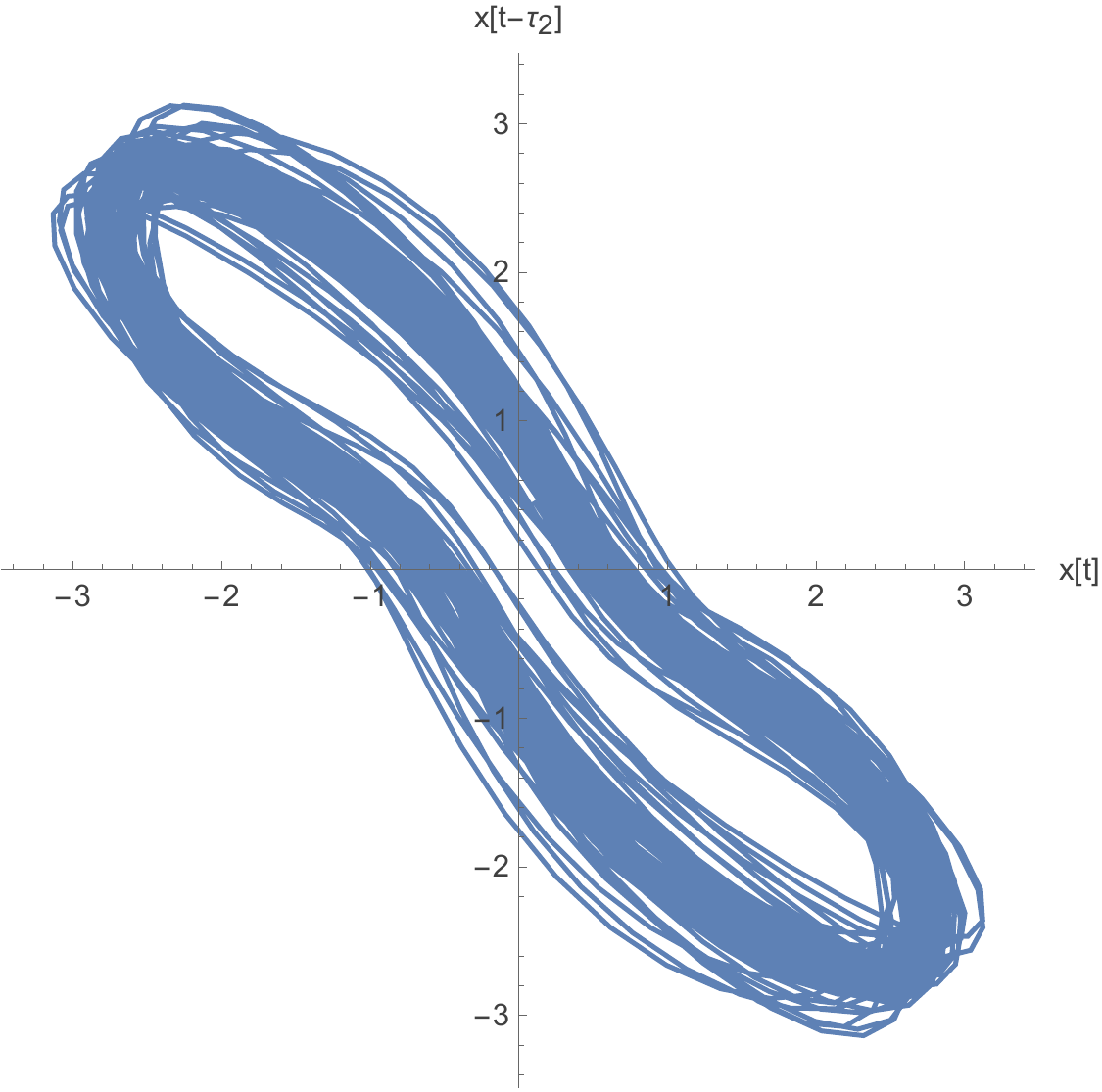}
\caption{$k=1.6$}
\end{subfigure}
\hfill
\begin{subfigure}[b]{0.19\textwidth}
\centering\
\includegraphics[width=\linewidth]{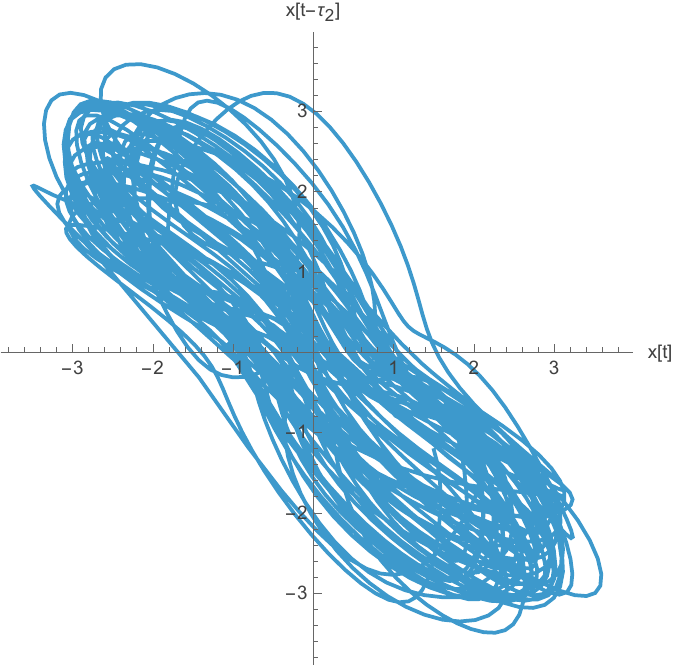}
\caption{$k=1.65$}
\end{subfigure}
\hfill
\begin{subfigure}[b]{0.19\textwidth}
\centering
\includegraphics[width=\linewidth]{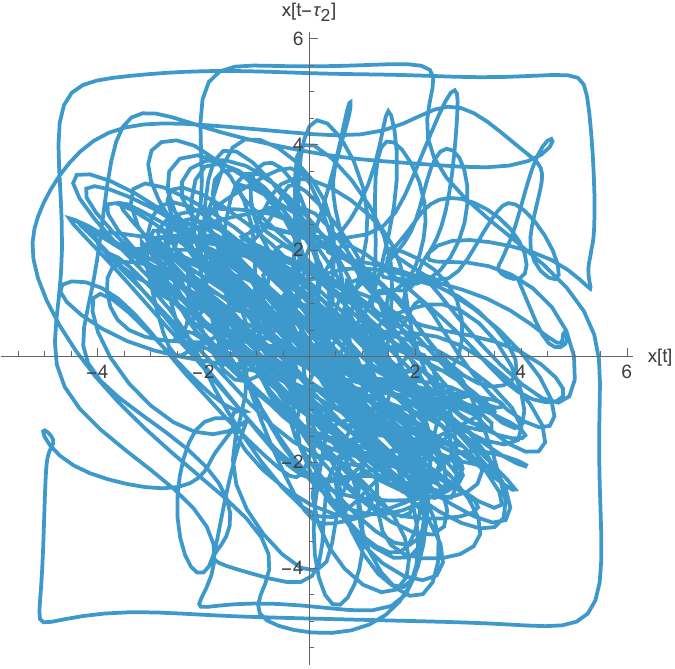}
\caption{$k=1.8$}
\end{subfigure}

\caption{Representative phase portraits illustrating the second transition from periodic to chaotic dynamics for larger values of $k$.}
\label{fig:k_sensitivity2}
\end{figure}

Representative phase portraits for different values of $k$ are shown in Figures~\ref{fig:k_sensitivity1} and~\ref{fig:k_sensitivity2}. Figure~\ref{fig:k_sensitivity1} agrees with the bifurcation diagram. The system is periodic for $k=0.8$ and $k=0.9$. Chaos appears near $k=0.95$ and becomes fully developed at the baseline value $k=1$. Increasing $k$ to $1.1$ restores periodic oscillations, confirming the existence of a periodic window.\\

Figure~\ref{fig:k_sensitivity2} shows the dynamics beyond the periodic window. The system remains periodic for $k=1.4$ and $k=1.5$. Near $k=1.6$, the periodic orbit loses stability and chaos reappears. For $k=1.65$ and $k=1.8$, the attractor becomes more irregular and spreads over a larger region of the phase space.\\
Next, we examine the effect of the damping coefficient $\gamma$. Here, $k=1$, $\tau_1=5.1$, and $\tau_2=3.8$ are fixed while $\gamma$ varies over $(0,0.2)$.
Figure~\ref{bifg} shows the bifurcation diagram with respect to $\gamma$. For small values of $\gamma$, the system exhibits periodic oscillations. Near $\gamma\approx0.055$, the periodic orbit loses stability and chaos appears. As $\gamma$ increases further, the chaotic region gradually shrinks, indicating weaker chaotic oscillations. Finally, for sufficiently large values of $\gamma$, the system returns to periodic behavior, showing that increased damping suppresses chaos.\\
\begin{figure}[H]
    \centering
    \includegraphics[width=0.75\textwidth]{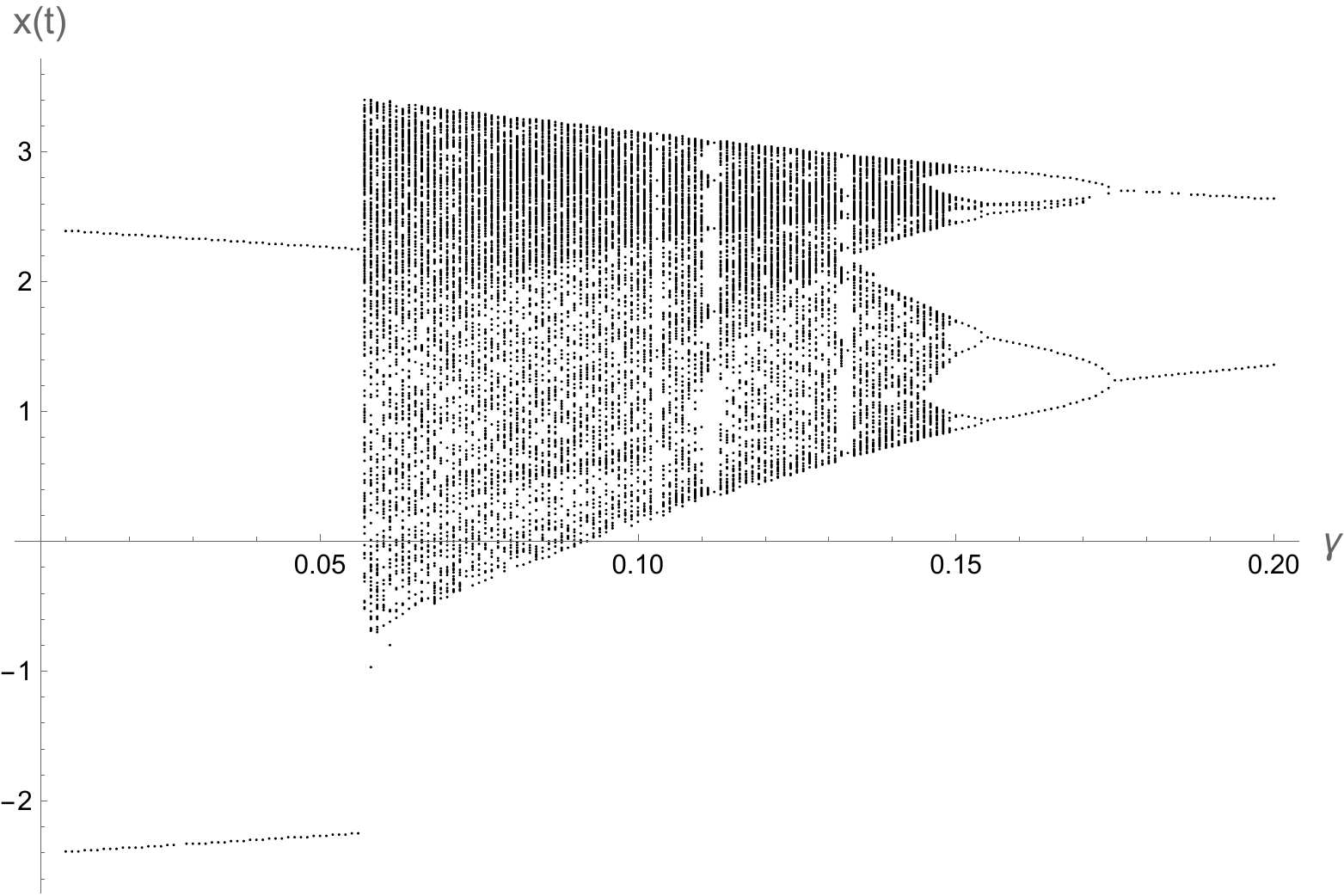}
    \caption{Bifurcation diagram obtained by varying  $\gamma$ while keeping $k=1$, $\tau_1=5.1$, and $\tau_2=3.8$ fixed.}
    \label{bifg}
\end{figure}
The phase portraits in Figure~\ref{fig:gamma_sensitivity} confirm the bifurcation diagram. The system is periodic for $\gamma=0$ and $0.05$. Chaos first appears near $\gamma=0.055$ and becomes fully developed at $\gamma=0.1$. Increasing $\gamma$ further reduces the size of the chaotic attractor, and by $\gamma=0.2$ the system again exhibits periodic oscillations.
\begin{figure}[H]
\centering

\begin{subfigure}[b]{0.19\textwidth}
    \centering
    \includegraphics[width=\linewidth]{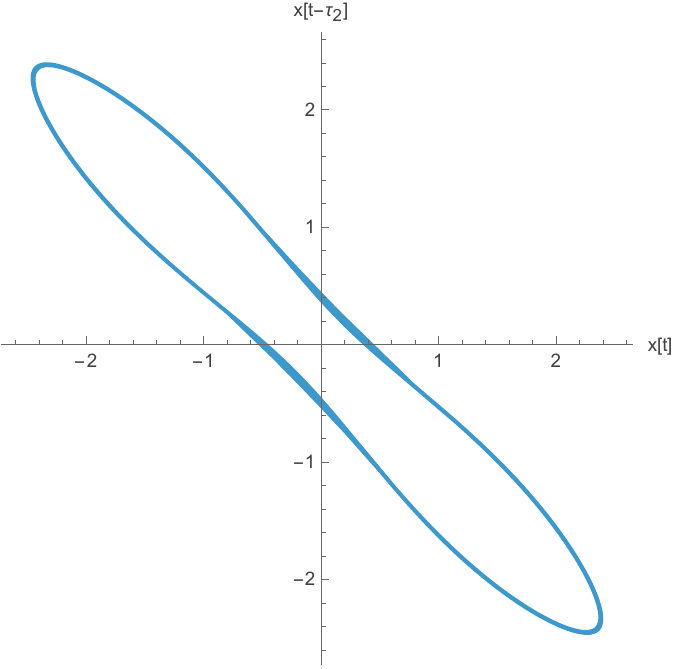}
    \caption{$\gamma=0$}
\end{subfigure}
\hfill
\begin{subfigure}[b]{0.19\textwidth}
    \centering
    \includegraphics[width=\linewidth]{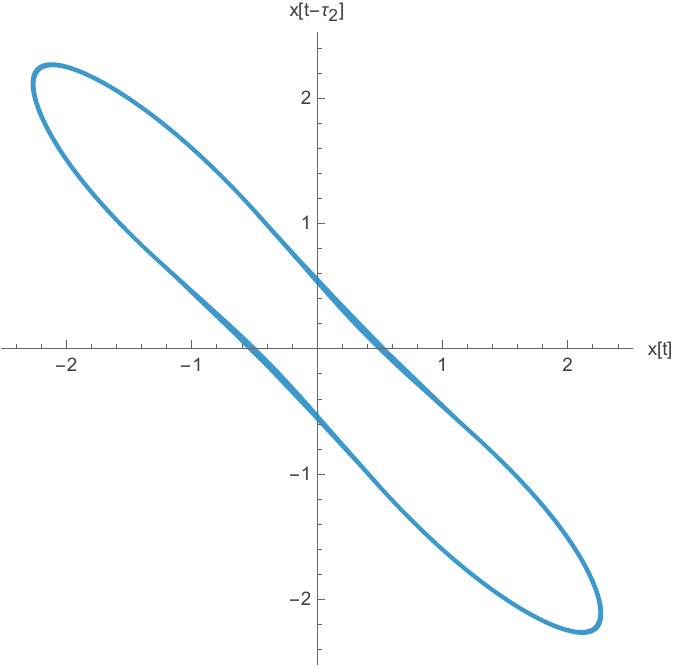}
    \caption{$\gamma=0.05$}
\end{subfigure}
\hfill
\begin{subfigure}[b]{0.19\textwidth}
    \centering
    \includegraphics[width=\linewidth]{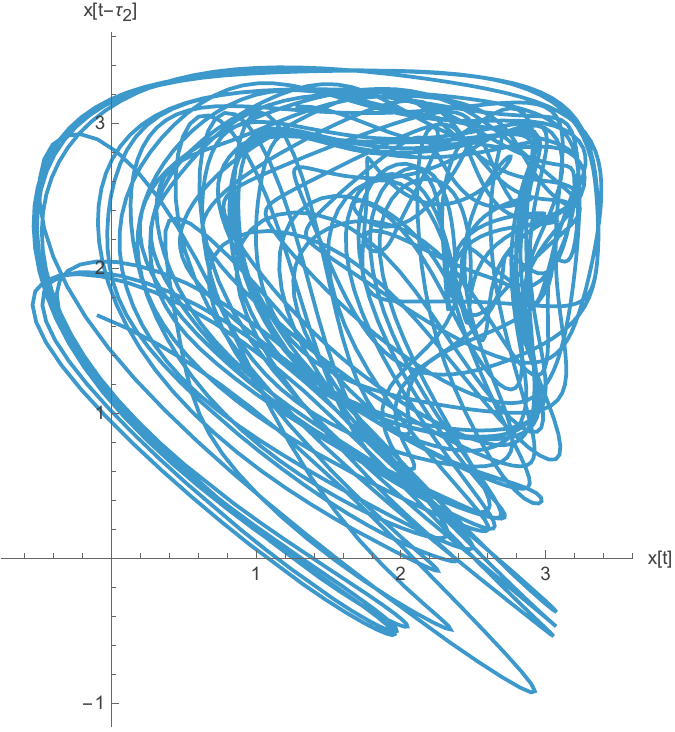}
    \caption{$\gamma=0.055$}
\end{subfigure}
\hfill
\begin{subfigure}[b]{0.19\textwidth}
    \centering
    \includegraphics[width=\linewidth]{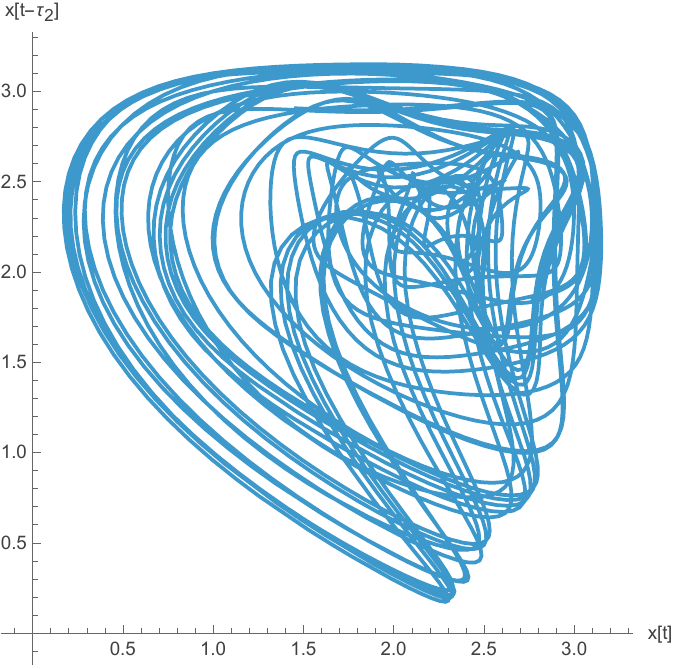}
    \caption{$\gamma=0.10$}
\end{subfigure}
\hfill
\begin{subfigure}[b]{0.19\textwidth}
    \centering
    \includegraphics[width=\linewidth]{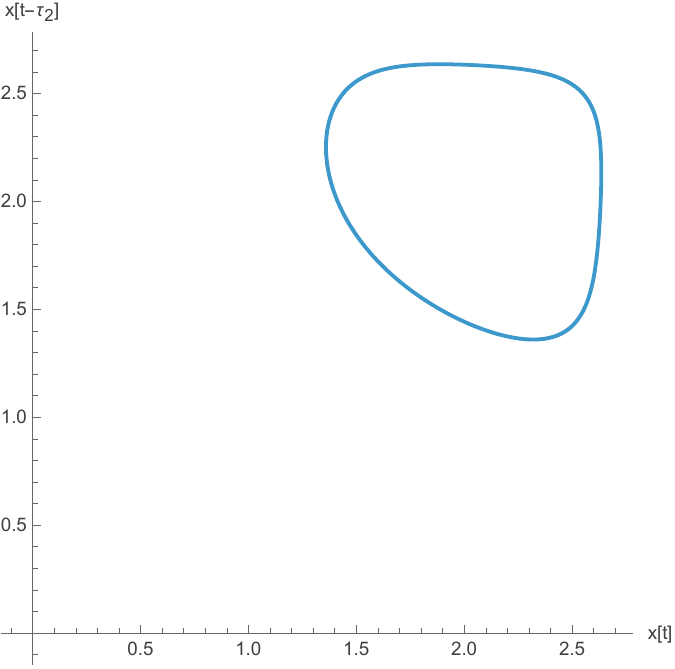}
    \caption{$\gamma=0.20$}
\end{subfigure}

\caption{Representative phase portraits illustrating the effect of $\gamma$ on the system dynamics for fixed $k=1$, $\tau_1=5.1$, and $\tau_2=3.8$.}
\label{fig:gamma_sensitivity}
\end{figure}

Finally, we study the effect of the first delay $\tau_1$. Throughout this subsection, $k=1$, $\gamma=0.1$, and $\tau_2=3.8$ are fixed, while $\tau_1$ varies over $3\le\tau_1\le7$.
Figure~\ref{fig:bif_tau1} shows the bifurcation diagram with respect to $\tau_1$. For smaller values of $\tau_1$, the system exhibits periodic oscillations. As $\tau_1$ increases, the periodic solution loses stability and chaotic dynamics appear. Several narrow periodic windows are embedded within the chaotic region, indicating repeated transitions between periodic and chaotic behavior. For larger values of $\tau_1$, chaos remains dominant, although isolated periodic windows continue to occur.\\
\begin{figure}[H]
\centering
\includegraphics[width=0.75\textwidth]{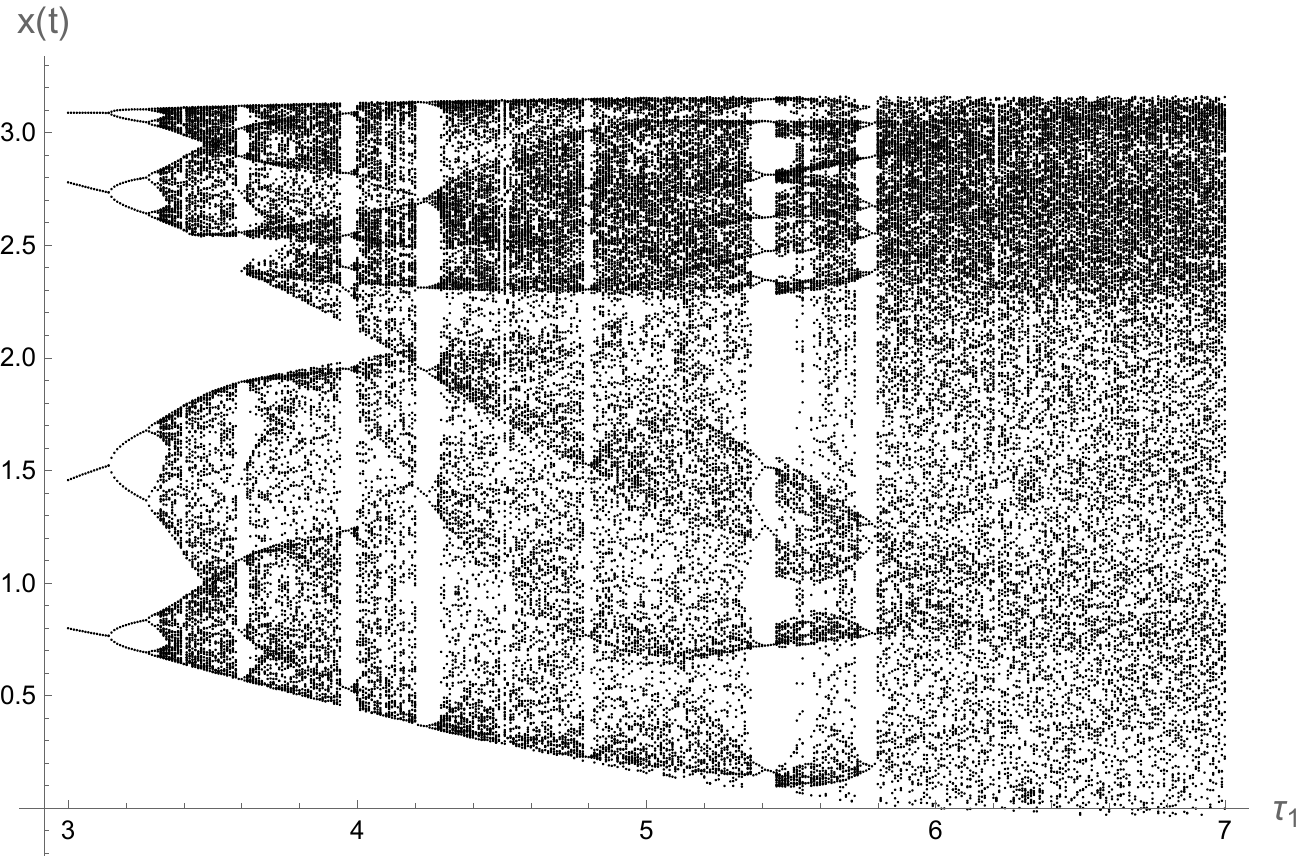}
\caption{Bifurcation diagram obtained by varying the delay parameter $\tau_1$ while keeping $k=1$, $\gamma=0.1$, and $\tau_2=3.8$ fixed.}
\label{fig:bif_tau1}
\end{figure}

The phase portraits in Figure~\ref{fig:tau1_sensitivity} are consistent with the bifurcation diagram. Periodic oscillations are observed for smaller values of $\tau_1$ (Figure (\ref{tau13})), while increasing $\tau_1$ leads to chaotic motion (see Figure (\ref{tau139})). Further changes in $\tau_1$ produce alternating periodic and chaotic attractors, corresponding to the periodic windows seen in the bifurcation diagram (Figures ~\ref{fig:tau1_sensitivity}\subref{tau1425}--
\ref{fig:tau1_sensitivity}\subref{tau154}). For larger values of $\tau_1$, the attractor again becomes chaotic (Figure (\ref{tau165})).

\begin{figure}[H]
\centering

\begin{subfigure}[b]{0.25\textwidth}
    \centering
    \includegraphics[width=\linewidth]{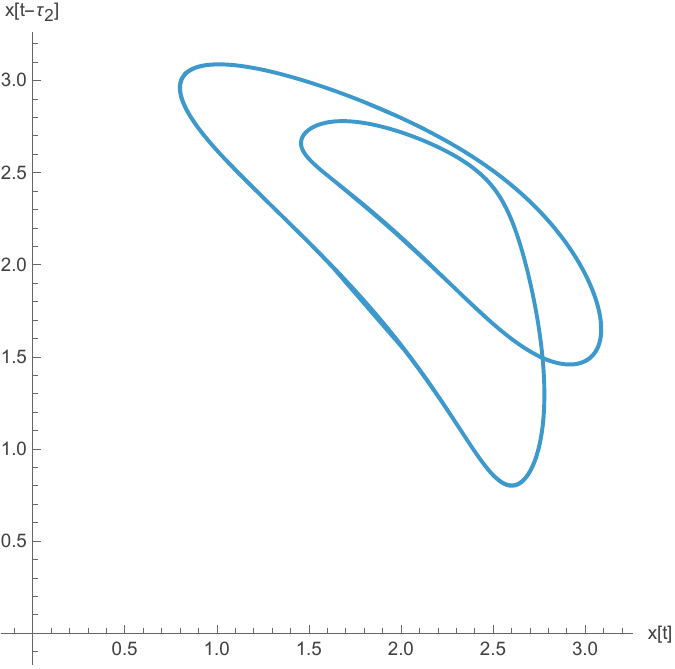}
    \caption{$\tau_1=3$}
	\label{tau13}
\end{subfigure}
\hfill
\begin{subfigure}[b]{0.25\textwidth}
    \centering
    \includegraphics[width=\linewidth]{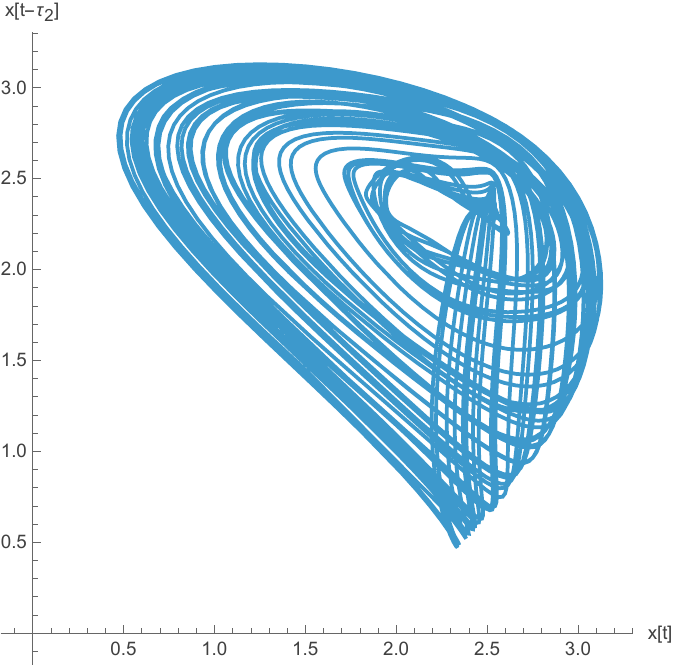}
    \caption{$\tau_1=3.9$}
	\label{tau139}
\end{subfigure}
\hfill
\begin{subfigure}[b]{0.25\textwidth}
    \centering
    \includegraphics[width=\linewidth]{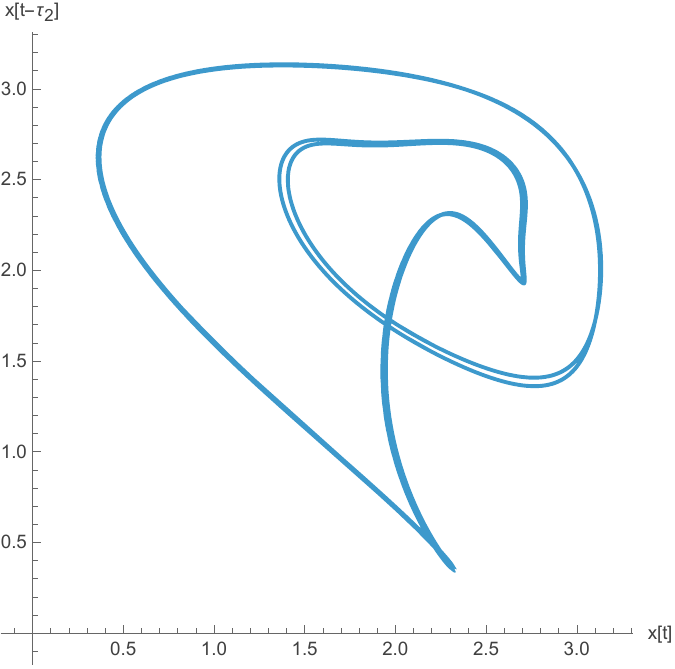}
    \caption{$\tau_1=4.25$}
	\label{tau1425}
\end{subfigure}
\hfill
\begin{subfigure}[b]{0.25\textwidth}
    \centering
    \includegraphics[width=\linewidth]{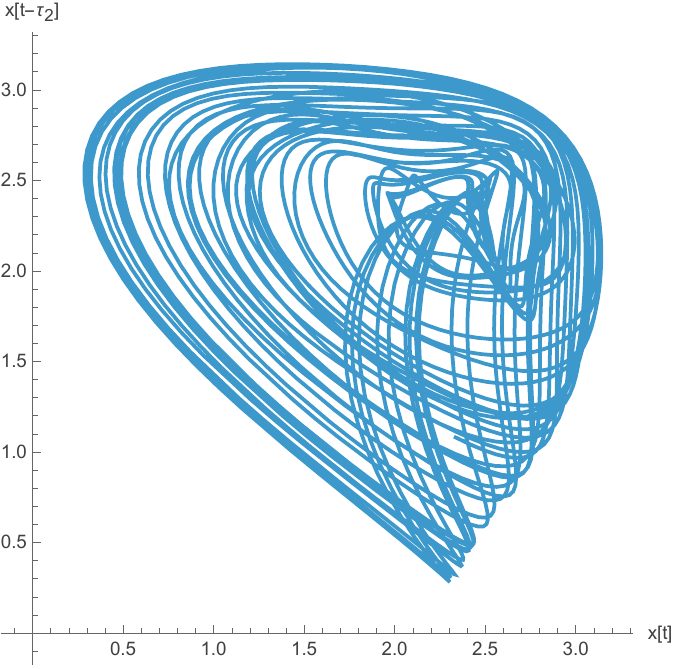}
    \caption{$\tau_1=4.5$}
	\label{tau145}
\end{subfigure}
\hfill
\begin{subfigure}[b]{0.25\textwidth}
    \centering
    \includegraphics[width=\linewidth]{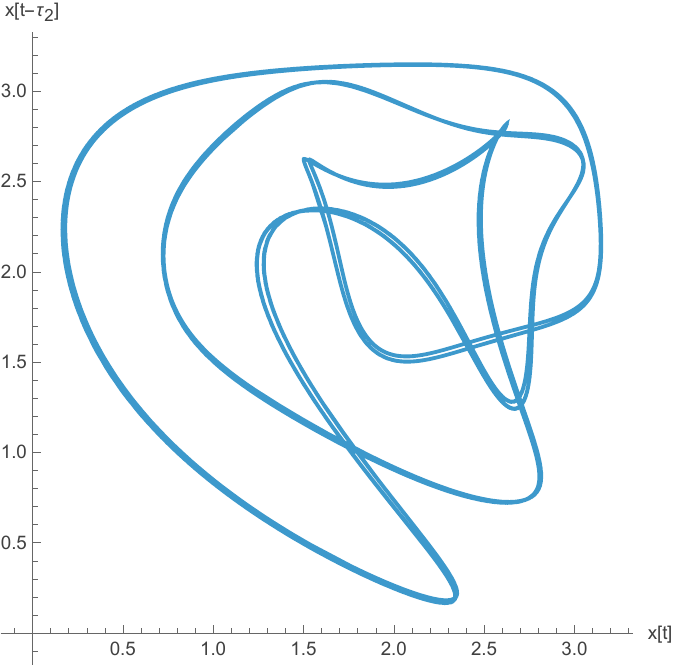}
    \caption{$\tau_1=5.4$}
	\label{tau154}
\end{subfigure}
\hfill
\begin{subfigure}[b]{0.25\textwidth}
    \centering
    \includegraphics[width=\linewidth]{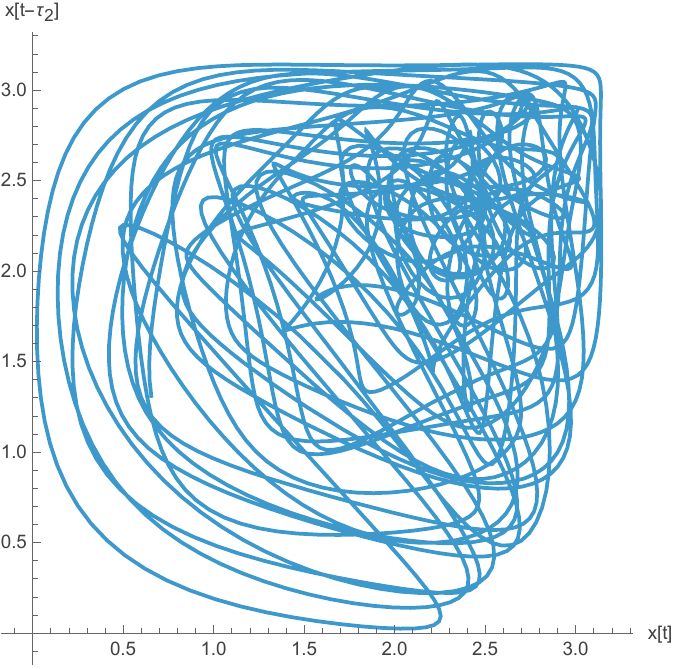}
    \caption{$\tau_1=6.5$}
	\label{tau165}
\end{subfigure}
\caption{Representative phase portraits corresponding to different values of $\tau_1$ illustrating the transitions between periodic and chaotic dynamics.}
\label{fig:tau1_sensitivity}
\end{figure}
}

{
\section{Chaos Control}
\label{sec:control}
In this section, we investigate the suppression of chaotic oscillations exhibited by system~\eqref{eq:sinx} through the introduction of a linear state feedback control applied in a neighborhood of the trivial equilibrium point $x_*=0$. The controlled system is described by
\begin{equation}
	\dot{x}(t)
	= -\gamma x(t)
	+ k\sin\big(x(t - \tau_1)\big)
	- k e^{-\gamma \tau_2}\sin\big(x(t - \tau_1 - \tau_2)\big)
	- \mu x(t),
	\label{eq:control}
\end{equation}
where $\mu>0$ denotes the control gain.}

{
\subsection*{Comparison with other chaos-control approaches}

The controller in~\eqref{eq:control} is intentionally simple: it uses only the current state, introduces no additional delay, and contains one constant gain. Its purpose differs from several established chaos-control strategies. Pyragas delayed feedback applies a signal of the form $K[x(t-T)-x(t)]$ and is non-invasive on a target periodic orbit of period $T$ \cite{pyragas1992continuous}. This is attractive when the objective is to preserve and stabilize an unstable periodic orbit embedded in a chaotic attractor. However, it requires an appropriate feedback delay, usually related to the unknown or measured orbit period, and its success depends on the orbit and gain being within a stabilizable range. In the present study the target is instead the known equilibrium $x_*=0$; hence, period estimation and an additional delayed measurement are unnecessary.

Adaptive controllers update one or more gains online and are particularly useful when system parameters are uncertain or vary with time \cite{tian1998adaptive,rezaie2011adaptive}. Their greater robustness to parameter changes is obtained at the cost of an adaptation law, additional tuning parameters, and a more involved Lyapunov or Lyapunov--Krasovskii analysis. Sliding-mode controllers can provide strong robustness against matched uncertainty and disturbances and have also been developed for scalar delayed chaotic systems \cite{vasegh2009sliding}. They require the design of a sliding surface and a reaching law, and discontinuous switching may produce chattering unless a boundary layer or a higher-order sliding algorithm is introduced. By contrast, $-\mu x(t)$ is smooth and does not generate switching-induced chattering, although it is not adaptive and its stability margin can deteriorate if the model parameters depart substantially from the values used in the gain calculation.

Table~\ref{tab:control_comparison} summarizes these differences. The advantage claimed for the proposed method is therefore not universal superiority, but low implementation complexity and a directly computable local stability boundary for the equilibrium-control objective considered here. For a prescribed parameter set, the critical gain $\mu_*$ follows from the imaginary-root conditions of the closed-loop characteristic equation. The same design procedure is demonstrated for two distinct parameter sets below, whereas robust performance under unknown parameters or external disturbances would require an adaptive or robust extension.
}
{
\begin{table}[H]
	\centering
	\small
	\caption{Qualitative comparison of chaos-control approaches for delayed systems.}
	\label{tab:control_comparison}
	\begin{tabular}{p{0.15\textwidth} p{0.24\textwidth} p{0.25\textwidth} p{0.25\textwidth}}
		\hline
		Method & Typical objective and control signal & Main advantages & Main requirements or limitations \\
		\hline
		Proposed linear feedback & Stabilize $x_*=0$ using $u(t)=-\mu x(t)$ & One smooth constant gain; current-state measurement only; no added delay; $\mu_*$ obtained from the characteristic equation & Model-dependent gain; local stability design; no online accommodation of uncertainty \\
		Pyragas feedback & Stabilize an unstable periodic orbit using $K[x(t-T)-x(t)]$ & Non-invasive on the target orbit; preserves the desired periodic motion & Requires selection of $T$ and $K$; effectiveness depends on the target orbit and delay matching \\
		Adaptive control & Stabilize a state or orbit with gains updated online & Can accommodate uncertain or time-varying parameters & Requires an adaptation law, additional tuning, and parameter or error monitoring \\
		Sliding-mode control & Drive the state to a designed sliding manifold & Robust to matched uncertainties and disturbances; potentially rapid convergence & Requires a sliding surface and uncertainty bounds; switching may cause chattering \\
		\hline
	\end{tabular}
\end{table}

Linearizing system~\eqref{eq:control} about $x=0$ uses the same expansion $\sin x=x+O(x^3)$ as in~\eqref{eq:linearized_origin}. The additional control term changes only the instantaneous coefficient, giving
\begin{equation}
	\dot{x}(t)
	= -(\gamma+\mu)x(t)
	+ kx(t-\tau_1)
	- k e^{-\gamma\tau_2}x(t-\tau_1-\tau_2).
	\label{lin}
\end{equation}

Substituting $x(t)=e^{\lambda t}$ into~\eqref{lin}, dividing by $e^{\lambda t}$, and collecting all terms gives
\begin{equation}
	\lambda
	= -(\gamma+\mu)
	+ k e^{-\lambda\tau_1}
	- k e^{-\gamma\tau_2-\lambda(\tau_1+\tau_2)}.
	\label{char}
\end{equation}
or, equivalently,
\begin{equation}
	\lambda+\gamma+\mu
	-k e^{-\lambda\tau_1}
	+k e^{-\gamma\tau_2}e^{-\lambda(\tau_1+\tau_2)}=0.
	\label{eq:controlled_characteristic_zero}
\end{equation}

For the parameter values used in the chaotic dynamics analysis, namely
\[
k=1,\quad \gamma=0.1,\quad \tau_1=5.1,\quad \tau_2=3.8,
\]
the uncontrolled system ($\mu=0$) exhibits chaotic behavior (see Figure (\ref{fig:control_chaos})). To determine a suitable value of the control gain $\mu$ that suppresses chaos, we analyze the stability boundary of the characteristic equation~\eqref{char} by imposing the condition $\lambda=i\omega$ $(\omega>0)$.

Before inserting numerical values, use
\[
	e^{-i\omega\tau}=\cos(\omega\tau)-i\sin(\omega\tau).
\]
Substitution of $\lambda=i\omega$ into~\eqref{char} and separation of real and imaginary parts yield the general boundary equations
\begin{align}
	0
	&=-(\gamma+\mu)
	+k\cos(\omega\tau_1)
	-k e^{-\gamma\tau_2}\cos\bigl(\omega(\tau_1+\tau_2)\bigr),
	\label{eq:control_boundary_real_general}\\
	\omega
	&=-k\sin(\omega\tau_1)
	+k e^{-\gamma\tau_2}\sin\bigl(\omega(\tau_1+\tau_2)\bigr).
	\label{eq:control_boundary_imag_general}
\end{align}
Notice that~\eqref{eq:control_boundary_imag_general} does not contain $\mu$. Therefore, the computational procedure is: first solve this scalar transcendental equation for each positive candidate frequency $\omega$; next substitute each frequency into~\eqref{eq:control_boundary_real_general} and calculate
\begin{equation}
	\mu(\omega)
	=-\gamma
	+k\cos(\omega\tau_1)
	-k e^{-\gamma\tau_2}\cos\bigl(\omega(\tau_1+\tau_2)\bigr).
	\label{eq:gain_from_frequency}
\end{equation}

For the parameter values above, the complex boundary equation becomes
\[
i\omega
= -\mu - 0.1
+ e^{-5.1 i\omega}
- e^{-0.38-8.9 i\omega}.
\]
Separating real and imaginary parts leads to the system
\begin{eqnarray}
	0 &=& -\mu - 0.1 + \cos(5.1\omega) - e^{-0.38}\cos(8.9\omega), \label{real} \\
	\omega &=& -\sin(5.1\omega) + e^{-0.38}\sin(8.9\omega). \label{imag}
\end{eqnarray}

Define
\[
	F(\omega)
	=\omega+\sin(5.1\omega)-e^{-0.38}\sin(8.9\omega).
\]
The relevant positive zero of $F$ is $\omega_*\approx1.0302$. Substitution into~\eqref{eq:gain_from_frequency} gives $\mu_*\approx1.077$, reported numerically as
\[
	\mu_*=1.0765.
\]
Thus, $\mu_*$ is the local stability-switching value associated with this imaginary characteristic-root pair. The characteristic roots must be checked on either side of the boundary to determine the stable side; the numerical root calculation and the simulations below show instability for gains just below $\mu_*$ and convergence to $x_*=0$ for the tested gains above it.
}

The effectiveness of the proposed control strategy is demonstrated through numerical simulations. Figure~\ref{fig:control_chaos} shows the chaotic behavior of the uncontrolled system for $\mu=0$. When the control gain is chosen close to the threshold value ($\mu=1.03<\mu_*$), trajectories move away from the equilibrium, indicating instability (Figure ~\ref{fig:control_unstable}). In contrast, for $\mu>\mu_*$, for example $\mu=1.4$, the chaotic oscillations are completely suppressed and the system converges to the equilibrium point $x=0$, as shown in Figure ~\ref{fig:control_stable}.

These numerical results are in excellent agreement with the analytical stability conditions derived from the characteristic equation, confirming that linear state feedback provides an effective and robust mechanism for chaos control in the proposed delay differential equation.

\begin{figure}[h!]
	\centering
	\begin{subfigure}[b]{0.32\textwidth}
		\centering
		\includegraphics[width=\linewidth]{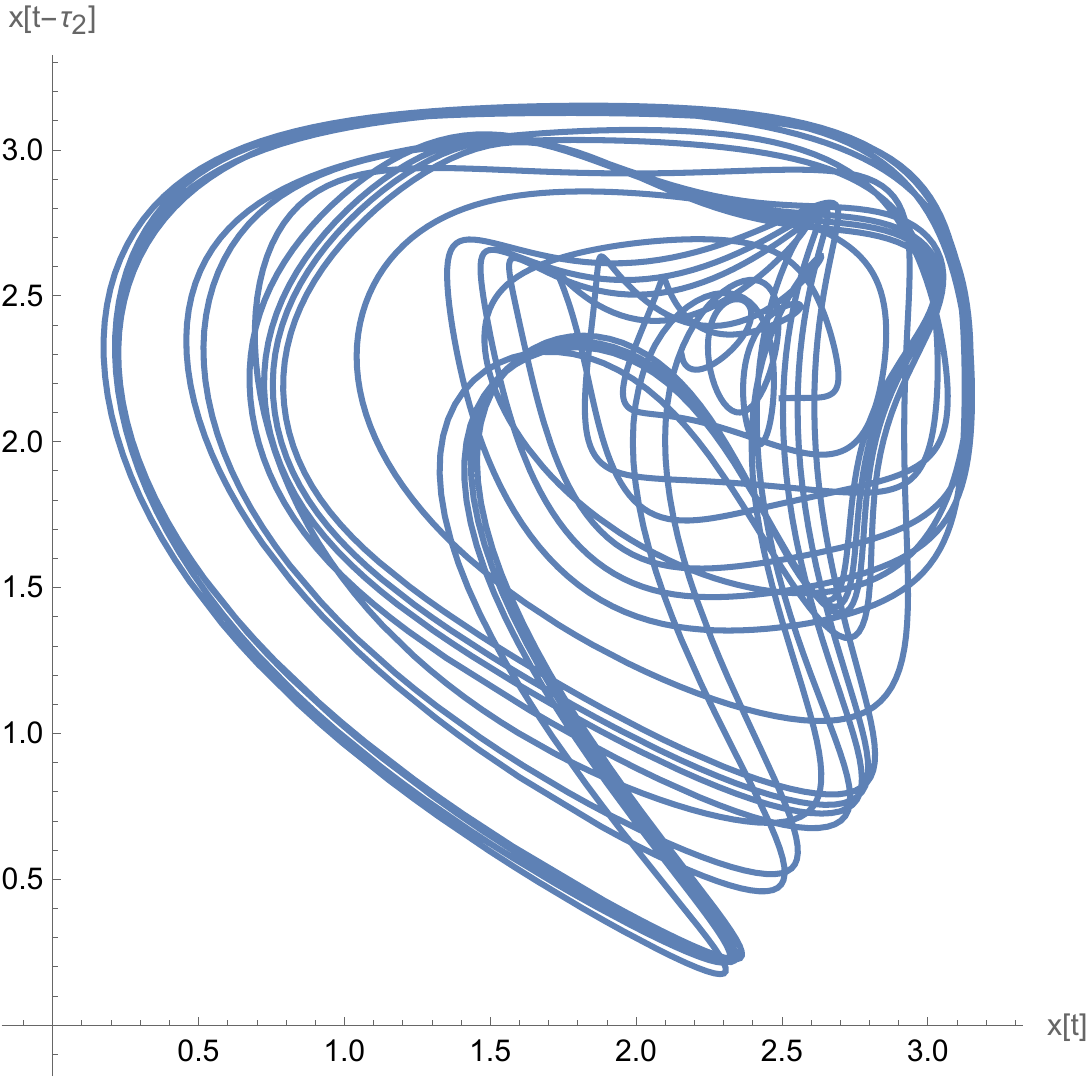}
		\caption{Uncontrolled ($\mu=0$).}
		\label{fig:control_chaos}
	\end{subfigure}
	\hfill
	\begin{subfigure}[b]{0.32\textwidth}
		\centering
		\includegraphics[width=\linewidth]{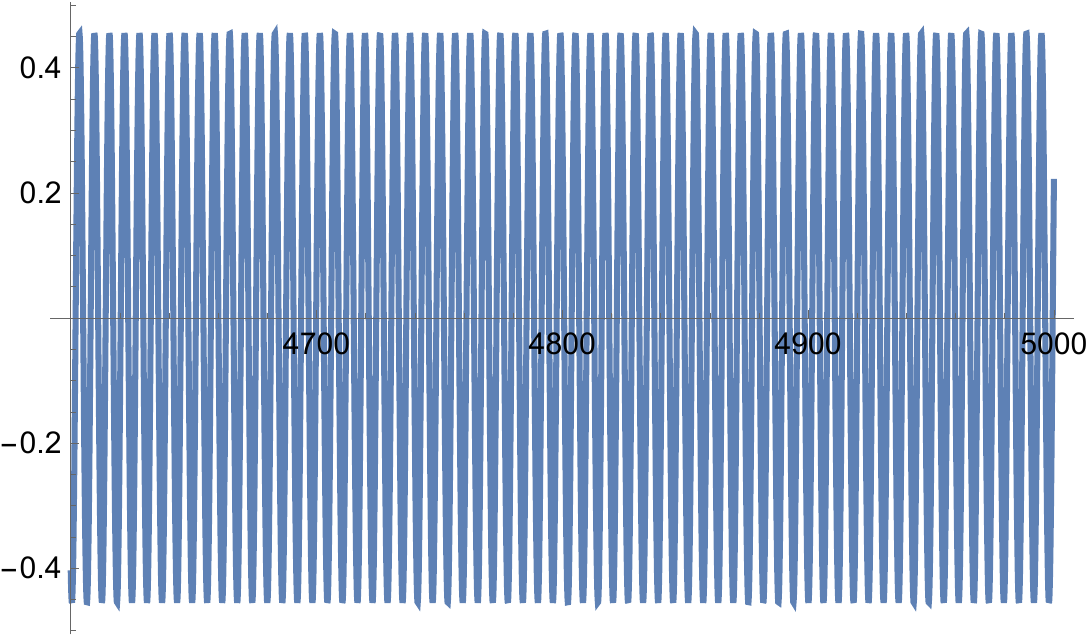}
		\caption{Unstable ($\mu=1.03 < \mu_*$).}
		\label{fig:control_unstable}
	\end{subfigure}	
	\hfill
	\begin{subfigure}[b]{0.32\textwidth}
		\centering
		\includegraphics[width=\linewidth]{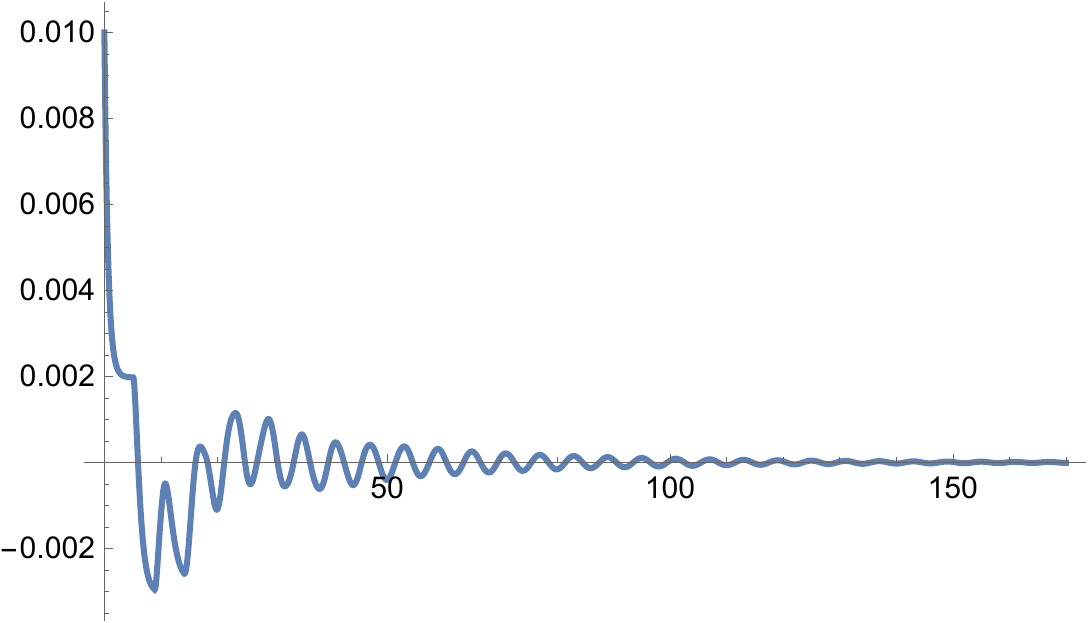}
		\caption{Stable ($\mu=1.5 > \mu_*$).}
		\label{fig:control_stable}
	\end{subfigure}
	\caption{Numerical verification of chaos suppression using linear state feedback control for $k=1,\;\gamma=0.1,\;\tau_1=5.1,\;\tau_2=3.8$.}
\end{figure}
\subsection*{Chaos control under the alternative parameter set}

The robustness of the proposed linear state feedback control is further examined using an alternative set of parameters, namely
\( k = 1.5 \), \( \gamma = 0.3 \), and \( \tau_1 = 4.5 \).
For \( \tau_2 = 7 \), the uncontrolled system exhibits chaotic dynamics.
Numerical simulations show that the same control strategy remains effective and successfully suppresses chaos when the control gain \( \mu \) exceeds a critical threshold value \( \mu_* = 1.0163 \).

Specifically, for \( \mu = 0 \), the controlled system reduces to the original system~\eqref{eq:sinx}, and a double-scroll chaotic attractor is observed, as shown in Figure ~\ref{4_k}.
When \( \mu < \mu_* \) and close to the critical value, the trajectories diverge from the trivial equilibrium, indicating instability (Figure ~\ref{4_l}).
However, for \( \mu > \mu_* \), chaotic oscillations are completely suppressed, and all trajectories converge asymptotically to the equilibrium point \( x_* = 0 \), as illustrated in Figure ~\ref{4_m}.
\begin{figure}[H]
	\centering
	\begin{subfigure}[b]{0.32\textwidth}
		\centering
		\includegraphics[width=\linewidth]{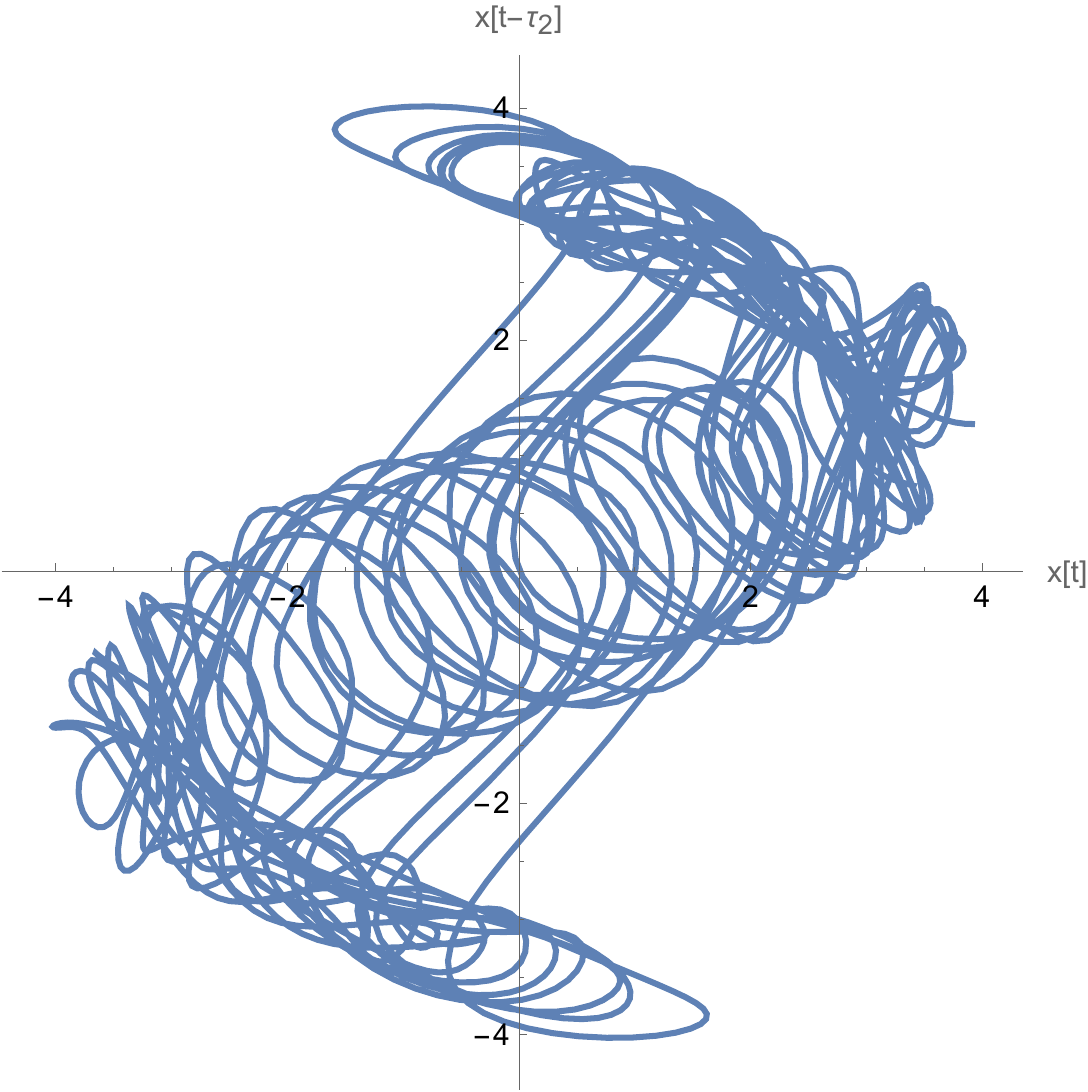}
		\caption{Double chaotic attractor for\\ $\mu=0$.}
		\label{4_k}
	\end{subfigure}\hfill
	\begin{subfigure}[b]{0.32\textwidth}
		\centering
		\includegraphics[width=\linewidth]{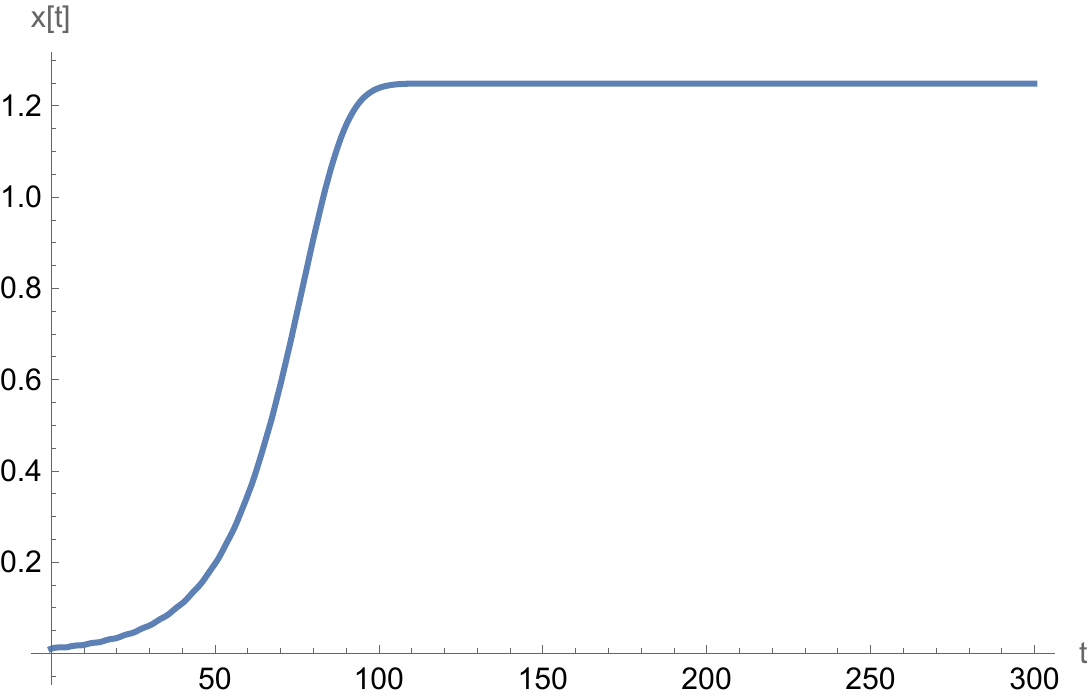}
		\caption{Unstable behavior for\\ $\mu=0.8<\mu_*$.}
		\label{4_l}
	\end{subfigure}\hfill
	\begin{subfigure}[b]{0.32\textwidth}
		\centering
		\includegraphics[width=\linewidth]{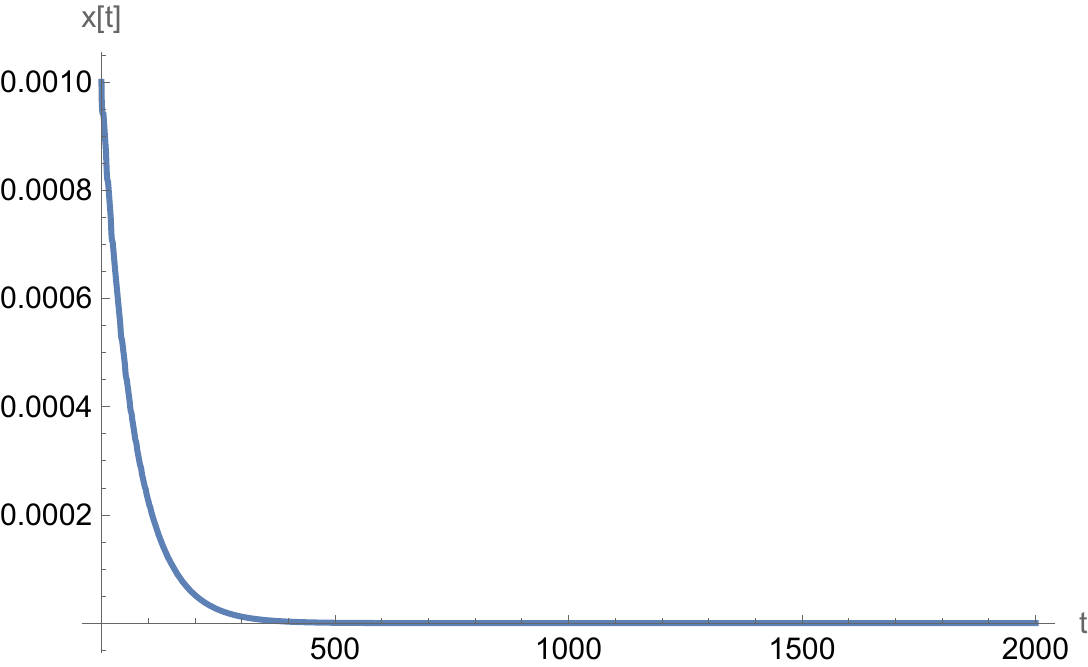}
		\caption{Converging solution for\\ $\mu=1.1>\mu_*$.}
		\label{4_m}
	\end{subfigure}
	\caption{Numerical verification of chaos suppression using linear state feedback control for $k=1.5,\;\gamma=0.3\;\tau_1=4.5,\;\tau_2=7$.}
\end{figure}
\section{Synchronization}
\label{sec:sync}
We say that the two given dynamical systems are synchronized if their trajectories evolve together after some time. To achieve synchronization, one has to apply a suitable control to one of these systems. In \cite{pecora1990synchronization}, Pecora and Carroll proposed the synchronization phenomena. It is investigated by the researchers \cite{boccaletti2002synchronization,landsman2007complete,rulkov1995generalized} that the chaotic systems can also be synchronized. Synchronization has important applications in secure communications \cite{kocarev1995general,mensour1998synchronization}, neural networks and brain dynamics \cite{klein2004synchronization}, power grid stability \cite{motter2013spontaneous} and other complex dynamical systems. In this section, we propose a simple linear feedback control and synchronize the system (\ref{eq:sinx}) with itself. \\
Consider the master system
\begin{equation}
	\dot{x}_1(t)
	= -\gamma x_1(t)
	+ k\sin\big(x_1(t - \tau_1)\big)
	- k e^{-\gamma \tau_2}\sin\big(x_1(t - \tau_1 - \tau_2)\big),
	\label{eq:ms}
\end{equation}
and the slave system
\begin{equation}
	\dot{x}_2(t)
	= -\gamma x_2(t)
	+ k\sin\big(x_2(t - \tau_1)\big)
	- k e^{-\gamma \tau_2}\sin\big(x_2(t - \tau_1 - \tau_2)\big)
	+ C(x_1,x_2),
	\label{eq:ss}
\end{equation}
where $C(x_1,x_2)$ is a control input.\\
Define the synchronization error
\begin{equation}
	e(t)=x_1(t)-x_2(t).
	\label{eq:error}
\end{equation}
Subtracting \eqref{eq:ss} from \eqref{eq:ms}, we obtain
\begin{align}
	\dot e(t) &= -\gamma e(t)
	+ k\Big[\sin(x_1(t-\tau_1))-\sin(x_2(t-\tau_1))\Big] \nonumber \\
	&\quad - k e^{-\gamma\tau_2}
	\Big[\sin(x_1(t-\tau_1-\tau_2))-\sin(x_2(t-\tau_1-\tau_2))\Big]
	- C(x_1,x_2).
	\label{eq:error_raw}
\end{align}
We choose the linear feedback
\[
C(x_1,x_2)=\delta (x_1-x_2)=\delta e(t),
\]
where $\delta\in \mathbf{R}$. Then
\begin{equation}
		\dot e(t) = -(\gamma+\delta)e(t)
	+ k\Big[\sin(x_1(t-\tau_1))-\sin(x_2(t-\tau_1))\Big] 
	- k e^{-\gamma\tau_2}
	\Big[\sin(x_1(t-\tau_1-\tau_2))-\sin(x_2(t-\tau_1-\tau_2))\Big].
	\label{eq:error_controlled}
	\end{equation}

Using the mean value theorem,
\[
\sin a - \sin b = \cos(\xi)(a-b),
\]
for some $\xi$ between $a$ and $b$, we obtain
\begin{equation}
	\dot e(t)
	= -(\gamma+\delta)e(t)
	+ k \cos(\xi_1(t)) e(t-\tau_1)
	- k e^{-\gamma\tau_2}\cos(\xi_2(t)) e(t-\tau_1-\tau_2),
	\label{eq:error_cos}
\end{equation}
where $|\cos(\xi_i(t))|\le 1$.\\

{
Taking absolute values give
\begin{equation}
	|\dot e(t)|
	\le -(\gamma+\delta)|e(t)|
	+ |k|\,|e(t-\tau_1)|+ |k| e^{-\gamma \tau_2}\,|e(t-\tau_1-\tau_2)|.
	\label{eq:bound_exact}
\end{equation}
Now, for $\gamma\ge0,\tau_2\ge0$, we have $e^{-\gamma\tau_2}\le1$, therefore
\begin{equation}
	|\dot e(t)|
	\le -(\gamma+\delta)|e(t)|
	+ |k|\,|e(t-\tau_1)|+ |k|\,|e(t-\tau_1-\tau_2)|.
	\label{eq:bound}
\end{equation}
}

\begin{theorem}[Delay-Independent Synchronization]
	\label{synch}
	If
	\begin{equation}
		\gamma+\delta > 2|k|,
		\label{eq:main_condition}
	\end{equation}
	then the zero solution of
	\begin{equation}
		\dot z(t)
		= -(\gamma+\delta)z(t)
		+ |k| z(t-\tau_1)
		+ |k| z(t-\tau_1-\tau_2)
		\label{eq:comparison}
	\end{equation}
	is asymptotically stable for all $\tau_1,\tau_2 \ge 0$. 
	Consequently,
	\[
	\lim_{t\to\infty} e(t)=0,
	\]
	and complete synchronization is achieved.
\end{theorem}

\begin{proof}
	The characteristic equation of \eqref{eq:comparison} is
	\[
	\lambda + (\gamma+\delta)
	- |k| e^{-\lambda\tau_1}
	- |k| e^{-\lambda(\tau_1+\tau_2)}=0.
	\]
	Suppose $\lambda = iv$ ($v \in \mathbb{R}$, $v\neq 0$) is a purely imaginary root.
	Substituting $\lambda = iv$ and using
	\[
	e^{-iv\tau}=\cos(v\tau)-i\sin(v\tau),
	\]
	we obtain
	\begin{align*}
		iv + (\gamma+\delta)
		- |k|\big(\cos(v\tau_1)-i\sin(v\tau_1)\big)  
		- |k|\big(\cos(v(\tau_1+\tau_2))
		- i\sin(v(\tau_1+\tau_2))\big)=0.
	\end{align*}
	Separating real and imaginary parts yields
	\begin{align}
		0 &= -(\gamma+\delta)
		+ |k|\cos(v\tau_1)
		+ |k|\cos(v(\tau_1+\tau_2)), 
		\label{eq:real_part} \\
		v &= - |k|\sin(v\tau_1)
		- |k|\sin(v(\tau_1+\tau_2)).
		\label{eq:imag_part}
	\end{align}	
	From \eqref{eq:real_part}, using $|\cos(\theta)|\le 1$, we obtain
	\[
-2|k| \le |k|\cos(v\tau_1)
	+ |k|\cos(v(\tau_1+\tau_2))
	\le 2|k|.
	\]	
	Thus, \eqref{eq:real_part} can hold only if
	\[
-2|k| \le \gamma+\delta \le 2|k|.
	\]	
	Hence, if
	\[
	\gamma+\delta > 2|k|,
	\]
	then equation \eqref{eq:real_part} has no solution.
	Therefore, the characteristic equation admits no purely imaginary roots
	for any $\tau_1,\tau_2 \ge 0$.
	For $\tau_1=\tau_2=0$, system \eqref{eq:comparison} reduces to
	\[
	\dot z(t)=-(\gamma+\delta-2|k|)z(t),
	\]
which is asymptotically stable when \eqref{eq:main_condition} holds. Therefore, the system (\ref{eq:comparison}) is stable for all $\tau_1,\tau_2\ge0$ under the condition (\ref{eq:main_condition}).\\ From inequality \eqref{eq:bound}, the  stability of the system (\ref{eq:comparison}) implies \[ |e(t)|\to 0 \quad \text{as } t\to\infty. \] Hence, complete synchronization is achieved. 
 \medskip \noindent
\end{proof}

{

\begin{theorem}[Less Conservative Fixed-$\tau_2$ Synchronization]
\label{thm:fixed_tau2_sync}
Assume that $\gamma\geq0$, fix $\tau_2\geq0$, and let $\tau_1\geq0$ be arbitrary. If the feedback gain $\delta$ satisfies
\begin{equation}
\gamma+\delta>|k|\bigl(1+e^{-\gamma\tau_2}\bigr),
\label{eq:less_conservative_sync}
\end{equation}
then the master and slave systems~\eqref{eq:ms}--\eqref{eq:ss} are globally exponentially synchronized. More precisely, let $\eta>0$ be the unique positive solution of
\begin{equation}
\eta
=
(\gamma+\delta)
-|k|e^{-\eta\tau_1}
-|k|e^{-\eta(\tau_1+\tau_2)}.
\label{eq:fixed_tau2_rate}
\end{equation}
Then, for every $\sigma\in(0,\eta)$, there exists a constant $M_\sigma>0$, determined by the initial error history, such that
\begin{equation}
|e(t)|\leq M_\sigma e^{-\sigma t},
\qquad t\geq0.
\label{eq:fixed_tau2_decay}
\end{equation}
Consequently,
\[
\lim_{t\to\infty}e(t)=0.
\]
\end{theorem}

\begin{proof}
Let
\[
v(t)=|e(t)|.
\]
Since $v(t)$ may fail to be differentiable at points where $e(t)=0$, we work with its upper right Dini derivative
\[
D^+v(t)
=
\limsup_{h\to0^+}
\frac{v(t+h)-v(t)}{h}.
\]
Using the error equation and the inequality
\[
|\sin u-\sin v|\leq |u-v|,
\]
we obtain
\begin{equation}
D^+v(t)
\leq
-(\gamma+\delta)v(t)
+|k|v(t-\tau_1)
+|k|e^{-\gamma\tau_2}
v(t-\tau_1-\tau_2),
\qquad t\geq0.
\label{eq:fixed_tau2_dini}
\end{equation}
This is a two-delay differential inequality of Halanay type.

Define
\[
F(r)
=
\gamma+\delta-r
-|k|e^{r\tau_1}
-|k|e^{-\gamma\tau_2+r(\tau_1+\tau_2)},
\qquad r\geq0.
\]
Then
\[
F(0)
=
\gamma+\delta-|k|-|k|e^{-\gamma\tau_2}>0
\]
by~\eqref{eq:less_conservative_sync}, whereas
\[
F(r)\to-\infty
\qquad\text{as }r\to\infty.
\]
Furthermore,
\[
F'(r)
=
-1
-|k|\tau_1e^{r\tau_1}
-|k|e^{-\gamma\tau_2}
(\tau_1+\tau_2)e^{r(\tau_1+\tau_2)}
<0.
\]
Thus, $F$ is strictly decreasing and has a unique positive zero $\eta$. The equation $F(\eta)=0$ is precisely~\eqref{eq:fixed_tau2_rate}.

Fix any $\sigma\in(0,\eta)$. Since $F$ is strictly decreasing,
\[
F(\sigma)
=
\gamma+\delta-\sigma
-|k|e^{\sigma\tau_1}
-|k|e^{-\gamma\tau_2}
e^{\sigma(\tau_1+\tau_2)}
>0.
\label{eq:positive_margin_sigma}
\]

Let
\[
e_0(s)=e(s),
\qquad
s\in[-(\tau_1+\tau_2),0],
\]
denote the initial error history and set
\[
M_\sigma
=
\max_{-(\tau_1+\tau_2)\leq s\leq0}
e^{\sigma s}|e_0(s)|.
\]
If $M_\sigma=0$, then the two initial histories coincide, and uniqueness of solutions gives
\[
e(t)\equiv0.
\]
Suppose therefore that $M_\sigma>0$ and define
\[
w(t)=M_\sigma e^{-\sigma t},
\qquad
t\geq-(\tau_1+\tau_2).
\]
By the definition of $M_\sigma$, for every
\[
s\in[-(\tau_1+\tau_2),0],
\]
we have
\[
e^{\sigma s}|e_0(s)|
\leq M_\sigma.
\]
Hence
\[
v(s)=|e_0(s)|
\leq
M_\sigma e^{-\sigma s}
=
w(s),
\]
and therefore
\[
v(s)\leq w(s),
\qquad
-(\tau_1+\tau_2)\leq s\leq0.
\]
Thus, the comparison function $w$ dominates the complete initial error history.

We now prove that
\[
v(t)\leq w(t),
\qquad t\geq0.
\]
Suppose, to the contrary, that there exists $t>0$ such that
\[
v(t)>w(t).
\]
By continuity of $v$ and $w$ and the fact that $v(s)\leq w(s)$ on the initial interval, there exists a first contact time $t_*>0$ such that
\[
v(t_*)=w(t_*),
\qquad
v(t)\leq w(t)
\quad
\text{for }
-(\tau_1+\tau_2)\leq t\leq t_*,
\]
and $v(t)>w(t)$ for times immediately to the right of $t_*$.

Consequently,
\[
D^+(v-w)(t_*)\geq0.
\]
On the other hand, using~\eqref{eq:fixed_tau2_dini} together with
\[
v(t_*-\tau_1)\leq w(t_*-\tau_1)
\]
and
\[
v(t_*-\tau_1-\tau_2)
\leq
w(t_*-\tau_1-\tau_2),
\]
we obtain
\begin{align*}
D^+v(t_*)
&\leq
-(\gamma+\delta)w(t_*)
+|k|w(t_*-\tau_1)\\
&\quad
+|k|e^{-\gamma\tau_2}
w(t_*-\tau_1-\tau_2).
\end{align*}
Since
\[
w(t)=M_\sigma e^{-\sigma t},
\]
we have
\[
w(t_*-\tau_1)
=
e^{\sigma\tau_1}w(t_*),
\]
and
\[
w(t_*-\tau_1-\tau_2)
=
e^{\sigma(\tau_1+\tau_2)}w(t_*).
\]
Therefore,
\begin{align*}
D^+v(t_*)
&\leq
\left[
-(\gamma+\delta)
+|k|e^{\sigma\tau_1}
+|k|e^{-\gamma\tau_2}
e^{\sigma(\tau_1+\tau_2)}
\right]w(t_*)\\
&=
-\left[\sigma+F(\sigma)\right]w(t_*)\\
&<
-\sigma w(t_*).
\end{align*}
Since
\[
w'(t_*)=-\sigma w(t_*),
\]
and $w$ is continuously differentiable, we obtain
\[
D^+(v-w)(t_*)
=
D^+v(t_*)-w'(t_*)
<0,
\]
which contradicts
\[
D^+(v-w)(t_*)\geq0.
\]
Hence
\[
v(t)\leq w(t),
\qquad t\geq0.
\]
Consequently,
\[
|e(t)|
=
v(t)
\leq
M_\sigma e^{-\sigma t},
\qquad t\geq0,
\]
which proves~\eqref{eq:fixed_tau2_decay}. Since $\sigma>0$,
\[
\lim_{t\to\infty}|e(t)|=0.
\]
Thus, complete synchronization is achieved.
\end{proof}

\begin{remark}[Conservativeness and comparison with existing criteria]
Condition~\eqref{eq:main_condition} is a sufficient global condition and is not necessary for synchronization. Its conservativeness enters at three stages. First, the trajectory-dependent factors $\cos(\xi_i(t))$ in~\eqref{eq:error_cos} are replaced by their worst-case magnitude one, although the synchronized trajectory may spend substantial time in regions where $|\cos(\xi_i(t))|\ll1$. Second, taking absolute values removes the opposite signs of the two delayed feedback terms and hence discards possible cancellation between them. Third, the attenuation factor $e^{-\gamma\tau_2}$ is replaced by one in passing from~\eqref{eq:bound_exact} to~\eqref{eq:bound}. The resulting bound is uniform over both delays, but it requires a larger coupling gain than may be needed for a specified pair $(\tau_1,\tau_2)$. Consequently, failure of $\gamma+\delta>2|k|$ does not imply failure of synchronization.

Retaining the attenuation factor in~\eqref{eq:bound_exact} gives the less conservative fixed-$\tau_2$ criterion established in Theorem~\ref{thm:fixed_tau2_sync}. Because $1+e^{-\gamma\tau_2}\leq2$, condition~\eqref{eq:main_condition} is the worst-case envelope of~\eqref{eq:less_conservative_sync} over all $\tau_2\geq0$. Even~\eqref{eq:less_conservative_sync} remains conservative because it does not use the signs or time variation of the cosine factors.

Sharper synchronization thresholds in the literature generally retain more system information. Mensour and Longtin \cite{mensour1998synchronization} studied synchronization of delay-differential equations through model-specific error dynamics. Pyragas \cite{pyragas1998synchronization} estimated coupling thresholds using Krasovskii--Lyapunov theory and a large-delay perturbation approach; such estimates can exploit the particular delayed dynamics but depend more strongly on the model and its operating regime. For synchronized equilibria, Michiels and Nijmeijer \cite{michiels2009synchronization} factorized the linearized characteristic function and constructed stability regions in the coupling-gain--delay plane. Applied to the present model, an analogous local analysis would retain the signed coefficients and the actual values of $\tau_1$ and $\tau_2$, and could therefore certify synchronization for gains below~\eqref{eq:main_condition}; however, it would provide a delay-dependent local result rather than the present closed-form global bound. Lyapunov--Krasovskii/LMI criteria offer another route to reduced conservativeness by incorporating delay intervals and additional decision matrices, at the cost of numerical matrix feasibility tests instead of a scalar analytical condition.

Thus, the main advantage of~\eqref{eq:main_condition} is its simplicity, global character, and independence of the two delays. Its price is conservativeness. Equation~\eqref{eq:less_conservative_sync} provides an intermediate alternative when $\tau_2$ is known, while characteristic-root or transverse-Lyapunov analyses are more appropriate when the smallest admissible gain for a particular delay pair is required.
\end{remark}
}

\subsection*{Numerical Verification}

To verify the theoretical result, numerical simulations are carried out
for parameter values satisfying condition \eqref{eq:main_condition}.
The master and slave systems are simulated with different initial
functions on $[-(\tau_1+\tau_2),0]$.
It is observed that $x_1(t)$ and $x_2(t)$ converge to the same trajectory
as $t$ increases (see Figures (\ref{4_5_1} and \ref{4_5_2})). 
The synchronization error $e(t)$ decreases to zero (Figures (\ref{4_5_3} and \ref{4_5_4})), confirming synchronization.
These simulations agree with the analytical result.

\begin{figure}[H]
	\centering
	\begin{subfigure}[b]{0.48\textwidth}
		\centering
		\includegraphics[width=\linewidth,height=8cm,keepaspectratio]{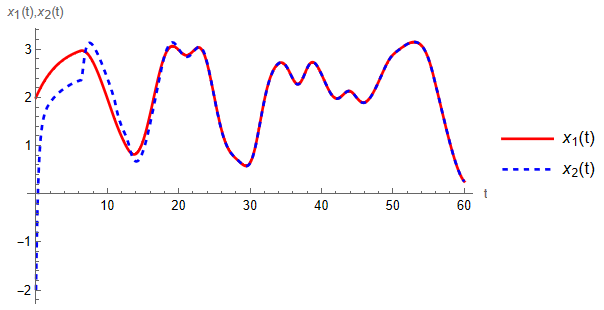}
		\caption{Time evolution of $x_1(t)$ and $x_2(t)$ for $k=1.4,\;\gamma=0.3,\;\delta=2.8,\;\tau_1=6.2,\;\tau_2=4.5$.}
		\label{4_5_1}
	\end{subfigure}\hfill
	\begin{subfigure}[b]{0.48\textwidth}
		\centering
		\includegraphics[width=\linewidth,height=8cm,keepaspectratio]{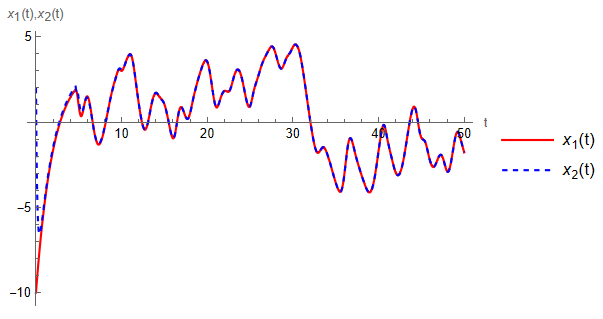}
		\caption{Time evolution of $x_1(t)$ and $x_2(t)$ for $k=3.2,\;\gamma=0.5,\;\delta=6,\;\tau_1=4.6,\;\tau_2=5.2$.}
		\label{4_5_2}
	\end{subfigure}
	
	\caption{Synchronization of master and slave states under linear state feedback control. The trajectories $x_1(t)$ and $x_2(t)$ converge for both parameter sets.}
\end{figure}
\begin{figure}[H]
	\centering
	\begin{subfigure}[b]{0.43\textwidth}
		\centering
		\includegraphics[width=\linewidth]{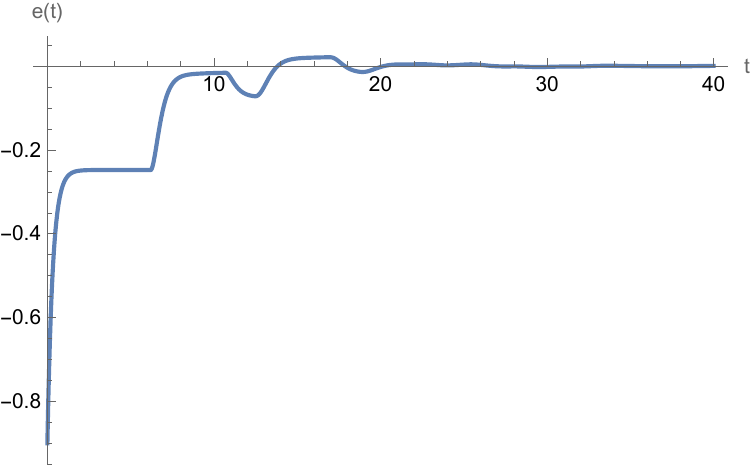}
		\caption{Synchronization error $e(t)$ for $k=1.4,\;\gamma=0.3,\;\delta=2.8,\;\tau_1=6.2,\;\tau_2=4.5$.}
		\label{4_5_3}
	\end{subfigure}\hfill
	\begin{subfigure}[b]{0.43\textwidth}
		\centering
		\includegraphics[width=\linewidth]{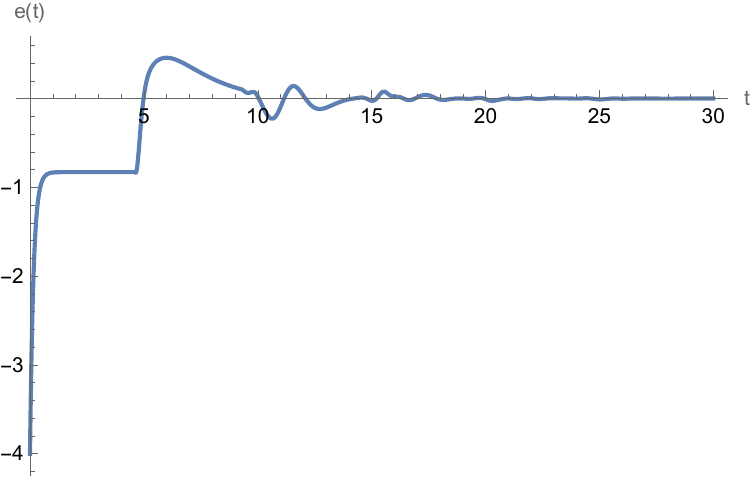}
		\caption{Synchronization error $e(t)$ for $k=3.2,\;\gamma=0.5,\;\delta=6,\;\tau_1=4.6,\;\tau_2=5.2$.}
		\label{4_5_4}
	\end{subfigure}
	
	\caption{Decay of the synchronization error $e(t)$ confirming asymptotic synchronization for parameter sets satisfying $\gamma+\delta>2|k|$.}
\end{figure}

To better show how fast the synchronization error decreases, 
Table~\ref{tab:error} lists the numerical values of $e(t)$ 
for the parameter set $k=1.4,\;\gamma=0.3,\;\delta=2.8,\;\tau_1=6.2,\;\tau_2=4.5$. 
It can be seen that the error decreases quickly and approaches zero as time $t$ increases. 
This agrees with the stability result given in Theorem~\ref{synch}. 
Hence, the numerical results confirm that the proposed linear feedback control is effective.

\begin{table}[H]
	\centering
	\caption{Time evolution of synchronization error}
	\label{tab:error}
	\begin{tabular}{cc}
	\hline
	$t$ & $e(t)$ \\
	\hline
	0 & 4. \\
	5 & 0.608387 \\
	10 & 0.418386 \\
	15 & 0.207187 \\
	20 & 0.0219953 \\
	25 & 0.00859616 \\
	30 & 0.0020686 \\
	35 & 0.00249555 \\
	40 & 0.00033673 \\
	\hline
\end{tabular}
\hspace{1.5cm}
\begin{tabular}{cc}
	\hline
$t$ & $e(t)$ \\
\hline
	45 & 0.00016778 \\
50 & 0.00013447 \\
55 & 0.00003783 \\
60 & 0.00002399 \\
65 & $7.29\times 10^{-6}$ \\
70 & $3.03\times 10^{-6}$ \\
75 & $6\times 10^{-8}$ \\
80 & $5.35\times 10^{-7}$ \\
	85 & $9\times 10^{-7}$ \\
\hline
\end{tabular}
\hspace{1.5cm}
\begin{tabular}{cc}
	\hline
	$t$ & $e(t)$ \\
	\hline
	
	90 & $2.57\times 10^{-7}$ \\
	95 & $6.7\times 10^{-8}$ \\
	100 & $1.8\times 10^{-8}$ \\
	105 &$ 6\times 10^{-9}$ \\
	110 & $1.2\times 10^{-8}$ \\
	115 & $8\times 10^{-9}$ \\
	120 & $1\times 10^{-9}$  \\
	125& $1\times 10^{-9}$\\
	130&0\\
	\hline
	
\end{tabular}

\end{table}

Synchronizing chaotic systems is challenging because they are highly sensitive to initial conditions. Even a small difference at the start can lead to completely different trajectories.

First, we consider the master and slave systems~\eqref{eq:ms}--\eqref{eq:ss} without control, that is,
\[
C(x_1,x_2)=0 \quad (\delta=0).
\]
For the initial conditions $x_1(0)=0.01$ and $x_2(0)=0.015$, the trajectories diverge over time, as shown in Figure (\ref{4_5_5}). This confirms the chaotic and unsynchronized behavior of the systems.

Next, the control input $C(x_1,x_2)$ is applied with $\delta>0$. Even with a much larger initial mismatch, namely $x_1(0)=0.01$ and $x_2(0)=1$, the trajectories quickly converge (Figure \ref{4_5_6}). This shows that the proposed control law is effective in achieving synchronization, as the synchronization error approaches zero despite large initial differences.

\begin{figure}[H]
	\centering
	\begin{subfigure}[b]{0.47\textwidth}
		\centering
		\includegraphics[width=\linewidth]{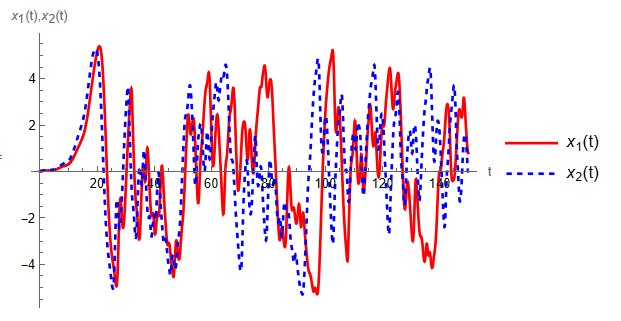}
		\caption{Uncontrolled system with $k=1.4,\;\gamma=0.3,\\
			\;\delta=0,
			\;\tau_1=6.2,\;\tau_2=4.5$}
	\label{4_5_5}
	\end{subfigure}\hfill
	\begin{subfigure}[b]{0.47\textwidth}
		\centering
		\includegraphics[width=\linewidth]{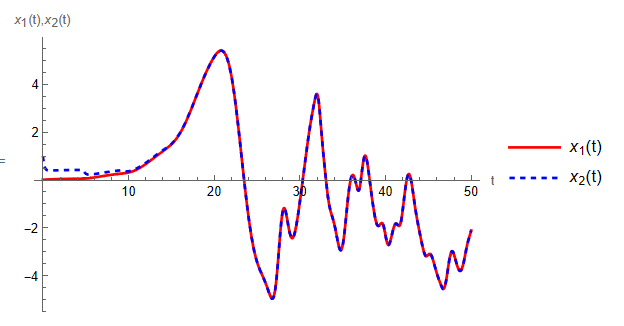}
		\caption{Controlled system with $k=1.4,\;\gamma=0.3,\;\\
			\delta=6.2,\;\tau_1=6.2,\;\tau_2=4.5$}
	\label{4_5_6}
	\end{subfigure}

\caption{Numerical demonstration of synchronization for the parameter set 
	$k=1.4$, $\gamma=0.3$, $\tau_1=6.2$, and $\tau_2=4.5$. 
	(a) In the absence of control ($\delta=0$), the master and slave trajectories 
	diverge due to sensitivity to initial conditions. 
	(b) With the proposed control ($\delta=6.2$), synchronization is achieved even 
	for a large initial mismatch, and the synchronization error converges to zero, 
	satisfying the condition $\gamma+\delta>2|k|$.}

\end{figure}

\section{Conclusion}
\label{sec:conc}

This work extends the analysis of the two-delay model \ref{eq:main} from equilibrium stability to nonlinear dynamics, chaos control, and synchronization for the nonlinear feedback $g(x)=k\sin x$. The stability analysis provides delay-independent stability and instability conditions, while the stability switching analysis identifies the critical delays at which the equilibrium changes its stability.

The numerical study reveals a rich variety of dynamical behaviors, including periodic oscillations, periodic windows, and single-scroll as well as double-scroll chaotic attractors. These dynamics are confirmed using bifurcation diagrams, largest Lyapunov exponents, Poincar\'e maps, and power spectral analysis. A parameter sensitivity analysis further shows that the system dynamics are highly dependent on the model parameters.

A linear state-feedback controller is proposed to suppress chaotic oscillations, and computable critical feedback gains are obtained from the controlled characteristic equation. The paper also develops a master--slave synchronization scheme and derives delay-independent sufficient conditions for complete synchronization based on the global Lipschitz property of the sine function and a Halanay-type comparison argument .

Overall, the results demonstrate that the coupled-delay structure gives rise to rich nonlinear dynamics and that these dynamics can be effectively analyzed, controlled, and synchronized.

	\section*{Acknowledgment}	
Pragati Dutta thanks the University of Hyderabad for the non-net fellowship (UH/DSW/F\&/Non-NET/2023/165).

\section*{Declarations}
The authors declare that they have no known competing financial interests or personal relationships that could have appeared to influence the work reported in this paper.\\
Generative AI tools were used only to improve the language and 
clarity of this manuscript. The authors carefully reviewed and 
edited the output and take full responsibility for the final content.

\bibliographystyle{unsrt}	
\bibliography{reff}

@article{krawiec2026lambert,
  title={Lambert W Function in Solving Delay Differential Equations for Modeling in Economics and Finance.},
  author={Krawiec, Adam and Matsumoto, Akio and Po{\l}e{\'c}, Anna},
  journal={Nonlinear Dynamics, Psychology \& Life Sciences},
  volume={30},
  number={1},
  pages={82},
  year={2026}
}

@article{berezowski2026influence,
  title={The influence of delay time on the generation of chaotic oscillations in a two-dimensional resonant system},
  author={Berezowski, Marek},
  journal={Chaos, Solitons \& Fractals},
  volume={208},
  pages={118158},
  year={2026},
  publisher={Elsevier}
}

@article{lin2025hopf,
  title={Hopf bifurcation and controller design for a predator--prey model with double delays},
  author={Lin, Jinting and Xu, Changjin and Zhao, Yingyan and Deng, Qinwen},
  journal={AIP Advances},
  volume={15},
  number={7},
  year={2025},
  publisher={AIP Publishing}
}

@article{boullu2020stability,
  title={Stability analysis of an equation with two delays and application to the production of platelets},
  author={Boullu, Lois and Pujo-Menjouet, Laurent and B{\'e}lair, Jacques},
  journal={Discrete and Continuous Dynamical Systems-Series S},
  volume={13},
  number={11},
  pages={3005--3027},
  year={2020}
}

@article{belair1994stability,
	title={Stability and bifurcations of equilibria in a multiple-delayed differential equation},
	author={B{\'e}lair, Jacques and Campbell, Sue Ann},
	journal={SIAM Journal on Applied Mathematics},
	volume={54},
	number={5},
	pages={1402--1424},
	year={1994},
	publisher={SIAM}
}

@article{bhalekar2016stability,
  title={Stability and bifurcation analysis of a generalized scalar delay differential equation},
  author={Bhalekar, Sachin},
  journal={Chaos: An Interdisciplinary Journal of Nonlinear Science},
  volume={26},
  number={8},
  year={2016},
  publisher={AIP Publishing}
}

@article{bhalekar2025analysis,
  title={Analysis of a class of two-delay fractional differential equation},
  author={Bhalekar, Sachin and Dutta, Pragati},
  journal={Chaos: An Interdisciplinary Journal of Nonlinear Science},
  volume={35},
  number={1},
  year={2025},
  publisher={AIP Publishing}
}

@book{erneux2009applied,
  title={Applied delay differential equations},
  author={Erneux, Thomas},
  year={2009},
  publisher={Springer}
}

@article{li1999stability,
  title={Stability and bifurcation in delay--differential equations with two delays},
  author={Li, Xiangao and Ruan, Shigui and Wei, Junjie},
  journal={Journal of Mathematical Analysis and Applications},
  volume={236},
  number={2},
  pages={254--280},
  year={1999},
  publisher={Elsevier}
}

@article{bhalekar2010synchronization,
  title={Synchronization of different fractional order chaotic systems using active control},
  author={Bhalekar, Sachin and Daftardar-Gejji, Varsha},
  journal={Communications in Nonlinear Science and Numerical Simulation},
  volume={15},
  number={11},
  pages={3536--3546},
  year={2010},
  publisher={Elsevier}
}

@article{bhalekar2012dynamical,
  title={Dynamical analysis of fractional order U{\c{c}}ar prototype delayed system},
  author={Bhalekar, Sachin},
  journal={Signal, Image and Video Processing},
  volume={6},
  number={3},
  pages={513--519},
  year={2012},
  publisher={Springer}
}

@book{lakshmanan2011dynamics,
  title={Dynamics of nonlinear time-delay systems},
  author={Lakshmanan, Muthusamy and Senthilkumar, Dharmapuri Vijayan},
  year={2011},
  publisher={Springer Science \& Business Media}
}

@article{sprott2007simple,
  title={A simple chaotic delay differential equation},
  author={Sprott, JC},
  journal={Physics Letters A},
  volume={366},
  number={4-5},
  pages={397--402},
  year={2007},
  publisher={Elsevier}
}

@article{tian1998adaptive,
  title={Adaptive control of chaotic continuous-time systems with delay},
  author={Tian, Yu-Chu and Gao, Furong},
  journal={Physica D: Nonlinear Phenomena},
  volume={117},
  number={1--4},
  pages={1--12},
  year={1998},
  doi={10.1016/S0167-2789(96)00319-3}
}

@article{rezaie2011adaptive,
  title={An adaptive delayed feedback control method for stabilizing chaotic time-delayed systems},
  author={Rezaie, Behrooz and Jahed Motlagh, Mohammad-Reza},
  journal={Nonlinear Dynamics},
  volume={64},
  pages={167--176},
  year={2011},
  doi={10.1007/s11071-010-9855-7}
}

@article{vasegh2009sliding,
  title={Chaos control in delayed chaotic systems via sliding mode based delayed feedback},
  author={Vasegh, Nastaran and Khaki Sedigh, Ali},
  journal={Chaos, Solitons \& Fractals},
  volume={40},
  number={1},
  pages={159--165},
  year={2009},
  doi={10.1016/j.chaos.2007.07.053}
}

@article{blakely2004highspeed,
  title={High-speed chaos in an optical feedback system with flexible timescales},
  author={Blakely, Jonathan N. and Illing, Lucas and Gauthier, Daniel J.},
  journal={IEEE Journal of Quantum Electronics},
  volume={40},
  number={3},
  pages={299--305},
  year={2004},
  doi={10.1109/JQE.2003.823021}
}

@article{cui2009synchronization,
  title={Synchronization of chaotic recurrent neural networks with time-varying delays using nonlinear feedback control},
  author={Cui, Baotong and Lou, Xuyang},
  journal={Chaos, Solitons \& Fractals},
  volume={39},
  number={1},
  pages={288--294},
  year={2009},
  doi={10.1016/j.chaos.2007.01.100}
}

@article{argyris2005chaos,
  title={Chaos-based communications at high bit rates using commercial fibre-optic links},
  author={Argyris, Apostolos and Syvridis, Dimitris and Larger, Laurent and Annovazzi-Lodi, Valerio and Colet, Pere and Fischer, Ingo and Garc\'{i}a-Ojalvo, Jordi and Mirasso, Claudio R. and Pesquera, Luis and Shore, K. Alan},
  journal={Nature},
  volume={438},
  number={7066},
  pages={343--346},
  year={2005},
  doi={10.1038/nature04275}
}

@article{belair1995age,
  title={Age-structured and two-delay models for erythropoiesis},
  author={B{\'e}lair, Jacques and Mackey, Michael C and Mahaffy, Joseph M},
  journal={Mathematical biosciences},
  volume={128},
  number={1-2},
  pages={317--346},
  year={1995},
  publisher={Elsevier}
}

@book{niculescu2002delay,
  title={Delay effects on stability: a robust control approach},
  author={Niculescu, Silviu-Iulian},
  year={2002},
  publisher={Springer}
}

@article{dutta2025some,
  title={Some stability results for the fractional differential equations with two delays},
  author={Dutta, Pragati and Bhalekar, Sachin},
  journal={arXiv preprint arXiv:2509.21937},
  year={2025}
}

@article{kodba2005detecting,
  title={Detecting chaos from a time series},
  author={Kodba, Stane and Perc, Matja{\v{z}} and Marhl, Marko},
  journal={European journal of physics},
  volume={26},
  number={1},
  pages={205--215},
  year={2005}
}

@book{hale1993introduction,
  title={Introduction to Functional Differential Equations},
  author={Hale, Jack K. and Verduyn Lunel, Sjoerd M.},
  year={1993},
  publisher={Springer}
}

@book{diekmann2012delay,
  title={Delay equations: functional-, complex-, and nonlinear analysis},
  author={Diekmann, Odo and Van Gils, Stephan A and Lunel, Sjoerd MV and Walther, Hans-Otto},
  volume={110},
  year={2012},
  publisher={Springer Science \& Business Media}
}

@article{mackey1977oscillation,
  title={Oscillation and chaos in physiological control systems},
  author={Mackey, Michael C and Glass, Leon},
  journal={Science},
  volume={197},
  number={4300},
  pages={287--289},
  year={1977},
  publisher={American Association for the Advancement of Science}
}

@book{gopalsamy2013stability,
  title={Stability and oscillations in delay differential equations of population dynamics},
  author={Gopalsamy, Kondalsamy},
  volume={74},
  year={2013},
  publisher={Springer Science \& Business Media}
}

@book{strogatz2001nonlinear,
  title={Nonlinear dynamics and chaos: with applications to physics, biology, chemistry, and engineering (studies in nonlinearity)},
  author={Strogatz, Steven H},
  volume={1},
  year={2001},
  publisher={Westview press}
}

@article{farmer1982chaotic,
  title={Chaotic attractors of an infinite-dimensional dynamical system},
  author={Farmer, J Doyne},
  journal={Physica D: Nonlinear Phenomena},
  volume={4},
  number={3},
  pages={366--393},
  year={1982},
  publisher={Elsevier}
}

@article{pyragas1992continuous,
  title={Continuous control of chaos by self-controlling feedback},
  author={Pyragas, Kestutis},
  journal={Physics letters A},
  volume={170},
  number={6},
  pages={421--428},
  year={1992},
  publisher={Elsevier}
}

@article{babloyantz1995control,
  title={Control of chaos in delay differential equations, in a network of oscillators and in model cortex},
  author={Babloyantz, Agnessa and Lourenco, C and Sepulchre, JA},
  journal={Physica D: Nonlinear Phenomena},
  volume={86},
  number={1-2},
  pages={274--283},
  year={1995},
  publisher={Elsevier}
}

@article{mensour1998chaos,
  title={Chaos control in multistable delay-differential equations and their singular limit maps},
  author={Mensour, Boualem and Longtin, Andr{\'e}},
  journal={Physical Review E},
  volume={58},
  number={1},
  pages={410},
  year={1998},
  publisher={APS}
}

@book{michiels2007stability,
  title={Stability and stabilization of time-delay systems: an eigenvalue-based approach},
  author={Michiels, Wim and Niculescu, Silviu-Iulian},
  year={2007},
  publisher={SIAM}
}

@article{lu2006generating,
  title={Generating multiscroll chaotic attractors: theories, methods and applications},
  author={L{\"u}, Jinhu and Chen, Guanrong},
  journal={International Journal of Bifurcation and chaos},
  volume={16},
  number={04},
  pages={775--858},
  year={2006},
  publisher={World Scientific}
}

@article{wolf1985determining,
  title={Determining Lyapunov exponents from a time series},
  author={Wolf, Alan and Swift, Jack B and Swinney, Harry L and Vastano, John A},
  journal={Physica D: nonlinear phenomena},
  volume={16},
  number={3},
  pages={285--317},
  year={1985},
  publisher={Elsevier}
}

@inproceedings{bhalekar2011new,
  title={A new chaotic dynamical system and its synchronization},
  author={Bhalekar, Sachin and Daftardar-Gejji, Varsha},
  booktitle={Proceedings of the international conference on mathematical sciences in honor of Prof. AM Mathai},
  pages={3--5},
  year={2011}
}

@article{pecora1990synchronization,
  title={Synchronization in chaotic systems},
  author={Pecora, Louis M and Carroll, Thomas L},
  journal={Physical review letters},
  volume={64},
  number={8},
  pages={821},
  year={1990},
  publisher={APS}
}

@article{kocarev1995general,
  title={General approach for chaotic synchronization with applications to communication},
  author={Kocarev, Ljupco and Parlitz, Ulrich},
  journal={Physical review letters},
  volume={74},
  number={25},
  pages={5028},
  year={1995},
  publisher={APS}
}

@article{mensour1998synchronization,
  title={Synchronization of delay-differential equations with application to private communication},
  author={Mensour, Boualem and Longtin, Andr{\'e}},
  journal={Physics Letters A},
  volume={244},
  number={1-3},
  pages={59--70},
  year={1998},
  publisher={Elsevier}
}

@article{pyragas1998synchronization,
  title={Synchronization of coupled time-delay systems: Analytical estimations},
  author={Pyragas, Kestutis},
  journal={Physical Review E},
  volume={58},
  number={3},
  pages={3067},
  year={1998},
  publisher={APS}
}

@article{michiels2009synchronization,
  title={Synchronization of delay-coupled nonlinear oscillators: An approach based on the stability analysis of synchronized equilibria},
  author={Michiels, Wim and Nijmeijer, Henk},
  journal={Chaos: An interdisciplinary journal of nonlinear science},
  volume={19},
  number={3},
  year={2009},
  publisher={AIP Publishing}
}

@article{boccaletti2002synchronization,
  title={The synchronization of chaotic systems},
  author={Boccaletti, Stefano and Kurths, J{\"u}rgen and Osipov, Grigory and Valladares, DL and Zhou, CS},
  journal={Physics reports},
  volume={366},
  number={1-2},
  pages={1--101},
  year={2002},
  publisher={Elsevier}
}

@article{landsman2007complete,
  title={Complete chaotic synchronization in mutually coupled time-delay systems},
  author={Landsman, Alexandra S and Schwartz, Ira B},
  journal={Physical Review E—Statistical, Nonlinear, and Soft Matter Physics},
  volume={75},
  number={2},
  pages={026201},
  year={2007},
  publisher={APS}
}

@article{klein2004synchronization,
  title={Synchronization of neural networks by mutual learning and its application to cryptography},
  author={Klein, Einat and Mislovaty, Rachel and Kanter, Ido and Ruttor, Andreas and Kinzel, Wolfgang},
  journal={Advances in Neural Information Processing Systems},
  volume={17},
  year={2004}
}

@article{motter2013spontaneous,
  title={Spontaneous synchrony in power-grid networks},
  author={Motter, Adilson E and Myers, Seth A and Anghel, Marian and Nishikawa, Takashi},
  journal={Nature Physics},
  volume={9},
  number={3},
  pages={191--197},
  year={2013},
  publisher={Nature Publishing Group UK London}
}

@article{rulkov1995generalized,
  title={Generalized synchronization of chaos in directionally coupled chaotic systems},
  author={Rulkov, Nikolai F and Sushchik, Mikhail M and Tsimring, Lev S and Abarbanel, Henry DI},
  journal={Physical Review E},
  volume={51},
  number={2},
  pages={980},
  year={1995},
  publisher={APS}
}

@article{grossman1982discrete,
  title={Discrete delay, distributed delay and stability switches},
  author={Grossman, Zvi},
  journal={Journal of mathematical analysis and applications},
  year={1982}
}

@article{beretta2002geometric,
  title={Geometric stability switch criteria in delay differential systems with delay dependent parameters},
  author={Beretta, Edoardo and Kuang, Yang},
  journal={SIAM Journal on Mathematical Analysis},
  volume={33},
  number={5},
  pages={1144--1165},
  year={2002},
  publisher={SIAM}
}

@article{bhalekar2024stability,
  title={Stability and bifurcation analysis of two-term fractional differential equation with delay},
  author={Bhalekar, Sachin and Gupta, Deepa},
  journal={arXiv preprint arXiv:2404.01824},
  year={2024}
}

@article{shahriari2023ordinal,
  title={Ordinal Poincar{\'e} sections: Reconstructing the first return map from an ordinal segmentation of time series},
  author={Shahriari, Zahra and Algar, Shannon D and Walker, David M and Small, Michael},
  journal={Chaos: An Interdisciplinary Journal of Nonlinear Science},
  volume={33},
  number={5},
  year={2023},
  publisher={AIP Publishing}
}

@article{yan2024construction,
  title={Construction and implementation of wide range parameter switchable chaotic system},
  author={Yan, Minxiu and Liu, Xindi and Jie, Jingfeng and Hong, Yue},
  journal={Scientific Reports},
  volume={14},
  number={1},
  pages={4059},
  year={2024},
  publisher={Nature Publishing Group UK London}
}

@article{thamilmaran2024extreme,
  title={Extreme events in a damped Korteweg--de Vries (KdV) autonomous system: A comprehensive analysis},
  author={Thamilmaran, K and Bhagyaraj, T and Sabarathinam, S},
  journal={Chaos, Solitons \& Fractals},
  volume={186},
  pages={115337},
  year={2024},
  publisher={Elsevier}
}

@article{mensour1998power,
  title={Power spectra and dynamical invariants for delay-differential and difference equations},
  author={Mensour, Boualem and Longtin, Andr{\'e}},
  journal={Physica D: Nonlinear Phenomena},
  volume={113},
  number={1},
  pages={1--25},
  year={1998},
  publisher={Elsevier}
}

@article{burykin2014dynamical,
  title={Dynamical density delay maps: simple, new method for visualising the behaviour of complex systems},
  author={Burykin, Anton and Costa, Madalena D and Citi, Luca and Goldberger, Ary L},
  journal={BMC Medical Informatics and Decision Making},
  volume={14},
  number={1},
  pages={6},
  year={2014},
  publisher={Springer}
}

\end{document}